\documentclass[12pt,a4paper]{amsart}

\RequirePackage[OT1]{fontenc}
\RequirePackage{amsthm,amsmath}
\RequirePackage[numbers]{natbib}
\RequirePackage[colorlinks,citecolor=blue,urlcolor=blue]{hyperref}

\usepackage{enumitem}
\usepackage[titletoc]{appendix}
\usepackage{color}

\usepackage{graphicx}
\newtheorem{theorem}{Theorem}
\newtheorem{lemma}[theorem]{Lemma}
\newtheorem*{lemma*}{Lemma}
\newtheorem{proposition}[theorem]{Proposition}
\newtheorem{corollary}[theorem]{Corollary}
\newtheorem{conjecture}[theorem]{Conjecture}
\newtheorem{definition}[theorem]{Definition}
\newtheorem{remark}[theorem]{Remark}
\newcommand{\wass}{{\mathcal W}}
\usepackage[utf8]{inputenc}
\usepackage[english]{babel}
\usepackage{latexsym}
\usepackage{amssymb}
\usepackage{amsmath}

\newcommand{\pii}{{\textstyle\frac{1}{2\pi}}}
\newcommand{\rand}{{\rm rand}}
\DeclareMathOperator{\sgn}{sgn}

\newcommand{\N}{\mathbb{N}}

\newcommand{\Z}{\mathbb{Z}}
\newcommand{\R}{\mathbb{R}}
\newcommand{\C}{\mathbb{C}}
\newcommand{\E}{\mathbb{E\,}}
\newcommand{\Prob}{\mathbb{P}}

\renewcommand{\Re}{\mathrm{Re} \,}
\renewcommand{\Im}{\mathrm{Im} \,}

\newcommand{\z}{Z}
\newcommand{\reg}[1]{{#1}^{{\mathbf r}}}
\renewcommand{\blacksquare}{\diamond}

\makeatletter
\def\@tocline#1#2#3#4#5#6#7{\relax
  \ifnum #1>\c@tocdepth % then omit
  \else
    \par \addpenalty\@secpenalty\addvspace{#2}%
    \begingroup \hyphenpenalty\@M
    \@ifempty{#4}{%
      \@tempdima\csname r@tocindent\number#1\endcsname\relax
    }{%
      \@tempdima#4\relax
    }%
    \parindent\z@ \leftskip#3\relax \advance\leftskip\@tempdima\relax
    \rightskip\@pnumwidth plus4em \parfillskip-\@pnumwidth
    #5\leavevmode\hskip-\@tempdima
      \ifcase #1
       \or\or \hskip 1em \or \hskip 2em \else \hskip 3em \fi%
      #6\nobreak\relax
    \dotfill\hbox to\@pnumwidth{\@tocpagenum{#7}}\par% <---- \dotfill -> \hfill
    \nobreak
    \endgroup
  \fi}
\makeatother

\author[E. Saksman]{Eero Saksman}
\address{University of Helsinki, Department of Mathematics and Statistics,
         P.O. Box 68, FIN-00014 University of Helsinki, Finland}
\email{eero.saksman@helsinki.fi}

\author[C. Webb]{Christian Webb}
\address{Department of mathematics and systems analysis, Aalto University, P.O.
Box 11000, 00076 Aalto, Finland}
\email{christian.webb@aalto.fi}


\thanks{The first author was supported by the Academy of Finland CoE in Analysis and Dynamics Research. The second author was supported by the Eemil Aaltonen Foundation grant Stochastic dynamics on large random graphs and Academy of Finland grants 288318 and 308123. The authors  wish to thank the Heilbronn Institute for Mathematical Research for support during the workshop Extrema of logarithmically correlated processes, characteristic polynomials, and the Riemann zeta function, where some of this work was carried out.}

\newcommand{\add}{\; \; \underset{{\rm unif}}{\sim}\;\;}

\numberwithin{equation}{section}
\numberwithin{theorem}{section}

\title[The zeta function on the critical line and multiplicative chaos]{The Riemann zeta function\\ and Gaussian multiplicative chaos:\\ statistics on the critical line}

\begin{document}

\begin{abstract}
We prove that if $\omega$ is uniformly distributed on $[0,1]$, then as $T\to\infty$, $t\mapsto \zeta(i\omega T+it+1/2)$ converges to a non-trivial random generalized function, which in turn is identified  as a product of a very well behaved random smooth function and a random generalized function known as a complex Gaussian multiplicative chaos distribution. This demonstrates a novel rigorous connection between probabilistic number theory and the theory of multiplicative chaos -- the latter is  known to be connected to various branches of modern probability theory and mathematical physics.  

We also investigate the statistical behavior of the zeta function on the mesoscopic scale. We prove that if we let $\delta_T$ approach zero slowly enough as $T\to\infty$, then $t\mapsto \zeta(1/2+i\delta_T t+i\omega T)$ is asymptotically a product of a divergent scalar quantity suggested by Selberg's central limit theorem and a strictly Gaussian multiplicative chaos. We also prove a similar result for the characteristic polynomial of a Haar distributed random unitary matrix, where the scalar quantity is slightly different but the multiplicative chaos part is identical. This says that up to scalar multiples, the zeta function and the characteristic polynomial of a Haar distributed random unitary matrix have an identical distribution on the mesoscopic scale.
\end{abstract}

\maketitle

{\hypersetup{linkcolor=black}
\tableofcontents}

\section{Introduction}\label{se:intro} 

In this article, we study a rather classical problem in probabilistic number theory, namely the statistical behavior of the Riemann zeta function on the critical line (for definitions, see Section \ref{subs:pre}), or what can be said about the behavior of the zeta function in the neighborhood of a random point on the critical line and far away from the origin. Before going into precise definitions and statements, we review briefly and informally our main contributions, their role compared to previous studies, and connections to other branches of modern probability theory and mathematical physics. 

One of our main results, Theorem \ref{th:main}, states that if this neighborhood of a random point is of a fixed size compared to the distance to the origin, then the zeta function behaves like a random generalized function which can be viewed as a product of a random smooth function and a very special random generalized function known as a (complex) Gaussian multiplicative chaos distribution. The main contributions of this article continue with Theorem \ref{th:meso1} and Theorem \ref{th:mesormt}. Theorem \ref{th:meso1} states that if the neighborhood of the random point shrinks slowly enough compared to the distance to the origin, one again has a limit theorem representing the zeta function in terms of Gaussian multiplicative chaos. Theorem \ref{th:mesormt} states that up to a constant random factor, we have identical behavior in a model which is believed to be very closely related to the zeta function, namely the characteristic polynomial of a large, Haar distributed random unitary matrix. Informally speaking, these results establish that on such a \emph{mesoscopic} scale, the statistical behavior of the zeta function is indistinguishable from that of a characteristic polynomial of a large random unitary matrix. We now turn to the historical background and context of our problem and results.

The study of the statistical behavior of the zeta function began with the work of Bohr and Jessen in the 1930's \cite{BJ,BJ2}. They studied pointwise statistics away from the critical line. Informally speaking, they found that if one considers a random point uniformly bounded away from the critical line, then the probability law of the zeta function at this random point converges to a non-trivial probability law as the distance of this random point from the origin tends to infinity. This left open the questions of what happens at a random point on the critical line and what are the functional statistics, or what happens in a neighborhood of this random point. The question of functional statistics  was later studied by Bagchi (\cite{Bagchi}, see also \cite[Chapter 5]{L}, and \cite{L} in general for an extensive discussion about the statistical behavior of the zeta function), who found that in a fixed neighborhood of a random point far away from the origin and at a positive distance  from the critical line, the zeta function behaves like a random analytic function. 

 The behavior on the critical line is more complicated. A classical result of Selberg's, \cite{Selberg}, states that to have a meaningful limit theorem for a random point on the critical line, one must consider the logarithm of the absolute value of the zeta function and divide it by a suitable divergent quantity related to the distance of the random point from the origin -- the limiting probability law being the standard normal one. For functional statistics, Laurin\v{c}ikas proved in \cite{L2} that on the critical line, (without a divergent normalization as in Selberg's limit theorem) one can not have a functional limit theorem where the limit would be a continuous function. This means that there are various approaches one could take to obtain a functional limit theorem: one might try to normalize the zeta function as in Selberg's limit theorem, one might consider shrinking neighborhoods, or one might try to relax the regularity of the limit.  In the spirit of Selberg's limit theorem, a partial result in the direction of a functional limit theorem for the logarithm of the zeta function on a mesoscopic scale was proven by Bourgade \cite{Bourgade}, the limiting object being a suitable Gaussian process (see also \cite{HNY} for a related result). From the point of view of relaxing the regularity of the limit, our Theorem \ref{th:main} can be seen as giving the first rigorous answer to the question of what can be said about the functional statistics on the critical line, and similarly Theorem \ref{th:meso1} answers the same question in case we shrink the neighborhood at a mesoscopic rate. The open, and very hard, problem remaining is to study the statistical behavior of the zeta function when one shrinks the neighborhood at a  "\emph{microscopic} rate". Here (up to a divergent constant random factor), one expects to have convergence to a certain random analytic function related to the sine process arising in random matrix theory -- see \cite{CNN}.

Since the 1970's, and increasingly during the last two decades,  a slightly different direction in the study of the statistical behavior of the zeta function has been provided by a conjectural connection between random matrix theory and the zeta function -- for a brief overview of this connection, see the book review of Conrey \cite{Conrey}. An influential conjecture in the development of random matrix theory, as well as certain branches of analytic number theory, has been the Montgomery-Dyson pair correlation conjecture, which states that after a suitable normalization, the non-trivial zeroes of the zeta function behave statistically speaking like the eigenvalues of a large random matrix. This conjectural connection was a motivation in the work of Keating and Snaith \cite{KS}, who suggested that the statistical behavior of the zeta function should be modelled by the characteristic polynomial of a random unitary matrix. Morevoer, they conjectured a connection between moments of the zeta function in the neighborhood of a random point on the critical line as well as moments of the characteristic polynomial of a random unitary matrix. Again the most interesting and most difficult problem would be establishing such a connection on the microscopic scale, where one can see individual zeroes of the zeta function. Nevertheless, our Theorem \ref{th:meso1} and Theorem \ref{th:mesormt} can be seen as a rigorous justification of this Montgomery-Dyson-Keating-Snaith picture, when considering phenomena related to the mesoscopic scale.

As a final remark concerning the background and general context of our work, we point out rather recent work of Fyodorov, Hiary, and Keating \cite{FHK} and Fyodorov and Keating \cite{FK}. Perhaps the main message of these articles is that on the global and mesoscopic scale, one should think of the logarithm of the zeta function on the critical line behaving statistically like log-correlated field -- a random generalized function which has a logarithmic singularity on the diagonal of its covariance kernel.\footnote{In particular, they argue that known properties of maxima of log-correlated fields, properties of characteristic polynomials of random matrices, and conjectures due to Montgomery and Farmer-Gonek-Hughes \cite{FGH} fit well together.} Log-correlated fields have emerged recently as an important class of objects arising naturally in many models of modern probability and mathematical physics (see e.g. \cite{ABH,AJKS,CLD,FK,DKRV,DS,HKO}). In addition to being connected to many important models, their interest also stems from the fact that despite being rough objects, one can make some sense of their fractal geometric properties, which typically end up being universal -- common to all log-correlated fields. It turns out that a good way to describe and study these universal geometric properties is through objects which are formally exponentials of these log-correlated fields. The rigorous construction of these objects and the study of their properties goes under the name of the theory of multiplicative chaos, and its foundations were laid in the work of Kahane \cite{Kahane} (see also \cite{RV} for a recent review and \cite{Be} for a recent study). From this point of view, many of the conjectures of \cite{FHK,FK} can be rephrased as stating that statistical properties of the zeta function can be expressed through the theory of multiplicative chaos. Some properties one would expect the zeta function to possess based on universality of properties of log-correlated fields have indeed been proven recently (without multiplicative chaos theory) -- see e.g. \cite{ABBRS,N} (see also \cite{ABB,Z,M} for related results in the setting of random matrices). From this point of view, our main results connecting the Riemann zeta function to multiplicative chaos can be seen as establishing a novel connection between the zeta function and a class of stochastic objects arising in various models of great interest as well as further evidence for universal behavior in the class of log-correlated fields.

The methods we use to prove our main theorems involve mainly tools of probability theory and analysis -- indeed the tools of analytic number theory we make use of are rather classical. From the point of view of number theory, our main results lie in the probabilistic part of analytic number theory. Our main contributions to probabilistic number theory should be viewed as introducing a new class of limit theorems and limiting objects that are shown to appear naturally in the theory. In the setting of random matrices, we provide a rigorous proof of a connection between the statistical behavior of the zeta function and random characteristic polynomials, while in the setting of the theory of log-correlated fields, our results can be seen as further evidence for log-correlated fields being an ubiquitous class of objects with many universal properties. In addition, from the point of view of analytic number theory, our contributions can be seen as providing rigorous support for a conceptual picture, describing what kind of behavior one should expect from the zeta function when studying questions on the global or mesoscopic scale. In the future, this connection to multiplicative chaos could also prove fruitful in studying fractal properties of the zeta function.

The remainder of this introduction is devoted to the precise definitions of the objects we study, the precise statements of our main results (together with some further results), along with some conjectures. In this whole article, we have tried to be generous with details in our presentation to make the text accessible for readers with various backgrounds.

\subsection{Definitions, basic concepts, and main result on the global scale}\label{subs:pre} 
In this section, we describe the main objects we study and state our main results concerning the study of the zeta function in a fixed size neighborhood of a random point on the critical line. As the study of multiplicative chaos is a less classical field of mathematics, we also discuss some of its background for the convenience of the reader.

We start by fixing our notation and recalling some basic facts about the \emph{Riemann zeta function} defined for $\sigma>1$ by the convergent series\footnote{We follow the custom in analytic number theory to denote  by $s=\sigma+it$  the real an imaginary parts of the complex variable $s$.}
$$
\zeta (s)=\sum_{n=1}^\infty n^{-s}.
$$
One may continue $\zeta $ as a meromorphic function to the whole complex plane, with only one pole. The pole is simple and located at $s=1$ with residue 1. The \emph{critical line} $\lbrace \sigma=\frac{1}{2}\rbrace$ is of particular importance in the study of the zeta function as the behavior of the zeta function on it is closely related to the distribution of the prime numbers, and is of course a subject of many fundamental conjectures in number theory. As stated at the beginning of our introduction, we are interested in the behavior of the zeta function in the neighborhood of a random point on the critical line. The precise object of our study is the following:

\begin{equation}\label{eq:randomshift}
\mu_T(x):=\zeta(1/2+ix+i\omega T)\quad\textrm{for}\quad x\in \R,
\end{equation}
where $\omega$ is a random variable which is uniform on $[0,1]$, $T\in(0,\infty)$, and we are interested in the behavior of this random function in the limit where $T\to\infty$. It is natural to expect that the precise distribution of the random variable $\omega$ is unimportant, but we will not discuss this further.

As discussed at the beginning of this introduction, the other main player in our study is a random generalized function which is called a Gaussian multiplicative chaos distribution. The foundations of the mathematical theory of  Gaussian multiplicative chaos were established in the 1980's by Kahane \cite{Kahane}.  At that time, the main motivation was the desire to build mathematical models for Kolmogorov's statistical theory of turbulence by providing a continuous counterpart for multiplicative cascades that were originally introduced  by Mandelbrot for the same purpose in the early 1970's. During the last 15 years there has been a new wave of interest in multiplicative chaos, partly due to its important connections to Stochastic Loewner Evolution \cite{Sheffield,AJKS}, to quantum gravity and scaling limits of random planar maps \cite{DS,MS2013,berestycki2014liouville,MS2015,MS2016,MS20162,DKRV}, as well as to models in finance and turbulence \cite[Section 5]{RV}.

In order to give a brief and informal description of multiplicative chaos, consider a sequence of a.s.\ continuous and centered real-valued Gaussian fields $X_n$, say on an interval $I\subset \R$. The elements of this sequence should be considered as suitable approximations of a (possibly distribution valued) Gaussian field $X$. For simplicity, assume that the increments $X_{n+1}-X_n$ are independent. One may then define the random measures $\lambda_n$  on $I$ by setting
$$
\lambda_n(dx):=\exp(X_n(x)-\frac{1}{2}\E X_n(x)^2)dx.
$$
In this situation basic martingale theory implies that almost surely there exists a (random) limit measure $\lambda=\lim_{n\to\infty}\lambda_n$, where the convergence is understood in the ${\rm weak^*}$-sense. The measure $\lambda$ is called the \emph{multiplicative chaos measure} defined by $X$, and Kahane proved that under suitable conditions the limit does not depend on the choice of the approximating sequence ($X_n$). However, a significant obstacle in defining a meaningful limiting object $\lambda$ is that it may very well be the zero measure almost surely.
The most important, and in some sense a borderline situation for defining meaningful limiting objects, is when the limit field $X$ is log-correlated, i.e.\ it has a covariance of the form
$$
C_X(x,y)=-2\beta^2\log |x-y|+ G(x,y),\qquad x,y\in I,
$$
where $G$ is a continuous and bounded function. Then Kahane's theory implies that  the limit measure is almost surely non-zero for $0<\beta <1$.  The limiting random Borel measure $\lambda=\lambda_\beta$ on the interval $I$  is almost surely singular and its basic properties like multifractal spectrum, tail of the total mass or scaling properties have been investigated. 

At  the threshold $\beta=\beta_{c}:=1$ one needs to add a  deterministic non-trivial renormalization factor that depends on $n$ in order to obtain the existence of a non-trivial object known as a \emph{critical chaos measure}. This limit can also be achieved through a random normalization known as the derivative martingale.  

Overall,  the dependence of the chaos measure on the generating Gaussian field has many delicate features.  E.g. the universality property (how the law of the limiting object is independent of the precise details of the approximation scheme) is far from trivial for multiplicative chaos  \cite{RV,Shamov,JS}. We refer to the nice survey \cite{RV} for the basic properties of these measures, to \cite{Be} for an elegant proof of the existence of subcritical chaos measures, and to \cite{DRSV1,DRSV2,BKNSW} for the existence and basic properties of critical Gaussian chaos. 

There is a further variant of multiplicative chaos that is important for the connection to the Riemann zeta function (and for random matrix theory as well), which is the concept of \emph{complex multiplicative chaos}, where in the above one allows for complex Gaussian fields. 
Two basic cases have been studied in the literature. In the first variant one allows the parameter $\beta$ take complex values, and it turns out that one obtains analyticity in the parameter $\beta$ for  $\beta\in U$, where $U\subset\C$ is an open subset  
whose intersection with the real axis equals $(-1,1)$ in the above normalization (see \cite{BJM1,BJM2} in the slightly simpler case of multiplicative cascades and e.g. \cite{AJKS} in the case of multiplicative chaos).  In the second case one assumes that $X=\beta_1X_1+i\beta_2X_2$ with 
$X_1,X_2$ independent copies of a log-correlated field and $\beta_1,\beta_2\in \R$. This case turns out to be more amenable to analysis, due to the independence of the real and imaginary parts, and many aspects of it have been studied thoroughly in \cite{RV3}. However, the complex chaos we need to study here does not quite fit into either of these models. In our situation, we study a complex Gaussian field $\mathcal{G}$, for which there is a very special mutual dependence between the real and imaginary parts $X_1$ and $X_2$, of the form
$$
\E X_1(x)X_2(y)=-\frac{\pi}{4}\sgn(x-y)\; +\; {\rm smooth},
$$
where $\sgn(x)$ denotes the sign of $x$ and the covariance is zero when $x=y$. Formally this suggests that, the 2-point function $\E e^{\mathcal{G}(x)+\overline{\mathcal{G}(y)}}$ is not absolutely integrable, which in general indicates that one can not use $L^2$-theory to construct a limiting object. Remarkably enough,  it is exactly the above peculiar dependence of the real and imaginary part that produces the dominant part $(i(x-y))^{-1}$ to the two-point function, and hence the basic theory of one-dimensional singular integrals applies to resurrect the $L^2$-theory. It turns out that the complex Gaussian chaos we study, which can be formally written as $\nu=e^{\mathcal{G}}$, has some unique features that arise from the fact that it can be considered as a boundary distribution of a random analytic function. For example, the finiteness of a moment $\E |\nu (\phi)|^p$ with $p>4$ can be shown to depend on the smoothness properties of the function $\phi$, and thus their properties differ in some respects from the complex chaos considered in \cite{RV3}. 

Before stating our result concerning the behavior of $\mu_T$ from \eqref{eq:randomshift}, we need to recall some notions related to generalized functions to properly state our convergence results. We shall make extensive use of the classical $L^2$-based \emph{Sobolev spaces} $W^{\alpha, 2}(\R)$, where $\alpha\in\R$ is an arbitrary smoothness index. The space $W^{\alpha, 2}(\R)$ consists of all Schwartz generalized functions $f\in \mathcal{S}'(\R)$ such that the Fourier  transform $\widehat f$ is locally $L^2$-integrable and satisfies
$$
\| f\|_{W^{\alpha,2}(\R)}^2\; :=\; \int_\R(1+\xi^2)^\alpha |\widehat f(\xi)|^2d\xi <\infty ,
$$
where our  convention for the normalization of the Fourier transform is given by
$$
\widehat f(\xi):=\int_\R e^{-2\pi i \xi x}f(x)dx
$$
 for $f\in \mathcal{S}(\R).$  One has the duality $(W^{\alpha,2}(\R))'=W^{-\alpha,2}(\R)$ with respect to the standard distributional pairing. If $\alpha >1/2,$ there exists a continuous embedding $W^{\alpha,2}(\R)\subset C(\R)$, where the latter space is equipped with the sup-norm. In turn, if $\alpha <-1/2,$ one has $\delta_a\in W^{\alpha,2}(\R)$ for all $a\in\R.$ We refer e.g. to \cite{Stein} or \cite{G}   for basic facts on harmonic analysis and the relevant function spaces -- we  apply mostly only the $L^2$-theory that can be often conveniently dealt with by basic Fourier analysis.

We now turn to our first main result, which shows unconditionally\footnote{That is, we do not assume the Riemann hypothesis.}, that the random shifts $\mu_T$ from \eqref{eq:randomshift} converge to a statistical limit that is a random generalized function  whose non-trivial behaviour is determined by a (complex) Gaussian multiplicative chaos distribution:

\begin{theorem}\label{th:main}\quad {\bf (i)}
There exists a non-trivial random variable $x\mapsto \zeta_\rand(1/2+ix)$ taking values in $\mathcal{S}'(\R)$ -- the space of tempered distributions -- such that as $T\to\infty$
\begin{equation*}
(1+x^2)^{-1}\mu_T(x)\stackrel{d}{\longrightarrow}(1+x^2)^{-1}\zeta_\rand(1/2+ix),
\end{equation*}
where the convergence in law is with respect to the strong topology of the Sobolev space $W^{-\alpha,2}(\R)$ for any $\alpha>1/2$.

\smallskip
\noindent  {\bf (ii)}\quad
Moreover, the law of the limit $\zeta_\rand$ can be characterized in the following way$:$ as random generalized functions
\begin{equation*}%\label{eq:identification}
\zeta_{\rand}(1/2+ix)=g(x)\nu(x),
\end{equation*}
where $\nu$ is a random generalized function known as a Gaussian multiplicative chaos distribution, 
which can be formally written as
$$
\nu(x)= "\; e^{\mathcal{G}(x)}\;",
$$
where $\mathcal{G}$ is a centred Gaussian field with the  correlation structure
$$
\E \mathcal{G}(x)\mathcal{G}(y)=0 \quad \textrm{and} \quad \E \mathcal{G}(x)\overline{\mathcal{G}(y)}=\log (\zeta (1+i(x-y)))\quad\textrm{ for}\quad x,y\in \R.
$$
The  factor $g$ is a random smooth function on $\R$, almost surely has no zeroes, and for which $\E\big(\|g(x)\|_{C^\ell(I)}^p+ \|1/g(x)\|_{C^\ell(I)}^p\big)$ is finite for all $p\in \R$, any $\ell\geq 0$, and any finite interval $I\subset \R$.
\end{theorem}

Above, the identification of $\zeta_{\rand}$ and $g \nu$ can be understood to mean that as elements of $W^{-\alpha,2}(\R)$, $(1+x^2)^{-1}\zeta_{\rand}(1/2+ix)=(1+x^2)^{-1}g(x)\nu(x)$  or then equality can be understood in the sense of tempered distributions. Note that as we have convergence in the Sobolev space, we can define the action of $\zeta_{\rand}(1/2+ix)$ on a larger class of test functions than just Schwartz functions. For a definition of the norm we use on $C^\ell$, see Lemma \ref{le:perus}. We now turn to the mesoscopic behavior of the zeta function.

\subsection{The mesoscopic scale}\label{sec:mesointro}

Recall that by mesoscopic, we mean that we study the zeta function in a suitably slowly shrinking neighborhood of the random point $1/2+i\omega T$, or in the notation of \eqref{eq:randomshift}, we are interested in the random function $\mu_T(\delta_T x)$ for some $\delta_T\to 0$ with $T\to \infty$ slowly enough. Before stating the description  of the mesoscopic behavior, which comprises our second main result, we first make a couple of heuristic comments in an attempt to motivate and clarify the emerging picture. As one may expect from Theorem \ref{th:main} that $\mu_T(\delta_T x)$ is close to $e^{\mathcal{G}(\delta_T x)}$, consider what happens to a log-correlated field under scaling -- for simplicity, let us look at a real translation invariant log-correlated field $X$ with covariance $C_X(x,y)=\log(1/ |x-y|)+g(x-y)$, where $g$ is smooth. If we scale the spatial variable by $\varepsilon$ and let $\varepsilon\to 0$, the covariance will roughly be $\log \varepsilon^{-1}+g(0)-\log |x-y|$. So it's natural to expect that the precise details of the covariance are irrelevant: the field $X(\varepsilon x)$ should roughly consist of "zero mode" or constant Gaussian random variable whose variance explodes as $\varepsilon\to 0^+$, as well as of a log-correlated field with covariance $-\log |x-y|+\mathrm{constant}$ (though this latter field may not strictly speaking exist on the whole axis -- see a more precise formulation below). 

For precise statements for our results in this setting,  we need to recall the definition of  the Sobolev space $W^{\alpha,2}(0,1)$.
For any $\alpha\in\R$ this space can be defined as the space of restrictions of elements in $W^{\alpha,2}(\R)$ to the set $(0,1)$. More formally,
$$
W^{\alpha,2}(0,1):= \{ g_{|(0,1)} \; :\; g\in W^{\alpha,2}(\R)\}
$$
and
$$
\|f\|_{W^{\alpha,2}(0,1)}:=\inf \{ \|g\|_{W^{\alpha,2}(\R)}\;: \; f=g_{|(0,1)}\}.
$$
The space $ W^{\alpha,2}(0,1)$ is again a separable Hilbert space.
Moreover, it is easily verified that $\|(1+x^2)^{-1}f\|_{W^{\alpha,2}(0,1)}\asymp\|f\|_{W^{\alpha,2}(0,1)}$ (where $a\asymp b$ means that $a/b$ and $b/a$ are bounded) for any $\alpha\in\R$, which lets us transfer  our earlier convergence results on the whole line to restrictions to the interval. We also find it convenient to strengthen our notion of convergence from convergence in law to convergence in the sense of the quadratic Wasserstein (or Kantorovich) metric $\wass_2$. Assume that $(X,d_X)$ is a complete separable metric space and  $\mu$ and $\nu$ are Borel probability distributions on $X$. Then
$$
\wass_2(\mu,\nu)_X\; :=\; \left(\inf_{(U,V)} \E \Big( d_X(U,V)^2\Big)\right)^{1/2},
$$
where $U,V$ are random variables on a common probability space taking values in $X$ so that $U\sim \mu$ and $V\sim \nu$. This defines a distance between probability distributions on $X$. In particular, convergence in the Wasserstein metric implies  convergence in distribution and uniform boundedness of second moments. The monograph \cite{Villani} is a good place to learn the basics about Wasserstein distances, although we hardly need more than the definition in this paper.

We are now ready to state our first mesoscopic result

\begin{theorem}\label{th:meso1}
There exists a deterministic $\delta_T>0$ so that $\delta_T\to 0^+$ as $T\to \infty$ and 

\begin{equation*}
\lim_{T\to \infty}\wass_2\big(\zeta(1/2+i\delta_T x+i\omega T), h_T(x)e^{Y_T}\eta(x)\big)_{W^{-\alpha,2}(0,1)}=0,
\end{equation*}

\noindent where $\alpha>1/2$, $h_T$ is a random smooth function satisfying $h_T\stackrel{d}{\to} 1$ in $C^1[0,1]$ as $T\to\infty$, 

\begin{equation*}
Y_T\stackrel{d}{=}\sqrt{\log (1/\delta_T)}Z+R,
\end{equation*}

\noindent where $Z$ is a standard complex Gaussian, $R$ is a complex random variable independent of $x$ and it satisfies $\E e^{\lambda |R|}<\infty$ for all $\lambda>0$. Finally $\eta(x)$ is a complex Gaussian multiplicative chaos distribution which can be formally written as 

\begin{equation*}
\eta(x)=\exp\left[\int_0^1 \frac{e^{-2\pi i xu}-1}{\sqrt{u}}dB_u^{\C}+\int_1^\infty \frac{e^{-2\pi i x u}}{\sqrt{u}}dB_u^{\C}\right],
\end{equation*}

\noindent where $B_u^{\C}$ is a standard complex Brownian motion.
\end{theorem}

\noindent Note that we don't claim any independence between $\eta$, $R$ or $Z$. 

Our next result (which we prove in Appendix \ref{app:mesormt}) is the claim that a nearly identical result holds for the characteristic polynomial of a random unitary matrix, or more precisely, for random unitary matrices whose probability law is the Haar measure on the unitary group. 

\begin{theorem}\label{th:mesormt}
Let $\eta$ be the multiplicative chaos distribution from Theorem \ref{th:meso1} and let $U_N$ be a Haar distributed $N\times N$ random unitary matrix. There exists a deterministic $\delta_N>0$ so that $\delta_N\to 0^+$ as $N\to \infty$ and 

\begin{equation*}
\lim_{N\to \infty}\wass_2\big(\det(I-e^{i\delta_N x}U_N), e^{Y_N}(\eta(x)+o(1))\big)_{W^{-\alpha,2}(0,1)}=0,
\end{equation*}

\noindent where $\alpha>1/2$, and $o(1)$ is a $W^{-\alpha,2}(0,1)$-valued random variable that tends to zero in probability, and

\begin{equation*}
Y_N\stackrel{d}{=}\sqrt{\log (1/\delta_N)}Z+R,
\end{equation*}

\noindent where $Z$ is a standard complex Gaussian, $R$ is a complex Gaussian random variable independent of $x$ with order one variance. 
\end{theorem}

\begin{remark}\label{rem:zeta-meso=RMT-meso}  {\rm According to  Theorem \ref{th:meso1} we also have the convergence of the form
$$
\wass_2\big(\zeta(1/2+i\delta_T x+i\omega T), e^{Y_T}(\eta(x)+o(1))\big)_{W^{-\alpha,2}(0,1)}=0,
$$
analogous to that in Theorem \ref{th:mesormt}. In other words,  in the mesoscopic limit the functional statistics of random matrices and the Riemann zeta function coincide.}
\end{remark}

Finally,  to obtain a proper limit theorem we formulate a result without the divergent term. For this, we define some further notation and concepts in order to consider almost surely non-zero generalised functions on $(0,1)$ modulo multiplicative constants.
More precisely, consider the space
$$
W^{-\alpha,2}_{{\rm mult}}(0,1):= (W^{-\alpha,2}(0,1)\setminus \{0\})/\sim,
$$
where two non-zero elements $f,g\in W^{-\alpha,2}(0,1)$ are identified by $\sim$ if $f=cg$ for some $c\in\C\setminus\{0\}.$
To have the structure of a complete metric space, we equip $W^{-\alpha,2}_{{\rm mult}}(0,1)$ with the natural metric
$$
d_{W^{-\alpha,2}_{{\rm mult}}(0,1)}(f,g):= \inf_{\theta\in [0,2\pi)}\Big\| e^{i\theta}{f}{\|f\|^{-1}_{W^{-\alpha,2}(0,1)}}-{g}{\|g\|_{W^{-\alpha,2}(0,1)}^{-1}}\Big\|_{W^{-\alpha,2}(0,1)}.
$$
In more geometric terms, we identify non-zero functions  by their radial projections to the unit sphere of the Sobolev space. Functions on the unit sphere that differ by a unimodular constant factor are also identified.

Our second formulation of the behavior on the mesoscopic scale is the following one.

\begin{theorem}\label{th:meso2}
Let $\delta_T$ and $\eta$ be as in Theorem \ref{th:meso1}. Then as $T\to \infty$

\begin{equation*}
\zeta(1/2+i\delta_T x+i\omega T)\stackrel{d}{\to}\eta(x)
\end{equation*}

\noindent in the topology of $W_{\mathrm{mult}}^{-\alpha,2}(0,1)$ for any $\alpha>1/2$.
\end{theorem}
\noindent Above, $\eta$ can be thought of as a representative (or in a sense, a regularization)  of a Gaussian multiplicative chaos
$$
\eta_{{\rm mult}}:="\exp\left(\mathcal{G}_{{\rm meso}}\right)",
$$
that is defined only modulo multiplcative constants. Here the field $\mathcal{G}_{{\rm meso}}$ is defined up to additive constants, formally 
$$
\mathcal{G}_{{\rm meso}}(x)= \int_0^\infty \frac{e^{-2\pi i x u}}{\sqrt{u}}dB_u^{\C},
$$
that can be evaluated only against test functions $f$ with vanishing integral: $\int_{\R}f(x)dx=0$. It has rather interesting (formal) covariance structure 
\begin{align*}\label{eq:mesolimit}
\E \big(\Re \mathcal{G}_{{\rm meso}}(x)\Re \mathcal{G}_{{\rm meso}}(y)\big)&=\frac{1}{2}\log \left(\frac{1}{|x-y|}\right) = \E \big(\Im \mathcal{G}_{{\rm meso}}(x)\Im \mathcal{G}_{{\rm meso}}(y)\big) \quad\textrm{and}\\
\E \big(\Re \mathcal{G}_{{\rm meso}}(x)\Im \mathcal{G}_{{\rm meso}}(y)\big)&=-\frac{\pi}{4}\sgn(x-y).
\end{align*}

\noindent Naturally such a result can be shown also on the RMT side, but we leave this to the reader. 

We now turn to an outline of the proof of our main results Theorem \ref{th:main}, Theorem \ref{th:meso1}, and Theorem \ref{th:mesormt}.

\subsection{Outline of the proof and further results}\label{le:outline}

We now describe roughly the  proofs of our main results as well as state some further ones. We start with Theorem \ref{th:main}(i) and sketch the overall content of Section \ref{se:convergence}.  For that end, recall that $\mu_T$,  as defined in \eqref{eq:randomshift} stands for the random shift of the zeta function.  In order to identify the limit we employ the truncated Euler products $\zeta_N(s)=\prod_{k=1}^N(1-p_k^{-s})^{-1}$ where $p_k$:s are the primes in an increasing order, and denote the corresponding random shifts  by $\mu_{T,N}(x)=\zeta_N(1/2+ix+iT\omega)$. The proof uses in a crucial manner the explicit $T\to \infty$ limit of the two-point functions
$$
\E \mu_T(x)\overline{\mu_T(y)}, \quad \E \mu_{T,N}(x)\overline{\mu_T(y)},\quad\textrm{and}\quad \E \mu_{T,N}(x)\overline{\mu_{T,N}(y)}.
$$
The precise statement we need is given in Theorem \ref{th:twopoint} below. We note here that the proof of this theorem is essentially due to  Ingham \cite{Ingham} (see also the more recent paper of Bettin \cite{Bettin}). However, for the reader's convenience, we include a proof of Theorem \ref{th:twopoint} in Appendix \ref{app:appendixa} -- see also the discussion before Theorem \ref{th:twopoint} and in Remark \ref{re:cross alternative} for further reasons to present a proof of the result in this article.

Interestingly enough,  the main term in $\E \mu_T(x)\overline{\mu_T(y)}$ is given by 
$
({i(x-y)})^{-1}
$
 i.e. the kernel of the Hilbert transform. Using this observation as a starting point, a  careful analysis enables us to deduce suitable  uniform estimates, which in turn  show that the second moment $\E |\mu_T(f)|^2$ converges  as soon as $f$ is square-integrable with nice enough decay at infinity.  From this one may already fairly easily infer that $((1+x^2)^{-1}\mu_T(x))_{T\geq 1}$ remains tight in $W^{-\alpha,2}(\R)$ if $\alpha >1/2.$

In order to get hold of the convergence as  $T\to\infty$, we prove that the quantities $\mu_{T,N}(f)$ approximate well
$\mu_{T}(f)$  in terms of variance if $N =N(f)$ is large enough, and one may then lift this to a good approximation on the level of negative index Sobolev functions.  The final piece of information one needs is to note  that as $T\to\infty, $ the random variable $\mu_{T,N}$ converges 
in law to the \emph{randomized truncated Euler product}  $\zeta_{N,\rand}(1/2+ix)$, where
$
\zeta_{N,\rand}(s):= \prod_{k=1}^N\left(\frac{1}{1-p_k^{-s}e^{2\pi i\theta_k}}\right),
$
and the  $\theta_k$:s are i.i.d.  random variables, each uniformly distributed on $[0,1]$. 
Finally, we will show that as $N\to\infty$,  $\zeta_{N,\rand}(1/2+ix)$ converges almost surely (in the sense of generalized functions) to the  \emph{randomized Riemann zeta function} $\zeta_{\rand}(1/2+ix)$, where
$$
\zeta_{\rand}(s):= \prod_{k=1}^\infty\left(\frac{1}{1-p_k^{-s}e^{2\pi i\theta_k}}\right).
$$
In addition to being a limit of $\zeta_{N,\rand}(1/2+ix)$, $\zeta_\rand(1/2+ix)$ can be understood as the boundary values (in the sense of generalized functions) of the random analytic function $\zeta_{\rand}(s)$   in the half-plane $\{\sigma >1/2\}$.

As an aside, we mention that it is important  to observe that the limiting statistical object $x\mapsto \zeta_\rand(1/2+ix)$ is truly a generalized function, and not a function or a complex measure. Note that this implies the same claim for the complex Gaussian multiplicative chaos $e^{\mathcal{G}}$. To our knowledge, this is the first proof of such a result in the setting of complex multiplicative chaos. 
\begin{theorem}\label{th:nonmeasure}
Almost surely $\zeta_{\rand}$ does not coincide with a $($random complex$)$ Borel measure on any open subinterval of the critical line 
$\{\sigma=1/2\}$.
\end{theorem}
\noindent  It is also easy to verify that $\zeta_\rand$ is not a Gaussian random generalised function. As discussed in \cite[p. 250]{L} and \cite{L2} it is known  that the (localised) random shifts of the zeta function on  the critical line do not converge in distribution in the space of continuous functions. Thus Theorem \ref{th:nonmeasure} gives an adequate explanation for this phenomenon: one is forced to seek for the  limiting objects  in a suitable space of generalized functions.

The proof of the second part of the Theorem \ref{th:main} (which we give in Section \ref{sec:compconv}) is based on the following result of independent interest, as it provides a direct functional Gaussian approximation in contrast to e.g. \cite{ABH,Z}.

 \begin{theorem}\label{th:gaussian_appro}
For each $N\geq 1$, and any $A>0$ there exists a decomposition
\begin{equation*}
\log \zeta_{N,\rand}(1/2+ix)=\mathcal{G}_N(x)+\mathcal{E}_N(x),
\end{equation*}
where $\mathcal{G}_N$ is a Gaussian process on $[-A,A]$ which can be written in the following way: let $(W_k^{(j)})_{k\in \Z_+, j\in\lbrace 0,1\rbrace}$ be i.i.d. standard Gaussians, then
\begin{equation*}
\mathcal{G}_N(x)=\sum_{k=1}^N \frac{1}{\sqrt{2 p_k}}p_k^{-ix}(W_k^{(1)}+iW_k^{(2)}).
\end{equation*} 
The function $\mathcal{E}_N$ is smooth and as $N\to\infty$, it a.s. converges uniformly to a random smooth function $\mathcal{E}\in C^\infty[-A,A]$. Moreover, the maximal error and its derivatives in this decomposition have finite exponential moments:
\begin{equation*}
\E\exp\left(\lambda\sup_{N\geq 1}\|\mathcal{E}_N(x)\|_{C^\ell[-A,A]}\right)<\infty \quad \mathrm{for \ all\ }\lambda>0\quad \textrm{and}\quad \ell\geq 0.
\end{equation*}
\end{theorem}

\noindent The proof of this is the content of Section \ref{sec:appro}, which relies on slightly technical Gaussian approximation arguments combined with classical estimates for the distribution of prime numbers. This result then allows us to use martingale techniques to establish Theorem \ref{th:main} (ii).

After this, we move onto Section \ref{sec:meso}, where we prove our mesoscopic results for zeta, namely Theorem \ref{th:meso1} and Theorem \ref{th:meso2}. The proof of these follow essentially from a Gaussian approximation result of a similar flavor as Theorem \ref{th:gaussian_appro}, but stronger. The proof of this combines slightly lengthy arguments from basic number theory, harmonic analysis, and probability theory.

As we discuss later, it is natural to expect that the zeta function would give rise also to real multiplicative chaos measures in the sense that $|\mu_T(x)|^\beta$ normalized by its mean would converge to non-trivial random measure on $\R$, and these measures could be used to analyze some of the fractal structure of $\mu_T$. Moreover, in view of Theorem \ref{th:main}, it is natural to expect that these measures could be constructed from $\zeta_{N,\mathrm{rand}}$ as well. As the existence of the measures constructed from $\zeta_{N,\mathrm{rand}}$ can be deduced from Theorem \ref{th:gaussian_appro} and slight modifications of it, we prove results concerning them in this article as well. The precise results are the following:

\begin{theorem}\label{th:gmcconv}
As $N\to\infty$, the random measure 

\begin{equation*}
\frac{|\zeta_{N,\rand}(1/2+ix)|^{\beta}}{\E |\zeta_{N,\rand}(1/2+ix)|^\beta}dx, \quad x\in [0,1],
\end{equation*}
converges almost surely  with respect to the weak topology of measures to a random measure $\eta_\beta(dx)$ on $[0,1]$. For $\beta\geq \beta_c=2$, this limit is almost surely the zero measure. For $\beta<\beta_c$, it can be written as 
\begin{equation*}
\eta_\beta(dx)=f_\beta(x)\lambda_\beta(dx),
\end{equation*}
where $\lambda_\beta(dx)$ is  a Gaussian multiplicative chaos measure and $f_\beta$ is a random continuous positive function such that for any $\ell \geq 1$ the norms  $\| f_{\beta}\|_{C^\ell[0,1]}$ and $\| 1/f_{\beta}\|_{C^\ell[0,1]}$ possess moments of all orders. Moreover, for $\beta<\beta_c$ and $p<4/\beta^2$,

\begin{equation*}
\E\eta_\beta[0,1]^p<\infty.
\end{equation*}
\end{theorem}

The previous result is formulated on a finite interval $[0,1]$ for the sake of simplicity.
From the theory of Gaussian multiplicative chaos and log-correlated fields, it's known that even for $\beta\geq \beta_c$, there's a way of constructing non-trivial random limiting measures. This involves a more complicated normalization procedure. The "critical measure" corresponding to $\beta=\beta_c$, that was already discussed in Section \ref{subs:pre}, is particularly important as it plays a significant role in the study of the maximum of the field, and it essentially determines the distribution of the limiting (atomic) measures for $\beta>\beta_c$. We also prove a result concerning this critical measure.
\begin{theorem}\label{th:critical}
As $N\to\infty$, the random measure 
\begin{equation*}
\sqrt{\log\log N}\frac{|\zeta_{N,\rand}(1/2+ix)|^{\beta_c}}{\E |\zeta_{N,\rand}(1/2+ix)|^{\beta_c}}dx
\end{equation*}
\noindent converges in distribution $($with respect to the topology of weak convergence$)$ to a non-trivial random measure $\xi_{\beta_c}(dx)$ which can be written as $f_{\beta_c}(x)\lambda_{\beta_c}(dx)$, where again $f_{\beta_c}$ is a positive random continuous function such that the norms  $\| f_{\beta}\|_{C^\ell[0,1]}$ and $\| 1/f_{\beta}\|_{C^\ell[0,1]}$ possess moments of all orders,  and $\lambda_{\beta_c}(dx)$ is a critical Gaussian multiplicative chaos measure. 
\end{theorem}

While we expect that Theorem \ref{th:gaussian_appro} could be used to describe how to construct non-trivial limiting measures for $\beta>\beta_c$ and to describe the maximum of $|\zeta_{N,\rand}|$ as in \cite{ABH}, this would require significant analysis of the Gaussian field $\mathcal{G}_N$ and we choose to not go into this here. The proofs of these two results are given in Section \ref{se:subcritical} and Section \ref{sec:critical}. Finally in Appendix \ref{app:appendixa}, we prove the relevant moment estimates for the zeta function, in Appendix \ref{app:appendixb}, we prove a Gaussian approximation result, in Appendix \ref{app:rmt}, we prove an analogue of Theorem \ref{th:main} for random unitary matrices, and in Appendix \ref{app:mesormt}, we prove Theorem \ref{th:mesormt}. 

To conclude this introduction, we offer some further discussion on the background of the problems we study and state some natural conjectures raised by our main results.

\subsection{Further comments, conjectures, and questions} There are several open question that are raised by or are related to  the present paper, and we mention here a few of them. First of all, one should study the properties of the relevant complex Gaussian chaos distribution. E.g. one may ask which finite moments does it possess? What kind of universality does it possess, i.e. under which conditions do different approximations for the complex Gaussian field lead to the same chaos? What are the a.s. exact smoothness properties of the realizations as generalized functions? 
How quickly can $\delta_T\to 0^+$ in the mesoscopic scaling result? Finally, to what extent do this type of limit theorems hold for more general functions such as L-functions.

Further, and likely far more difficult questions, lie in studying fractal properties of the (statistical behavior) of the zeta function. A key tool here would be establishing a connection between $|\mu_T(x)|^\beta$ and real multiplicative chaos. More precisely, one would expect that the measures 

$$
\frac{|\mu_T(x)|^\beta}{\E|\mu_T(x)|^\beta}dx
$$

\noindent would converge to the measures of Theorem \ref{th:gmcconv} for $0<\beta<2$, and a similar statement for $\beta=2$ concerning the measure of Theorem \ref{th:critical}. Such results are also likely to be of key importance in settling some of the conjectures of Fyodorov and Keating. Such results would of course not be very surprising in view of Theorem \ref{th:main}, but there are still significant obstacles. First of all, in proving such results, one typically needs good asymptotics for the normalizing constant $a(T):=\E|\mu_T(x)|^\beta$. Note that asymptotics for $a(T)$ would also be of obvious interest for understanding the typical size of the zeta function on an interval of fixed size. A guess for the asymptotics of this quantity can be obtained from Selberg's central limit theorem (see \cite{RS2,Selberg}) -- namely the fact that

\begin{equation*}%\label{eq:selberg}
(\frac{1}{2}\log\log(T))^{-1/2}\log\big|\zeta (1/2+iT+i\omega T) \big|\overset{d}\longrightarrow N(0,1)
\end{equation*}
as $T\to\infty.$  In view of Selberg's this, one would expect that for some $c(\beta)>0$
\begin{equation}\label{eq:as}
a(T)\sim c(\beta)e^{\frac{1}{4}\beta^2 \log\log T},
\end{equation}
which is exactly the long-standing prediction for the Riemann zeta function.  Unfortunately, \eqref{eq:as} is known unconditionally only for $\beta=2,4$ due to Hardy-Littlewood and Ingham. Also a lower bound of the desired type is known unconditionally \cite{RS1}. A conditional  (assuming the RH) upper bound of the same type was given in \cite{R1,R2,H-B} for $\beta\leq4$, and for $\beta \geq 4$ in \cite{H}. 
In addition, some fairly sharp conditional estimates for the shifted moments such as the two-point function are given in \cite{C}, but bounds of the correct order in $\log T$ are still unknown. The asymptotics of $a(T)$ are actually just the tip of the iceberg. To prove the type of convergence results one would naturally expect for the relevant measures, one would expect to need precise asymptotics for quantities such as $\E \mu_T(x)^\beta \mu_T(y)^\beta$. There are in fact very precise conjectures concerning various kinds of asymptotics of correlation functions of the zeta function -- see e.g. \cite{CK,CFKRS}, but it seems that with current knowledge, such results still remain out of reach. Nevertheless, we formulate as precise conjectures what we expect the connection between the statistical behavior of the zeta function and real multiplicative chaos to be.

\begin{conjecture}\label{con:2} For $\beta\in (0,2)$  the random densities
$$
(\log T)^{-\frac{1}{4}\beta^2}|\zeta (1/2+ix+iT)|^\beta,  \quad  x\in [0,1]
$$ 
converge in distribution to  a constant multiple of the multiplicative chaos measure described in Theorem \ref{th:gmcconv}.
\end{conjecture}

\begin{conjecture}\label{con:3}  The previous conjecture holds for $\beta=\beta_c=2$ as soon as one adds the normalizing factor $(\log\log T)^{1/2}.$
\end{conjecture}

\begin{conjecture}\label{con:4}  There are mesoscopic analogues of the above conjectures and they can be formulated in a similar way as in Section \ref{sec:meso}.
\end{conjecture}

These conjectures all can be viewed as (very strong versions of) statements that the real part of $\log \zeta(1/2+i\omega T+ix)$ behaves like a log-correlated field. In view of Theorem \ref{th:main}, it's of course natural to expect similar statements for the imaginary part of $\log \zeta(1/2+i\omega T+ix)$. This seems even further out of reach currently as the imaginary part of $\log \zeta$ is closely related to the location of the zeros of $\zeta$. This being said, there do exist some results of the flavor of the imaginary part of $\log \zeta$ being log-correlated. In particular, this is closely related to the issue of linear statistics of zeroes of the zeta function, which has been studied by Bourgade and Kuan \cite{BK}, Rodgers \cite{Rod}, as well as Maples and Rodgers \cite{MR}, who found that for suitable test functions, the fluctuations of mesoscopic linear statistics can be described in terms of a log-correlated Gaussian field. Much stronger results exist in the setting of random matrix theory -- for a connection between multiplicative chaos and the corresponding object in random matrix theory, see e.g. \cite{LOS,Webb}. We will not formulate any precise conjectures concerning the imaginary part of $\log \zeta$ and multiplicative chaos, but simply note that this is also an interesting direction of study.

\smallskip

\noindent {\bf Acknowledgement:} We wish to thank Antti Kupiainen for valuable comments on the manuscript.

\section{\texorpdfstring{Existence of the limit $\mu_T\to \mu$: proof of Theorem \ref{th:main}{\rm (i)}}{Proof of Theorem \ref{th:main}{\rm (i)}}}\label{se:convergence}

The structure of this section is as follows: we start  by proving several auxiliary results that rely on basic Fourier analysis. Then we apply the Ingham-Bettin type result presented in Appendix \ref{app:appendixa} to get control of the second moment of $\|\mu_{T}-\mu_{T,N}\|$ in a suitable weighted Sobolev norm. These auxiliary results are formulated in such a manner that after they are established the proof of the first part of Theorem \ref{th:main} will follow effortlessly. 

We begin by defining a suitable space of test functions for our random distributions.

\begin{definition}\label{def:V}
{{\rm The space $V$ to consist of locally integrable functions $g:\R\to \C$ such that 

\begin{equation*}
||g||_V^2:=\int_\R \frac{|g(x)|^2}{1+x^2}dx<\infty.
\end{equation*}}}
\hfill $\blacksquare$
\end{definition}
\noindent Obviously $L^\infty(\R)\subset V$. Next we define another norm that plays a central role in what follows.
\begin{definition}\label{def:norm} {\rm{Assume that $f\in L^1(\R)$, or more generally,  that $f\in {\mathcal S}'(\R)$ is such that $\widehat f\in C(\R)$. Then we set 
$$
\|f\|_{\z}\; :=\; \left(\sum_{n=1}^\infty \frac{|\widehat f(\pii\log (n))|^2}{n}\right)^{1/2}.
$$
If $\|f\|_\z<\infty$ and $\|g\|_\z<\infty$  we write
$$
\langle f,g\rangle_\z := \sum_{n=1}^\infty \frac{\widehat f(\pii\log (n)) \overline{\widehat g(\pii\log (n))}}{n}.
$$
Given an integer $N\geq 1$ we denote by $\N_N$  the positive integers such that their prime factors are contained in the set $\{ p_1,p_2,\ldots , p_N\}$ consisting of the first $N$ primes. Then we write analogously
$$
\|f\|_{\z,N}\; :=\; \left(\sum_{n\in\N_N}\frac{|\widehat f(\pii\log (n))|^2}{n}\right)^{1/2} \quad\textrm{and}
$$
$$
\langle f,g\rangle_{\z,N} := \sum_{n\in\N_N}\frac{\widehat f(\pii\log (n))\overline{\widehat g(\pii\log (n))}}{n}.
$$}}
\hfill $\blacksquare$
\end{definition}
Above the assumption that $\widehat f$ is continuous makes sure that point evaluations $\widehat f(a)$ for $a\in\R$ are well-defined, and hence the quantity $\|f\|_\z$ is well-defined, but of course it may still take the value $\infty. $
It is useful to note that for any such $f$ one has
\begin{equation}\label{eq:meno}
\| f\|_{\z}\;=\; \lim_{N\to\infty}\| f\|_{\z,N}
\end{equation}

We shall make use of the following embedding result.
\begin{lemma}\label{le:embed} There exists a positive constant $C$ such that for all $f\in L^1$ 
$$
\|f\|_\z^{2}\; \leq \; C \int_\R (1+x^2)|f(x)|^2dx.
$$
\end{lemma}
\begin{proof}
Let us write $g:=\widehat f$. By the definition of the Sobolev space $W^{1,2}(\R)$ it is equivalent to prove the 
inequality
$$
\|f\|_\z^2 =\sum_{n=1}^\infty \frac{|g(\pii\log (n))|^2}{n}\lesssim \|g\|_{W^{1,2}(\R)}^2.
$$
Approximating by smooth functions we may also assume that $g\in\mathcal{S}(\R)$. In this case, $||f||_\z$ is finite and by using the dual representation of the $\ell_2$ norm, we see that 

\begin{align*}
\|f\|_\z\; =\;&\sup_{\|(b_n)\|_{\ell^2}\leq 1}\big| \sum_{n=1}^\infty  b_nn^{-1/2}g(\pii\log(n))\big|\;=\;\sup_{\|(b_n)\|_{\ell^2}\leq 1}\big|\langle \sum_{n=1}^\infty  b_nn^{-1/2}\delta_{\pii\log n},g\rangle\big|,
\end{align*}
so it is enough to show that
\begin{equation}\label{eq:jazz1}
\|\sum_{n=1}^\infty  b_nn^{-1/2}\delta_{\pii\log n}\|_{W^{-1,2}(\R)}\lesssim \|(b_n)\|_{\ell^2}.
\end{equation}

We need the auxiliary estimate
\begin{equation}\label{eq:jazz2}
\|\delta_u-h^{-1}\chi_{(u,u+h]}\|_{W^{-1,2}(\R)}\lesssim h^{1/2}\quad \textrm{for}\quad u\in \R,\; h>0.
\end{equation}
For this end, we may translate to $u=0$ and  simply compute that the square of the left hand side equals
\begin{align*}
&\int_\R \left|  1- \frac{1-e^{-2\pi i h\xi}}{2\pi ih\xi}\right|^2\frac{d\xi}{1+\xi^2}
\lesssim\int_{|\xi|\leq h^{-1}}h^2d\xi+\int_{|\xi|> h^{-1}}\frac{d\xi}{1+\xi^2}
\lesssim h,
\end{align*}
where we noted that by Taylor's formula $\big|  1- \frac{1-e^{-2\pi i h\xi}}{2\pi ih\xi}\big|^2\lesssim |h\xi|^2$ for $|\xi|\leq h^{-1}.$

To return to the proof of \eqref{eq:jazz1}, write $\widetilde b_n:= b_nn^{-1/2}(\frac{1}{2\pi}(\log(n+1)-\log (n)))^{-1}$ so that $|\widetilde b_n|\sim 2\pi n^{1/2}|b_n|$, and observe that \eqref{eq:jazz2} yields the estimate
\begin{align}\label{eq:jazz3}
&\left\| \sum_{n=1}^\infty b_n n^{-1/2}\delta_{\pii\log n} \; - \; \sum_{n=1}^\infty \widetilde b_n\chi_{(\pii\log(n),\pii\log (n+1)]}\right\|_{W^{-1,2}(\R)} \\
\notag \lesssim &\sum |b_n|n^{-1/2}(\frac{1}{2\pi}\log(n+1)-\frac{1}{2\pi}\log(n))^{1/2}\;\lesssim \; \sum_{n=1}^\infty |b_n|n^{-1}\lesssim \|(b_n)\|_{\ell^2}.
\end{align}
On the other hand, as  the characteristic functions have disjoint supports we obtain
\begin{align}%\label{eq:jazz4}
&\| \sum_{n=1}^\infty \widetilde b_n \chi_{(\pii\log(n),\pii\log (n+1)]}\|_{W^{-1,2}(\R)} 
\leq \;\| \sum_{n=1}^\infty \widetilde b_n \chi_{(\pii\log(n),\pii\log (n+1)]}\|_{L^2(\R)} \lesssim \;\|(b_n)\|_{\ell^2},\notag
\end{align}
Now \eqref{eq:jazz1} is an immediate consequence of  this estimate in combination with \eqref{eq:jazz3}.
\end{proof}
\begin{remark}\label{re:optimality}
{\rm By slightly modifying the  above argument  one obtains the sharper result $
\|f\|_\z^{2}\; \lesssim\int_\R (1+x^2)^a|f(x)|^2dx
$
for $a>1/2$, which is  optimal as one can check that the estimate fails for the choice $a=1/2$. }
\hfill $\blacksquare$
\end{remark}

The following computation of a specific Fourier transform is used to establish  the slightly delicate  equality \eqref{eq:i22} which is crucial for our purposes.
\begin{lemma}\label{le:zeta_hat} Assume that $a>0$ and define the function $h_a:\R\to\C$ by setting  $h_a(u):=\displaystyle\zeta (1+iu)-\frac{e^{-iau}}{iu}$, where the value at $u=0$ is defined as the limit as $u\to 0.$ Then
$$
\widehat h_a\;=\; \sum_{n=1}^\infty n^{-1}\delta_{-\pii \log(n)}\;-\;  2\pi\chi_{(-\infty, -\pii a]}.
$$
\end{lemma}
\begin{proof}
Observe that $g(z):= \zeta (1+z)-z^{-1}e^{-az}$ is entire and as $z\to \infty$, $g(z)$ grows at most logarithmically in the half-plane $\{\Re z \geq 0\}$. This implies that its Fourier-transform over the imaginary axis is obtained as the limit of the Fourier transform over the line $\{\Re z =\varepsilon \}$, in the limit $\varepsilon\to 0^+$. In other words,
$$
\mathcal{F}(h_a) \;=\; \lim_{\varepsilon\to 0^+}\mathcal{F}(\zeta(1+\varepsilon+i\cdot)) - \lim_{\varepsilon\to 0^+}\mathcal{F}((\varepsilon+i\cdot)^{-1}(e^{-a(\varepsilon+i\cdot)}))\;=:\; \eta_1-\eta_2,
$$
with convergence in $\mathcal{S}'$, as soon as we verify that both limits exist separately. Since $\zeta(1+\varepsilon+iu)=\sum_{n=1}^{\infty}n^{-1-\varepsilon}e^{-i\log(n)u},$ with uniform convergence, we immediately see that
$$
 \eta_1=\lim_{\varepsilon\to 0^+}\sum_{n=1}^\infty n^{-1-\varepsilon}\delta_{-\pii \log(n)}=\sum_{n=1}^\infty n^{-1}\delta_{-\pii \log(n)}.
$$
In turn, we observe that for any $\varepsilon>0$ the function $2\pi e^{2\pi \varepsilon \xi}\chi_{(-\infty, -\pii a]}(\xi)$ belongs to $L^1(\R)$ and 
$$
2\pi \int_{-\infty}^{-\pii a}e^{2\pi i u\xi} e^{2\pi \varepsilon \xi}d\xi = \int_{-\infty}^{-a}e^{\ i u\xi'} e^{ \varepsilon \xi'}d\xi' = \frac{e^{-a(\varepsilon+iu)}}{\varepsilon+iu}.
$$ 
In other words, the Fourier transform of $\frac{e^{-a(\varepsilon+iu)}}{\varepsilon+iu}$ equals $2\pi e^{2\pi \varepsilon \xi}\chi_{(-\infty, -\pii a]}(\xi)$, and letting  $\varepsilon\to 0^+$ we see that $\eta_2=2\pi\chi_{(-\infty, -\pii a]}.$
\end{proof}

Our next lemma connects the norms $\|\cdot \|_\z$ and $\|\cdot \|_{\z,N}$ to the Riemann zeta function and the truncated Euler product. The principal claims are formulas \eqref{eq:i22} and  \eqref{eq:i33}. Recall that the truncated Euler product was defined by
$$
\zeta_N(s):=\prod_{k=1}^N(1-p_k^{-s})^{-1}.
$$
\begin{lemma}\label{eq:limits} Assume that $f,g\in L^1$ and $\int_\R (1+x^2)(|f(x)|^2+|g(x)|^2)dx<\infty$. Then for any $N\geq 1$
\begin{align}\label{eq:i2}
\langle f,g\rangle_{\z}\;&=\; \lim_{\varepsilon\to 0^+}\int_{\R^2} f(x)\zeta(1+\varepsilon +i(x-y))\overline{g(y)}dxdy\\
\notag\;&=\;  \lim_{N\to\infty}\langle f,g\rangle_{\z,N}
\end{align}
and
\begin{equation}\label{eq:i1}
\langle f,g\rangle_{\z,N}\;=\; \int_{\R^2} f(x)\zeta_N(1+i(x-y))\overline{g(y)}dxdy.
\end{equation}
We also have
\begin{align}\label{eq:i22}
\langle f,g\rangle_{\z}\;&= \lim_{T\to\infty}K_T(f,g),
\end{align}
where
$$
K_T(f,g)\; =\; \int_{\R^2} f(x)\left(\zeta(1+i(x-y))+\frac{\zeta(1-i(x-y))}{1-i(x-y)}T^{-i(x-y)}\right)\overline{g(y)}dxdy
$$
Moreover, there is a constant $C$ such that 
\begin{align}\label{eq:i33}
|K_T(f,g)|\leq C \left(\int_\R (1+x^2)|f(x)|^2dx\right)^{1/2}\left(\int_\R(1+y^2)|g(y)|^2dy\right)^{1/2} \quad\textrm{for all}\quad T\geq 1.
\end{align} 
\end{lemma}
\begin{proof} Note that under our assumption on $f$ (and $g$), we have by Cauchy-Schwarz that for some $\varepsilon >0$ 
\begin{equation}\label{eq:i3}
\int_\R |f(x)|(1+|x|)^\varepsilon <\infty. 
\end{equation}
In particular, this entails that all the integrals in the statements are well defined, due to the uniform boundedness of $\zeta_N$ (which follows from the definition of $\zeta_N$) and the at most logarithmic growth of $\zeta$ over the line $\sigma=1$ -- see e.g. \cite[Theorem 3.5]{Titchmarsh}.

We observe first that 
$$
\zeta_N(1+i(x-y))= \sum_{n\in \N_N}\frac{n^{-i(x-y)}}{n}.
$$
Since  $\sum_{n\in \N_N}\frac{1}{n}=\prod_{k=1}^N(1-p_k^{-1})^{-1}<\infty$, we may bring the integral in \eqref{eq:i1} under the sum, and to prove  \eqref{eq:i1} it remains to note that
$$
\int_{\R^2} n^{-1-i(x-y)}f(x)\overline{g(y)}dxdy= n^{-1}\widehat f(\pii\log n)\overline{\widehat g(\pii\log n)}.
$$

Next, the fact that $\sum_{n=1}^\infty n^{-1-\varepsilon}<\infty$ allows us to deduce in an analogous way that 
$$
\int_{\R^2}f(x)\zeta(1+\varepsilon +i(x-y))\overline{g(y)}dxdy= \sum_{n=1}^\infty n^{-1-\varepsilon}\widehat f(\pii\log n)\overline{\widehat g(\pii\log n)},
$$
and by letting $\varepsilon\to 0^+$ we deduce  \eqref{eq:i2} as according to Lemma \ref{le:embed} and the definition of the inner product $\langle f, g\rangle_\z,$ the right hand series above has a convergent majorant series of the form
$ \sum_{n=1}^\infty n^{-1}|\widehat f(\pii\log n)||\overline{\widehat g(\pii\log n)}|$.

In order to prove  \eqref{eq:i22}, we note first that   that the function $s\mapsto (1-s)^{-1}\zeta(1-s)+s^{-1}$ is analytic and bounded over the imaginary axis, whence  \eqref{eq:i3} implies that
$$
f(x)\overline{g(y)}\big(\frac{\zeta(1-i(x-y))}{1-i(x-y)}+ (i(x-y))^{-1}\big)
$$
is integrable over $\R^2$.
A fortiori,
$$
\lim_{T\to\infty}\int_{\R^2} \left[f(x)\overline{g(y)}\Big(\frac{\zeta(1-i(x-y))}{1-i(x-y)}+ (i(x-y))^{-1}\Big)\right]T^{-i(x-y)}dxdy =0
$$
by the Riemann-Lebesgue lemma.
In the notation of Lemma \ref{le:zeta_hat} it remains to show that
\begin{align*}%\label{eq:i4}
\langle f,g\rangle_{\z}\; =\;  \lim_{T\to\infty}\int_{\R^2} h_{\log T}(x-y)f(x)\overline{g(y)}dxdy
\end{align*}
Assume first that $f,g\in C_0^{\infty}(\R)$. Then we may compute directly
$$
\int_{\R^2} f(x)\overline{g(y)}h_{\log T}(x-y)dxdy=\int_\R\widehat f(\xi)\overline{\widehat g(\xi)}\; \widehat h_{\log T}(-\xi)d\xi,
$$
where the integral on the right hand side has to be understood as the distributional pairing between the Schwartz test function $\widehat f(\xi)\overline{\widehat g(\xi)}$ and the Schwartz distribution $\widehat h_{\log T}(-\cdot)$. Lemma \ref{le:zeta_hat} verifies that the right hand side equals
\begin{align}\label{eq:i5}
 \; &\sum_{n=1}^\infty n^{-1}\widehat f(\pii\log(n))\overline{\widehat g(\pii\log(n))} - 2\pi\int^{\infty}_{\pii \log T}\widehat f(\xi)\overline{\widehat g(\xi)}d\xi.
\end{align}
For general $f,g$, we note that in particular, our assumptions imply that $f,g\in L^2$. Then by Lemma \ref{le:embed} and the $L^2$-continuity of the Fourier transform we may approximate general $f$ and $g$ by test functions and take limits in \eqref{eq:i5} to verify its validity in full generality. Finally, \eqref{eq:i22} and \eqref{eq:i33} are immediate consequences of \eqref{eq:i5} and Lemma \ref{le:embed}.
\end{proof}

In what follows it is notationally convenient to introduce the 'weight regularization' of a given function $f:\R\to\C$ by setting
$$
\reg{f} (x):=(1+x^2)^{-1}f(x).
$$
Directly from the definitions we note that  if $f\in V$ (recall Definition \ref{def:V}) we have for $T\geq 1$
\begin{align*}
\reg{\mu}_T(f)\;=\;&\mu_T(\reg{f})\; =\; \int_\R\zeta(1/2+ix+iT\omega)\reg{f}(x)dx\quad\textrm{and}\\
\E\reg{\mu}_T(f)\; =\; &\E \mu_T(\reg{f})\; =\;\frac{1}{T}\int_{0}^T\int_\R\zeta(1/2+ix+it)\reg{f}(x)dxdt,
\end{align*}
These quantities are well defined since by Cauchy-Schwarz and a classical growth bound on the zeta function, which we recall in \eqref{eq:zcgrowth}, we deduce that for some $C_T,\widetilde{C}_T>0$
\begin{align*}
|\mu_T(\reg{f})|&\leq C_T \int_\R |\reg{f}(x)|(1+x^2)^{1/12}dx\\
&=C_T \int_\R \big(|f(x)|(1+x^2)^{-1/2}\big)(1+x^2)^{1/2+1/12-1}dx
\;\leq \;\widetilde{C}_T\|f\|_{V}.
\end{align*}
Analogous formulas hold for $\mu_{T,N}(f)$ with $\zeta$ replaced by $\zeta_N$.

We next turn to the crucial estimates for the two-point functions. Let us first give names to the relevant objects.
\begin{definition}\label{def:v}{\rm{
Let $\omega$ be uniform on $[0,1]$, $T>0$, and $x,y\in \R$. Then we set
\begin{align*}
 V_T^{(1)}(x,y)\;=\;&\E\Big[\zeta(1/2+i\omega T+ix)\overline{\zeta(1/2+i\omega T+iy)}\Big],\\
V_T^{(2,1)}(x,y)\;=\;&\E\left[\zeta_N(1/2+i\omega T+ix)\overline{\zeta(1/2+i\omega T+iy)}\right],\\
V_T^{(2,2)}(x,y)\;=\;&\E\left[\zeta(1/2+i\omega T+ix)\overline{\zeta_N(1/2+i\omega T+iy)}\right],\\
V_T^{(3)}(x,y)\;=\;&\E\left[\zeta_N(1/2+i\omega T+ix)\overline{\zeta_N(1/2+i\omega T+iy)}\right],
\end{align*}
and finally
\begin{align*}
V_T^{(2)}(x,y)\;=\;&V_T^{(2,1)}(x,y)+ V_T^{(2,2)}(x,y).
\end{align*}}}
\hfill $\blacksquare$
\end{definition}
A direct application of Fubini shows that
\begin{align}\label{eq:cov}
\E \Big[\mu_T(\reg{f})\overline {\mu_T(\reg{g})}\Big]\;=\;\int_{\R^2} \reg{f}(x)\overline{\reg{g}(y)} V_T^{(1)}(x,y)dxdy.
\end{align}
Analogous formulas hold for  
$$
\E \big[\mu_{T,N}(\reg{f})\overline {\mu_T(\reg{g})}\big],\;\;  \E \big[\mu_T(\reg{f})\overline{\mu_{T,N}(\reg{g})}\big],
\; \;\textrm{and}\;\; \E \big[\mu_{T,N}(\reg{f})\overline {\mu_{T,N}(\reg{g})}\big],
$$
respectively, where one replaces $V_T^{(1)}(x,y)$ in \eqref{eq:cov} by
$V_T^{(2,1)}(x,y)$, $V_T^{(2,2)}(x,y)$, and $V_T^{(3)}(x,y)$, respectively. Especially, one should observe that for real-valued test functions $f$ one has
\begin{align}\label{eq:risti}
&\E \Big[\mu_T(\reg{f})\overline {\mu_{T,N}(\reg{f})}+\mu_{T,N}(\reg{f})\overline {\mu_{T}(\reg{f})}\Big]=\;& \int_{\R^2} \reg{f}(x)\reg{f}(y) V_T^{(2)}(x,y)dxdy.
\end{align}

We next establish asymptotics for the kernels $V_T^{(\cdot)}(x,y)$. In the case of  $V_T^{(1)}(x,y)$ this essentially reduces  to a classical result on second moments of the shifted zeta function due already to Ingham \cite{Ingham},  and to a generalization by Bettin \cite{Bettin}. However, we also need asymptotics of mixed quantities like $V_T^{(2)}(x,y)$, and Bettin's proof  does not work as such in this case.  Due to these reasons, and for the convenience of a reader who is not an expert in number theory, we choose to provide the relevant arguments in  detail in Appendix \ref{app:appendixa}, where we actually  modify the approach of \cite{Atkinson} for our purposes. However, see Remark \ref{re:cross alternative} below in this connection.
\begin{theorem}\label{th:twopoint}
As $T\to\infty$,
\begin{align}\label{eq:error11}
V_T^{(1)}(x,y)=&\zeta(1+i(x-y))+\frac{\zeta(1-i(x-y))}{1-i(x-y)}\left(\frac{T}{2\pi}\right)^{-i(x-y)}
\;+\; E_1(x,y,T)
\end{align}

and for fixed $N$, 
\begin{equation}\label{eq:error22}
V_T^{(2)}(x,y)=2\zeta_N(1+i(x-y))\;+\;E_2(x,y,T),
\end{equation}

\noindent where 

$$
E_1(x,y,T),E_2(x,y,T)=\mathcal{O}\Big( T^{-1/12}\big(1+|x|^{5/12}+|y|^{5/12})\Big).
$$

Fox fixed $N\geq 1$ we have
\begin{equation}\label{eq:error33}
V_T^{(3)}(x,y)=\zeta_N(1+i(x-y))\;+\;\mathit{o}(1),
\end{equation}
where $\mathit{o}(1)$   is uniform in $x,y\in\R$ as $T\to\infty$.
Above  the error estimates for $V_T^{(2)}(x,y)$ and $V_T^{(3)}(x,y)$ are not  necessarily uniform in $N$. 
\end{theorem}

\begin{proof}
See Appendix \ref{app:appendixa}.
\end{proof}
\begin{remark}\label{re:real-valued}{\rm  For simplicity, we state the above asymptotics only for the sum term $V_T^{(2)}(x,y)$, instead of the
summands $V_T^{(2,i)}(x,y)$, so that in some of the exact formulas below we need to assume that the test function $f$ is real valued. However, this will be enough for the relevant estimates.}
\hfill $\blacksquare$
\end{remark}

\begin{lemma}\label{le:conv1} Assume that $f$ is real valued and $f\in V.$ Then
\begin{align}\label{eq:bound1}
\E|\mu_T(\reg{f})|^2\;\leq\; C\| \reg{f}\|_{\z}^2\; \lesssim\| f\|_{V}^2
\end{align}
where $C$ is independent of both $T\geq 1$ and $f\in V.$ Moreover, 
\begin{align}\label{eq:bound2}
\lim_{T\to\infty}\E|\mu_T(\reg{f})|^2 \;=\; \|\reg{f}\|_\z^2\;\lesssim\| f\|_{V}^2.
\end{align}
\end{lemma}
\begin{proof}
Let us write
$$
V_T^{(1)}(x,y)=\zeta(1+i(x-y))+\frac{\zeta(1-i(x-y))}{1-i(x-y)}\left(\frac{T}{2\pi}\right)^{-i(x-y)} +E_1(x,y,T)
$$
Due to the error estimate \eqref{eq:error11} and Cauchy-Schwarz we have
\begin{align*}%\label{eq:E1}
\notag &\Big|\int_{\R^2}E_1(x,y,T)\reg{f}(x)\overline {\reg{f}(y)}dxdy\Big|\\
\;\lesssim\;&T^{-1/12}\int_{\R^2}|\reg{f}(x)||\overline {\reg{f}(y)}|\big(1+|x|^{5/12}+|y|^{5/12})dxdy\\
\notag \leq\; & CT^{-1/12}\int_{\R^2}|f(x)||f(y)|(1+x^2)^{5/24-1}(1+y^2)^{5/24-1}dxdy\\
\leq\; &CT^{-1/12}\Big(\int_\R |f(x)|^2(1+x^2)^{-1}dx\Big),\notag
\end{align*}
which shows that the contribution from the error term vanishes in the limit $T\to\infty.$ Finally, by recalling the notation of Lemma \ref{eq:limits} we see that the contribution of the main term equals $K_{\log \frac{T}{2\pi}}(\reg{f}, \reg{f}),$ and the uniform boundedness along with the statement about the limit of this term follows immediately from Lemma  \ref{eq:limits} as soon as we observe that $\int |\reg{f}(x)|^2(1+x^2)=\|f\|_{V}^2.$
\end{proof}
We have analogous estimates for  similar quantities arising from the truncated Euler product and for the cross terms involved.

\begin{lemma}\label{le:conv2} Assume that $f$ is real-valued and $f\in V.$ Then
\begin{equation}\label{eq:bound3}
\E |\mu_{T,N}(\reg{f})|^2\;\leq\; C(N)\| \reg{f}\|_{\z}^2\; \leq\;  C(N)\| f\|_{V}^2,
\end{equation}
where $C(N)$ is independent of $T\geq 1$ and $f\in V.$ Moreover, 
\begin{align}\label{eq:bound4}
\lim_{T\to\infty}\E|\mu_{T,N}(\reg{f})|^2 \;=\; \|\reg{f}\|_{\z,N}^2\;\lesssim\;\| f\|_{V}^2.
\end{align}
In a similar vein, for all $T\geq 1$ one has
\begin{align}\label{eq:bound5}
\big| \E\big(\mu_{T,N}(\reg{f})\overline{\mu_{T}(\reg{f})}+ \mu_{T}(\reg{f})\overline{\mu_{T,N}(\reg{f})}\big)\big|
 \leq\;  C(N)\| f\|_{V}^2,
 \end{align}
 and
 \begin{align}\label{eq:bound6}
\lim_{T\to\infty}\E\big(\mu_{T,N}(\reg{f})\overline{\mu_{T}(\reg{f})}+ \mu_{T}(\reg{f})\overline{\mu_{T,N}(\reg{f})}\big) \;=\; 2\|\reg{f}\|_{\z,N}^2\;\leq\; C(N)\| f\|_{V}^2.
\end{align}
\end{lemma}
\begin{proof}
The statements \eqref{eq:bound3} and \eqref{eq:bound4} are proven exactly as the corresponding statements in the previous lemma
by invoking the error estimate  \eqref{eq:error33}. In turn, \eqref{eq:bound5}  is obtained by Cauchy-Schwarz from \eqref{eq:bound1} and \eqref{eq:bound3}, and the proof of  \eqref{eq:bound6} is also analogous in view of  \eqref{eq:error22} and \eqref{eq:risti}.
\end{proof}

As an immediate consequence of the previous lemmata we will now deduce an interesting intermediate result, which already gives quantitative estimates for our approximation when considering a fixed test function. We also point out here a fact that can be easily checked by adapting the rest of our proof in a simple way, namely that when considering a fixed test function, we don't actually need any smoothness from it -- it's enough for it to be locally square integrable and have nice enough decay at infinity. Thus the action of our limiting object can be defined for some test functions that are not in the Sobolev space $W^{\alpha,2}(\R)$. This is a common phenomenon when studying random generalized functions.

\begin{proposition}\label{pr:A} For each $N\geq 1$ there exists a constant $C(N)<\infty$ such that  for any {\rm (}possibly complex-valued{\rm )} $f\in V$ 
\begin{align}\label{eq:bound7}
\E|\mu_T(\reg{f})- \mu_{T,N}(\reg{f})|^2\;\leq\; C(N)\| \reg{f}\|_{\z}^2\; \leq\;  C(N)\| f\|_{V}^2\;
\end{align}
Moreover, we have
\begin{align}\label{eq:bound8}
\limsup_{T\to\infty}\E|\mu_T(\reg{f})- \mu_{T,N}(\reg{f})|^2 \;\leq \; 4(\|\reg{f}\|_\z^2-\|\reg{f}\|_{\z,N}^2)\;\leq\; C' \| f\|_{V}^2.\;,
\end{align}
where $C'$ is independent of $f$ and $N$.
\end{proposition}
\begin{proof}
Let us first assume in addition that $f$ is real-valued. Now \eqref{eq:bound7} is an immediate consequence of the previous two lemmata. In addition, by combining \eqref{eq:bound2}, \eqref{eq:bound4}, and \eqref{eq:bound6} we find
\begin{equation*}%\label{eq:tasan}
\E|\mu_T(\reg{f})- \mu_{T,N}(\reg{f})|^2 \; \to\;   \|\reg{f}\|_{\z}^2- \|\reg{f}\|_{\z,N}^2 \quad\textrm{as}\quad T\to\infty.
\end{equation*}
Then \eqref{eq:bound8} follows in the case of complex-valued $f$ by considering separately the corresponding quantity for real and imaginary parts of $f$.
\end{proof}
One should note that the previous proposition shows that for a fixed  test function $f$, choosing \emph{first} $N$ large (depending on $f$), and \emph{then} $T$  large enough (depending on $f$ and $N$), the second moment of $\mu_T(\reg{f})- \mu_{T,N}(\reg{f})$ can be made arbitrarily small.

We fix $\alpha>1/2$  for the rest of this section and consider both $\reg{\mu}_T$ and $\reg{\mu}_{T,N}$ as $W^{-\alpha,2}(\R)$-valued 
random variables. Actually, by the growth estimates of $\zeta$, namely \eqref{eq:zcgrowth} and the fact that $\mu_{T,N}$ is bounded, they are even $L^2(\R)\cap L^1(\R)$-valued random variables. Measurability is obvious since they are continuous  mappings from $\Omega:=[0,1]$ to $L^2(\R)$  (and $W^{-\alpha,2}(\R)$). Our next step is to extend the approximation hinted at in the previous proposition to the level of mean square approximation in the Sobolev space.

\begin{proposition}\label{pr:B} Assume that $\alpha>1/2$. Then 
\begin{align*}%\label{eq:bound10}
\lim_{N\to\infty}\limsup_{T\to\infty} \E\|\reg{\mu}_T-\reg{\mu}_{T,N}\|_{W^{-\alpha,2}(\R)}^2 \; =\; 0
\end{align*}
and, consequently
\begin{align*}%\label{eq:bound11}
\lim_{N\to\infty}\limsup_{T\to\infty} \wass_2\big(\reg{\mu}_T\;,\;\reg{\mu}_{T,N}\big)_{W^{-\alpha,2}(\R)} \; =\; 0.
\end{align*}

\end{proposition}
\begin{proof}
We first estimate the $\limsup$ in the statement  for a fixed $N$. As noted above, the random variables actually take values also in $L^1(\R)$, and hence their Fourier-transform can be considered pointwise. Denote the complex exponential function $x\to e^{-2\pi i\xi x}$ by $e_\xi$ and note that we have the estimate $\|{e_\xi}\|_{V}\leq c_0$ for all $\xi\in \R$. Hence we may apply Fubini's theorem to compute
\begin{align*}
&\E\|\reg{\mu}_T-\reg{\mu}_{T,N}\|_{W^{-\alpha,2}(\R)}^2=\int_{\R}(1+\xi^2)^{-\alpha}\E \big|\mu_T\big(\reg{e}_\xi\big)-\mu_{T,N}\big(\reg{e}_\xi\big)\big|^2d\xi
\end{align*}
Proposition \ref {pr:A} verifies that the integrand has an integrable majorant of the form $C'(N)(1+\xi^2)^{-\alpha}$ and that
$$
\limsup_{T\to\infty}\E \big|\mu_T\big(\reg{e}_\xi\big)-\mu_{T,N}\big(\reg{e}_\xi\big)\big|^2\leq C''\big(\|\reg{e}_\xi\|^2_\z-\|\reg{e}_\xi\|^2_{\z,N}\big).
$$
Clearly the integrand is continuous  jointly in  $T$ and $\xi$, and hence a simple application of the dominated convergence theorem shows that
$$
\limsup_{T\to\infty}\E\|\reg{\mu}_T-\reg{\mu}_{T,N}\|_{W^{-\alpha,2}}^2\; \leq \;
C''\int_{\R}(1+\xi^2)^{-\alpha}\big(\|\reg{e}_\xi\|^2_\z-\|\reg{e}_\xi\|^2_{\z,N}\big)d\xi.
$$
The proposition now follows by applying \eqref{eq:meno}.
\end{proof}

We next record a simple fact whose validity is seen easily by approximating the full sum by a partial sum.
\begin{lemma}\label{le:simple} Assume that $X_n,$ $n\geq 1$, are random variables and
assume that $Y_{T,n}$ are uniformly bounded random variables, $|Y_{T,n}|\leq 1$ for all $T\in [1,\infty)$ $n\geq 1.$  Assume also that for
 all $\ell\geq 1$
$$
(Y_{T,1},Y_{T,2},\ldots , Y_{T,\ell})\overset{d}{\longrightarrow} (X_1,X_2,\ldots,X_\ell)\quad\textrm{as}\quad T\to\infty.
$$
Then, if $(u_j)$ is a sequence of elements in a Banach space $E$ such that $\sum_{n=1}^\infty \|u_n\|_E<\infty$, one has
$$
\wass_2(\sum_{n=1}^\infty Y_{T,n}u_n\; ,\;  \sum_{n=1}^\infty X_{n}u_n)_E\to 0\quad \textrm{as}\quad T\to\infty.
$$
\end{lemma}

For the rest of this section we shall denote by $\theta_k$, $k \geq 1$, i.i.d. random variables, uniformly distributed on $[0,1].$ Then the variables 
$$
e^{2\pi i\theta_1},\; e^{2\pi i\theta_2},\ldots
$$
are i.i.d. and uniformly distributed on the unit circle. The truncated randomized Euler product is defined by the formula
\begin{equation*}%\label{eq:zetarandN}
\zeta_{N,\rand}(s):= \prod_{k=1}^N\left(\frac{1}{1-p_k^{-s}e^{2\pi i\theta_k}}\right),
\end{equation*}
and for us, most often $\zeta_{N,\rand}$ is shorthand for $\zeta_{N,\rand}(1/2+ix).$
The next proposition constructs the final limiting element $\zeta_\rand$ and shows that it is well-approximated in distribution by suitably chosen $\mu_{T,N}$.
\begin{proposition}\label{pr:C}  {\bf (i)}\quad Assume that $\alpha >1/2$. For each $N\geq 1$
\begin{equation}\label{eq:iml1}
\wass_2\Big(\reg{\mu}_{T,N}\;,\;\reg{\zeta}_{N,\rand}\Big)_{W^{-\alpha,2}(\R)}\to 0\quad\textrm \quad{as}\quad T\to\infty.
\end{equation}

\smallskip

\noindent {\bf (ii)}\quad Almost surely there exists a $W^{-\alpha,2}(\R)$-valued limiting random variable
\begin{equation}\label{eq:bohr}
\reg{\zeta}_\rand\; := \; \lim_{N\to\infty}\reg{\zeta}_{N,\rand}.
\end{equation}
In addition,
\begin{equation*}%\label{eq:iml2}
\wass_2\Big(\reg{\zeta}_{\rand}\; ,\; \reg{\zeta}_{N,\rand}\Big)_{W^{-\alpha,2}(\R)}\to 0\quad\textrm  \quad{as}\quad N\to\infty.
\end{equation*}
\end{proposition}
\begin{proof}  We shall make use of the well-known fact that for any any $N\geq 1$ as $T\to \infty$ one has that
$$
(p_1^{-i\omega T},...,p_N^{-i\omega T})\overset{d}{\longrightarrow} (e^{2\pi i\theta_1},...,e^{2 \pi i\theta_N})
$$
This follows simply by observing that if $r_1,...,r_N$ are integers which aren't all zero, then all the mixed moments
\begin{equation*}
\E \prod_{k=1}^N p_k^{ir_k \omega T}=\int_0^1e^{iuT \sum_{k=1}^N r_k\log p_k}du,
\end{equation*}
tend  to zero  as we let $T\to\infty$ since $\sum_{k=1}^N r_k\log p_k\neq 0$ by the uniqueness of the prime number decomposition of integers. As a consequence,
for any $\ell\geq 1$
\begin{equation}\label{eq:bohr2}
(1^{-i\omega T},2^{-i\omega T},...,\ell^{-i\omega T})\overset{d}{\longrightarrow} (e^{2\pi i\theta\cdot \alpha(1)},\ldots, e^{2\pi i\theta\cdot \alpha(\ell)})
\quad \textrm{as}\quad T\to\infty,
\end{equation}
where $\theta:=(\theta_1,\theta_2,\ldots)$  and the sequence $\alpha(\ell)=(\alpha_1(\ell),\alpha_2(\ell),\ldots)$ is defined via $\ell=\prod_{k=1}^\infty p_k^{\alpha_k(\ell)}.$
We may write
$$
\reg{\mu}_{T,N}(x)=\sum_{n\in\N_N}n^{-1/2-ix}\frac{n^{-i\omega T}}{1+x^2},
$$
where $\sum_{n\in\N_N}n^{-1/2}<\infty.$
Hence Lemma \ref{le:simple} yields \eqref{eq:iml1} in view of \eqref{eq:bohr2} even with respect to the 2-Wasserstein distance in $L^2$ instead of $W^{-\alpha,2}(\R).$ This proves part (i).

In order to treat (ii), we observe that by the Gauss mean value theorem 
$$
\E \left(\frac{1}{1-p_k^{-s}e^{2\pi i\theta_k}}\right)=1
$$
 for any
$s\in \{\sigma>0\}$. Hence, by the very definition and independence, the sequence of random variables $(\zeta_{N,\rand})_{N\geq 1}$ forms a $W^{-\alpha,2}(\R)$-valued martingale sequence, with respect to the natural filtration $\mathcal{F}_N:=\sigma (\theta_1,\ldots,\theta_N)$. Moreover, we claim that this martingale is bounded in  $L^2$:
$$
\E\big( \| \zeta_{N,\rand}\|_{W^{-\alpha,2}(\R)}^2\big) \leq C<\infty,
$$
where $C$ does not depend on $N$. In order to verify the above uniform bound we compute as in the proof of Proposition \ref{pr:B}: 
for any $T\geq 1$
\begin{align}\label{eq:iml3}
\E\|\reg{\mu}_{T,N}\|_{W^{-\alpha,2}(\R)}^2
=\int_{\R}(1+\xi^2)^{-\alpha}\E \big|\mu_{T,N}\big(\reg{e}_\xi\big)\big|^2d\xi.
\end{align}
Since obviously $\sup_{T\geq 1}\|\mu_{T,N}\|_{L^\infty(\R)}$ is bounded in the probability space variable, also $\E \big|\mu_{T,N}\big(\reg{e}_\xi\big)\big|^2$ is bounded in $\xi,T$. We deduce that the integrand in  \eqref{eq:iml3} has an integrable majorant.  In view of \eqref{eq:bound4} we see that 
$$
\lim_{T\to\infty}\E \big|\mu_{T,N}\big(\reg{e}_\xi\big)\big|^2=\|\reg{e}_\xi\|_{\z,N}^2\leq c_0
$$
and by recalling \eqref{eq:iml1} and  taking the limit $T\to\infty$ inside the integral \eqref{eq:iml3} it finally follows that
$$
\E\big( \| \zeta_{N,\rand}\|_{W^{-\alpha,2}(\R)}^2\big) \leq c_0 \int_{\R}(1+\xi^2)^{-\alpha}d\xi
\quad \textrm{for all}\quad N\geq 1.
$$

At this stage, both claims of part (ii) are a direct consequence of basic Banach space valued martingale theory, see  e.g. \cite[Theorem 3.61, Theorem 1.95]{HvNVW} or \cite[Corollary V.2.4, Corollary III.2.13]{DU} --  the basic fact being that $W^{-\alpha,2}(\R)$ is a separable Hilbert space and thus possesses the Radon-Nikod\'ym property.
\end{proof}

We are then ready to prove the first part of our  main result.

\begin{proof}[Proof of Theorem \ref{th:main}{\rm(i)}] Assume that $\varepsilon >0$ is given. We first apply Proposition \ref{pr:C}(ii) to pick
$N_0$ large enough that for any $N\geq N_0$ we have
$$
\wass_2(\reg{\zeta}_{\rand},\reg{\zeta}_{N,\rand})_{W^{-\alpha,2}(\R)}\leq \varepsilon \quad\textrm{for}\quad N\geq N_0.
$$
Next, according to  Proposition \ref{pr:B} we may  select  $N_1\geq N_0$   and   $T_0\geq 1$ so large that 
$$
\wass_2\big(\reg{\mu}_T\;,\;\reg{\mu}_{T,N_1}\big)_{W^{-\alpha,2}(\R)}\leq \varepsilon \quad\textrm{if}\quad T\geq T_0.
$$
Finally, Proposition \ref{pr:C}(i) allows us to find $T_1\geq T_0$ so that
$$
\wass_2\big(\reg{\zeta}_{N_1,\rand}\;,\;\reg{\mu}_{T,N_1}\big)_{W^{-\alpha,2}(\R)}\leq \varepsilon  \quad\textrm{when}\quad T\geq T_1.
$$
By setting $N=N_1$ in the first inequality, and combining all the three  estimates it follows that
$$
\wass_2\big(\reg{\mu}_T\;,\;\reg{\zeta}_{\rand}\big)_{W^{-\alpha,2}(\R)}\leq 3\varepsilon  \quad\textrm{whenever}\quad T\geq T_1,
$$
and Theorem \ref{th:main}(i) follows.
\end{proof}

\begin{remark}\label{rem:va}
{\rm It's natural to ask whether the class of test functions for which this convergence holds can be enlarged - or if one can regularize our objects in some weaker way than by $f\mapsto \reg{f}$.  We do not consider this question further here, but just point out that the answer to both of these questions  is positive in view of Remark \ref{re:optimality} and the the fact that we have not striven for optimality in this respect in the  proof of Theorem \ref{th:main}(i), including the  error estimates of Theorem \ref{th:twopoint}.}
\hfill $\blacksquare$
\end{remark}

\begin{remark} \label{re:cross alternative}{\rm 
Here we sketch another route to the proof presented in this section that bypasses the explicit estimates for the 2-point function $V_T^{(2)}(x,y)$ presented in Appendix \ref{app:appendixa}. We start by considering a cross term with finite Dirichlet polynomials
and $\zeta$ itself.
\begin{lemma}\label{le:finitepolynomial}
Assume that $A\subset \N$ is a finite subset and denote $\zeta_{A}(s):=\sum_{n\in A}n^{-s}.$ Then, if one defines
\begin{align*}\label{eq:risti2}
V_T^{({A},1)}(x,y)\;=\;&\E\left[\zeta_{A}(1/2+i\omega T+ix)\overline{\zeta(1/2+i\omega T+iy)}\right],\\
V_T^{(A,2)}(x,y)\;=\;&\E\left[\zeta(1/2+i\omega T+ix)\overline{\zeta_{A}(1/2+i\omega T+iy)}\right],
\end{align*}
it holds that
\begin{equation}\label{eq:risti31}
V_T^{({A},1)}(x,y)\;= \; \zeta_{A}(1+i(x-y))\;+\; \mathcal{O}\Big( T^{-1/6}\big(1+|y|^{2/6})\Big)
\end{equation}
and
\begin{equation}\label{eq:risti32}
V_T^{({A},2)}(x,y)\;=\; \; \zeta_{A}(1+i(x-y))\;+\; \mathcal{O}\Big( T^{-1/6}\big(1+|x|^{2/6})\Big),
\end{equation}
where both error estimates may depend on $A$.
\end{lemma}

\begin{proof} We will first prove the latter statement \eqref{eq:risti32}.  As we do not care on the dependence of the error on the finite set $A$ it is enough to treat the case of a singleton, say  $A=\{ m\}$. Let us first assume  that $|x|\leq T/2$, write
\begin{eqnarray}\label{eq:k1}
V_T^{({A},2)}(x,y)= \frac{1}{iT}\int_{1/2}^{1/2+iT}m^{s-1+iy}\zeta(s+ix)ds,
\end{eqnarray}
and use the residue theorem to move the integration to the line $[3/2, 3/2+iT],$ so that 
$$
V_T^{({A},2)}(x,y)=m^{-1-i(x-y)} +T^{-1}\sum_{n=1,\;n\not=m}^\infty m^{1/2+iy}n^{-3/2-ix}\frac{(m/n)^{iT}-1}{\log(m/n)}
+E(x,y,T,{A}).
$$
Here the error term consists of the horizontal integrals and the (possible) residue at $s={1-ix}$ (with obvious modification for $x=0$). The contribution of the residue is $T^{-1}\mathcal{O}(1)$ and by \eqref{eq:zcgrowth} the horizontal terms contribute at most $T^{-1}\mathcal{O}(T^{1/6}+ |x|^{1/6})=\mathcal{O}(T^{-1/6}(1+|x|^{2/6}))$. The infinite sum term is  of order $\mathcal{O}(T^{-1})$ uniformly in $x,y$ (note though that almost all of our error terms depend on $m$). In case $|x|\geq T/2$ we note that size the whole  integral is $\mathcal{O}(T^{1/6}+ |x|^{1/6})= \mathcal{O}(T^{-1/6}(1+ |x|^{2/6}))$. Putting everything together, the estimate \eqref{eq:risti32} follows. In turn, the proof of  \eqref{eq:risti31} is similar: one starts from an analogue of  \eqref{eq:k1} and makes a change of variables $s\to 1-s$ so that  the integration is along the segment $[1/2,1/2-iT]$, and the rest is analogous.
\end{proof}

Let us define $ \langle f,g\rangle_{\z,A} := \sum_{n\in A}n^{-1}{\widehat f(\pii\log (n))\overline{\widehat g(\pii\log (n))}}$.
One then proceeds as in the  proof of the first part of Theorem \ref{th:main} (i) by substituting in the original proof the asymptotics provided by the previous lemma. In the first step one obtains an analogue of Lemma \ref{le:conv2}. In particular, when we define the $W^{-\alpha,2}(\R)$-valued random variables 
$$
\mu_{T,A}(x):=\zeta_{A}(1/2+ix+iT\omega):=\sum_{n\in A}n^{-1/2+ix+iT\omega},
$$ 
we have
\begin{align*}%\label{eq:bound44}
\lim_{T\to\infty}\E|\mu_{T,A}(\reg{f})|^2 \;=\; \|\reg{f}\|_{\z,A}^2\;\lesssim\;\| f\|_{V}^2.
\end{align*}
 and
 \begin{align*}%\label{eq:bound6a}
\lim_{T\to\infty}\E\big(\mu_{T,A}(\reg{f})\overline{\mu_{T}(\reg{f})}+ \mu_{T}(\reg{f})\overline{\mu_{T,A}(\reg{f})}\big) \;=\; 2\|\reg{f}\|_{\z,A}^2\;\leq\; C(A) \| f\|_{V}^2.
\end{align*}
Then the analogue of the second part of Proposition \ref{pr:A} reads
\begin{align*}%\label{eq:bound8a}
\lim_{T\to\infty}\E|\mu_T(\reg{f})- \mu_{T,A}(\reg{f})|^2 \; =\; (\|\reg{f}\|_\z^2-\|\reg{f}\|_{\z,A}^2)\;\leq\; C' \| f\|_{V}^2.\;<\infty,
\end{align*}
where the first statement is now  equality  and $C'$ is independent of $A$.

Especially, given $\varepsilon >0$, by running the Sobolev-estimates as before we deduce easily that for large enough $N\geq 1$ and a large enough, but finite, subset  $A_k:=\{1,\ldots ,k\}\subset \N_N$  we have for large enough $T$ that both
$$
 \wass_2\big(\reg{\mu}_T\;,\;\reg{\mu}_{T,A_k}\big)_{W^{-\alpha,2}(\R)} <\varepsilon
$$
and
$$
 \wass_2\big(\reg{\zeta}_{N,\rand}\;,\;\reg{\mu}_{T,A_k}\big)_{W^{-\alpha,2}(\R)} <\varepsilon,
 $$
 where the latter estimate follows easily by the uniform convergence of the series defining $\mu_{T,A_k}$ on the critical line. The rest can be dealt with as before, in view of Proposition \ref{pr:C}.

One should note that it appears harder to get good error estimates  for the approximation yielded by truncated Euler products by using this alternative argument.
}
\hfill $\blacksquare$
\end{remark}

We next  consider briefly $\zeta_\rand$ as a random  analytic function in the half plane $\{ \sigma >1/2\}$.
First of all, in our convergence statements above we may equally well switch from the factor $(1+x^2)^{-1}$ to $(1/2+ix)^{-2}$. Namely, if we write for $x\in\R$ 
$$
m(x):=  (1/2+ix)^{-2}(1+x^2)
$$
we clearly have that $ \|(\frac{d}{dx})^k m\|_{L^\infty (\R)}\leq C_k$ for all $k\geq 1.$  Thus $m$ is a multiplier in the Sobolev spaces
$W^{u,2}(\R)$, i.e. $\| mf\|_{W^{u,2}(\R)}\leq C\| f\|_{W^{u,2}(\R)}$ for all $f\in W^{u,2}(\R)$. This well-known fact is easily verified  for $u\in \N$,   it follows by interpolation for non-integer $u\geq 0$, and finally by duality for negative $u$. The convergence statement in \eqref{eq:bohr} can thus be rewritten as 
\begin{equation}\label{eq:analyticbohr}
s^{-2}\zeta_\rand(s)\; := \; \lim_{N\to\infty}s^{-2}\zeta_{N,\rand}(s)\quad \textrm{on the line }\quad \{\sigma=1/2\},
\end{equation}
with convergence in $W^{-\alpha,2}(\R)$ for, say $\alpha\in (1/2,1].$ It is classical that the Poisson extension of a (generalized) function  $f\in W^{-\alpha,2}(\{\sigma=1/2\})$ is well-defined and yields a harmonic function in $\{\sigma>1/2\}$ and $f$ is obtained as the distributional boundary limit of its harmonic extension. Moreover, the point evaluations in a compact subset $K\subset\{\sigma>1/2\}$ are uniformly bounded with respect to  the Sobolev-norm on the boundary. Hence, if a sequence of boundary functions $(f_k)$ satisfies $\|f-f_k\|_{W^{-\alpha,2}(\R)}\to 0$ as $k\to\infty,$ then the Poisson extensions of $f_k$ tend to that of $f$ locally uniformly in $\{\sigma>1/2\}$. In our situation the Poisson extension of $s^{-2}\zeta_{N,\rand}(s)$ obviously equals 
$$
s^{-2}\zeta_{N,\rand}(s):= s^{-2}\prod_{k=1}^N\left(\frac{1}{1-p_k^{-s}e^{2\pi i\theta_k}}\right),\quad \sigma>1/2.
$$
Moreover, as $N\to\infty, $ the product in the right hand side obviously  converges  locally uniformly in $\{\sigma >1\}$ to
\begin{equation}\label{eq:analyticzetarand}
s^{-2}\zeta_\rand(s):=  s^{-2}\prod_{k=1}^\infty\left(\frac{1}{1-p_k^{-s}e^{2\pi i\theta_k}}\right).
\end{equation}
All these facts together with the convergence \eqref{eq:analyticbohr} enable us to deduce that, almost surely,  the product in \eqref{eq:analyticzetarand} actually converges locally uniformly in  $\{\sigma >1/2\}$ to a random analytic function, with distributional boundary limit $s^{-2}\zeta_\rand$.
Then of course, the same convergence statement holds for 
$$
\zeta_\rand(s)=  \prod_{k=1}^\infty\left(\frac{1}{1-p_k^{-s}e^{2\pi i\theta_k}}\right) \quad\textrm{in}\quad \sigma>1/2.
$$
A fortiori, as a consequence of Theorem \ref{th:main} we obtain the basic known results of the functional statistical convergence of the zeta function in the open strip $\{ 1/2<\sigma <1\}$ (or in the half plane $\{\sigma >1/2\})$. 

One should note that by the local uniform convergence and Hurwitz's theorem, the randomized zeta function $\zeta_\rand (s)$ is non-zero for  $\sigma >1/2$. This has been noticed before, see e.g. \cite[Theorem 4.6.]{HLS} where  the different realisations of $\zeta (s)$ were obtained as possible vertical limit functions of the Riemann zeta function.

We finish this section with the

\begin{proof}[Proof of Theorem \ref{th:nonmeasure}]  Let us first establish the following auxiliary result
\begin{lemma}\label{le:blowup}
Almost surely one has for almost every $t\in\R$ 
$$
\limsup_{\sigma\to 1/2^{+}}|\zeta_\rand (\sigma+it)|=\infty.
$$
\end{lemma}
\begin{proof}
In order to prove the lemma, we cover $\R$ by the intervals  $[m,m+1)$, and note that by translation invariance it is enough to prove that the blow-up takes place almost surely for almost every $t\in [0,1)$. By  joint measurability, translation invariance, and Fubini's theorem it actually suffices to show that
\begin{equation*}%\label{eq:0blowup}
\limsup_{\sigma\to 1/2^{+}}|\zeta_\rand (\sigma)|=\infty.
\end{equation*}
For that end we note that by the Euler product we may infer 
$$
\log|\zeta_\rand (\sigma)|= \sum_{k=1}^\infty \frac{\cos(2\pi\theta_k)}{p_k^{\sigma}}+ E(\sigma),
$$
where for $\sigma >1/2$
$$
E(\sigma)=\frac{1}{2}\sum_{k=1}^\infty \frac{\cos(4\pi\theta_k)}{p_k^{2\sigma}} +\mathcal{O}(1),
$$
and the $\mathcal{O}(1)$-term is uniform in $\sigma\geq 1/2.$ We now skip ahead slightly and make use of Lemma \ref{le:perus} to argue that since  $\sum_{k=1}^\infty p_k^{-4\sigma}<\infty,$ almost surely the whole error term $E(\sigma)$ is uniformly $\mathcal{O}(1)$. Here one uses the fact that a standard Dirichlet series is bounded on the half-line $[\sigma_0,\infty)$ assuming that it is converges at $\sigma_0.$

It remains to check that almost surely
\begin{equation*}%\label{eq:1blowup}
\limsup_{\sigma\to 1/2^{+}}\sum_{k=1}^\infty \frac{\cos(2\pi\theta_k)}{p_k^{\sigma}}=: \limsup_{\sigma\to 1/2^{+}}g(\sigma)\; =\;\infty.
\end{equation*}
Let us note that for any $\sigma\in (1/2,1)$ we may simply compute
$$
\E |g(\sigma)|^4 = 3\sum_{k\not =j}^\infty p_k^{-2\sigma}p_j^{-2\sigma}\big(\E\cos^2(2\pi\theta_1)\big)^2+\sum_{k=1}^\infty p_k^{-4\sigma}\E\cos^4(2\pi\theta_1)\leq 3 \big(\E |g(\theta)|^2\big)^2,
$$
since $\E\cos^2(2\pi\theta_1)=1/2$ and $\E\cos^4(2\pi\theta_1)=3/8$. The Paley-Zygmund inequality now verifies that
$$
\Prob \left(|g(\sigma)|\geq (\E |g(\sigma)|^2)^{1/2}\right)\geq c_0>0 \quad \textrm{for all}\quad \sigma\in (1/2,1).
$$
Especially, as $\lim_{\sigma\to 1/2^{+}} (\E |g(\sigma)|^2)^{1/2}=\infty$ we gather that $\Prob (\limsup_{\sigma\to 1/2^{+}}|g(\sigma)| =\infty) \geq c_0$, and  the claim follows by symmetry and Kolmogorov's 0-1 law.
\end{proof}

To return to the proof of the original statement, it is thus enough to show that  the Poisson extension of an element $g\in W^{-\alpha,2}(\R)$ has radial boundary limits in a set of positive measure assuming that it coincides with a measure on an interval $I\subset \R$. However, we may then  decompose  $g$ by a suitable cut-off function into the sum $u_1+u_2$, where $u_1$ is the Poisson extension of  a measure supported on $I$, and $u_2$  is a $W^{-\alpha,2}(\R)$-function, supported on $\R\setminus I',$ where $I'\subset I$ is an open subinterval. Then $u_2$ has zero boundary values on $I'$ and by basic harmonic analysis (see e.g. \cite[Theorem 11.24]{R} for a variant on the unit circle), $u_1$ has finite radial limits at almost every point on $I'$.

Lemma \ref{le:blowup} then implies that $\zeta_{\rand}(1/2+it)$ can't be a measure on any open interval.

\end{proof}

\section{A Gaussian approximation for the field: proof of Theorem \ref{th:gaussian_appro}}\label{sec:appro}
The goal of this section is to prove Theorem \ref{th:gaussian_appro} - namely that on any interval $[-A,A]$ we can indeed write $\log\zeta_{N,\rand}(1/2+ix)=\mathcal{G}_N(x)+\mathcal{E}_N(x)$, where $\mathcal{G}_N$ converges to a (complex) log-correlated Gaussian field, and $\mathcal{E}_N$ converges to a smooth function. Since this has no consequence for the proof, we shall simplify notation slightly and replace the interval $[-A,A]$ by $[0,1].$

We shall make repeated use of the  following auxiliary technical result that can also be deduced from known estimates 
due to e.g. Kahane and Kwapien for random series in Banach spaces. However, for the reader's convenience we provide a proof below.
\begin{lemma}\label{le:perus} {\bf (i)}\quad
For $n\geq 1$ consider the random functions $F_n$ on the interval $[0,1]$ defined by the series
$$
F_n(x):=\sum_{k=1}^n A_kf_k(x).
$$
Here the $A_k$:s are i.i.d., centered and possibly complex valued random variables that are either  standard normal variables or they are symmetric and uniformly bounded: $|A_1|\leq C <\infty$ almost surely. The functions $f_k$ are assumed to be continuously differentiable on  $[0,1]$ with
\begin{equation}\label{eq:se0}
\sum_{k=1}^\infty \left(\|f_k\|_{L^\infty[0,1]}^2 + \|f'_k\|_{L^\infty[0,1]}^2 \right)<\infty.
\end{equation}
Then almost surely  the series
\begin{equation*}%\label{eq:se1}
F(x):=\sum_{k=1}^\infty A_kf_k(x),
\end{equation*}
converges uniformly on $[0,1]$ to a continuous limit function $F$. Moreover, one has
\begin{equation}\label{eq:se2}
\E \exp\big(\lambda\, \|F\|_{L^{\infty}[0,1]}\big)<\infty \quad  \textrm{for all}\quad \lambda >0\qquad \textrm{and}
\end{equation} 
\begin{equation}\label{eq:se3}
\E \exp\big(\lambda \sup_{0\leq n'<n}|| F_n- F_{n'}||_{L^{\infty}[0,1]}\big)<\infty \quad  \textrm{for all}\quad \lambda >0.
\end{equation} 

\smallskip

\noindent{\bf (ii)}\quad Assume, in addition, that the functions $f_k$ are smooth and that
\begin{equation*}%\label{eq:se4}
\sum_{k=1}^\infty \|f^{(\ell)}_k\|_{L^\infty[0,1]}^2 <\infty \quad\textrm{for all}\quad \ell\geq 0.
\end{equation*}
Then almost surely $F\in C^\infty[0,1]$ and for all $\ell\geq 1$
\begin{equation*}%\label{eq:se4b}
\E \exp\big(\lambda\, \| F\|_{C^{\ell}[0,1]}\big)<\infty \quad  \textrm{for all}\quad \lambda >0.
\end{equation*} 
Moreover, for every index $\ell\geq 0$ it holds that
\begin{equation*}%\label{eq:se5}
\E \exp\big(\lambda \sup_{0\leq n'<n}|| F_n- F_{n'}||_{C^{\ell}[0,1]}\big)<\infty \quad  \textrm{for all}\quad \lambda >0.
\end{equation*} 
Here we have written $\| f\|_{C^\ell[0,1]}$ for the norm $\sum_{j=0}^\ell \| f^{(j)}\|_{L^\infty[0,1]}$ and $F_0=0$.

\end{lemma}

\begin{proof} (i)\quad We may assume that the $A_k$:s and $f_k$:s are real-valued, since the general case is then obtained  by considering the four series obtained by multiplying the real or imaginary part of $A_k$:s  with the real or imaginary parts of $f_k$:s (note that for \eqref{eq:se2} and \eqref{eq:se3} one applies H\"older's inequality). 
We recall the standard Sobolev embedding
\begin{equation}\label{eq:sobo}
\| f\|_{C[0,1]}\leq 2\|f\|_{W^{1,2}(0,1)}:=2\Big(\int_0^1\big( |f(x)|^2+|f'(x)|^2\big)dx \Big)^{1/2}.
\end{equation}
Since the variables $A_k$ are independent and centred,  we find for $M<N$, 
\begin{align*}
\E ||F_{N}-F_M||_{W^{1,2}(0,1)}^2\leq 2\sum_{k=M+1}^N\left[\|f_k\|_{L^\infty[0,1]}^2 + \|f'_k\|_{L^\infty[0,1]}^2\right].
\end{align*}
Then L\'evy's inequality (see \cite[Lemma 2.3.1]{Kahane2}), applied  here to our  $W^{1,2}(0,1)$-valued symmetric random variables) yields that
\begin{equation}\label{eq:levyineq}
\E \sup_{M\leq r\leq N}\|F_{r}-F_M\|_{W^{1,2}(0,1)}^2\leq  4\sum_{k=M+1}^N
\left[\|f_k\|_{L^\infty[0,1]}^2 + \|f'_k\|_{L^\infty[0,1]}^2\right].
\end{equation}
By our assumption \eqref{eq:se0} we may pick a subsequence $(N_j)$ so that
\begin{equation*}
4\sum_{k=N_j+1}^{N_{j+1}}\left[\|f_k\|_{L^\infty[0,1]}^2 + \|f'_k\|_{L^\infty[0,1]}^2\right]<j^{-6}, 
\end{equation*}
for $j\geq 1$. Borel-Cantelli along with \eqref{eq:levyineq} then ensures that there exists a random threshold $j_0$, which is almost surely finite, such that for $j\geq j_0$, 
\begin{equation*}
\sup_{N_{j}\leq u\leq N_{j+1}}||F_{u}-F_{N_j+1}||_{W^{1,2}(0,1)}\leq j^{-2}
\end{equation*}
for $j\geq j_0$. Thus summing over $j$ implies that the subsequence $F_{N_j}$ almost surely converges absolutely in the space $W^{1,2}(0,1)$, and by \eqref{eq:sobo} also in $C([0,1],\C)$. Also, at the same time the above inequality implies that  the full sequence  $F_N$  converges uniformly to a random continuous function. 

It remains to check the claim about the exponential moments.
For that purpose an application of Azuma's inequality and  \eqref{eq:se0} implies the existence of a finite number $C$ such that for small enough $a>0$
\begin{equation}\label{eq:azuma}
\E e^{a|F(x)|^2}\leq C \qquad \mathrm{and} \qquad \E e^{a|F'(x)|^2}\leq C \quad \mathrm{for\ all\ } x\in[0,1].
\end{equation}
Making use of Jensen's inequality (applied to integration with respect to $x$) and the observation that $t\mapsto e^{at}$ is convex on $[0,\infty)$, we obtain by \eqref{eq:azuma} and Fubini's theorem for small enough $a>0$
\begin{align*}
\E\, e^{a\|F\|^2_{W^{1,2}(0,1)}}\; =\; \E\, e^{a\int_0^1(|F(x)|^2+|F'(x)|^2)dx}
\notag &\leq \E \int_0^1 e^{a (|F(x)|^2+|F'(x)|^2)}dx<\infty.
\end{align*}
L\'evy's inequality then strengthens this for small enough $a>0$  to
\begin{equation*}
\E e^{a\sup_{N\geq 1}||F_N||_{L^\infty[0,1]}^2}<\infty\nonumber.
\end{equation*}
The obtained estimates clearly imply \eqref{eq:se2} and \eqref{eq:se3}.

The case where the variables $A_k$ are i.i.d. standard normal random variables can be proven with exactly the same proof or more directly by considering $F$ as a $W^{1,2}(0,1)$-valued random variable that is well-defined by the assumption \eqref{eq:se0}. Then \eqref{eq:se3} is obtained directly from Fernique's theorem \cite[Theorem 12.7.2]{Kahane2}.

(ii)\quad The statement is a direct consequence of part (i) applied successively to the derivative series defining $F^{(\ell)}(x)$, $\ell\geq 0.$
\end{proof}

We will now start the proper consideration of the truncated randomised Euler products by expanding
$$
\log \zeta_{N,\rand}(1/2+ix)=\sum_{j=1}^\infty\sum_{k=1}^N \frac{1}{j}e^{2\pi ij\theta_k}p_k^{-j(\frac{1}{2}+ix)}
$$ 
and first verifying that the sum over the $j\geq 2$ terms yields  a negligible contribution in the sense that it is almost surely smooth over $x\in[0,1].$
\begin{lemma}\label{le:jgeq2}
Let 
\begin{equation*}
L_N(x)=\sum_{j=2}^\infty \sum_{k=1}^N \frac{1}{j}e^{2\pi ij\theta_k}p_k^{-j(\frac{1}{2}+ix)}.
\end{equation*}
Almost surely, as  $N\to\infty$, $L_N$ converges uniformly on $[0,1]$ to a random $C^\infty$-function $L$. The derivatives $L^{(\ell)}_N$ also converge uniformly, and    for any $\ell\geq 0$
\begin{equation*}%\label{eq:expmom1}
\E e^{\lambda \sup_{N\geq 1}\|L_N(x)\|_{C^{\ell}[0,1]}}<\infty \quad \textrm{for all}\quad \lambda\in \R.
\end{equation*}

\end{lemma}
\begin{proof}
Let us separate $L_N$ into the $j=2$ part and the $j\geq 3$ part. As the series 
\begin{equation*}
\sum_{k=1}^\infty \sum_{j=3}^\infty \frac{1}{j}p_k^{-j/2}(\log p_k)^\ell
\end{equation*}
converges for any $\ell\geq 0$, we see that  the series
\begin{equation*}
\sum_{k=1}^\infty\sum_{j=3}^\infty \frac{1}{j}e^{2\pi ij\theta_k}p_k^{-\frac{j}{2}-ijx}\nonumber
\end{equation*}
together with the arbitrarily many times differentiated series converges absolutely and uniformly to a (deterministically) bounded quantity. Thus this part of the sum certainly satisfies the statements of the lemma. 

In turn, the $\ell$:th derivative of the $k$:th  term of the $j=2$ sum has the deterministic upper bound $p_k^{-1}(\log p_k)^\ell$ and the claims  for the $j=2$ sum follow immediately from Lemma \ref{le:perus}(ii) together with the fact
$
\sum_{k=1}^\infty p_k^{-2}(\log p_k)^{2\ell}<\infty.
$

\end{proof}

To study the main term, i.e. the part of the sum with $j=1$, we will  split the field into a sum over blocks, where within the blocks, the quantities $\log p$ are roughly constant, and perform a Gaussian approximation on each block separately. To make this formal, let $(r_k)_{k=1}^{\infty}$ be a sequence of strictly increasing  positive integers with $r_1=1$ and then $\lbrace p_{r_k},...,p_{r_{k+1}-1}\rbrace$ will be the set of primes appearing in the block we've mentioned. 
We will make some preliminary requirements on the blocks. We assume that  $1<r_{k+1}-r_k\to \infty$ as $k\to\infty$, but on the other hand $r_{k+1}/r_k\to 1$ as $k\to\infty$. Also it is useful to assume that $p_{r_{m+1}-1}/p_{r_m}\leq 2$. for all $m\geq 1.$

We then define the blocks of the field as well as a "freezing approximation", where we approximate $p_k^{-ix}$ by $p_{r_m}^{-ix}$ within a block.

\begin{definition}\label{def:blocks}{\rm
For $(r_m)_{m=1}^\infty$ as above, define for $x\in[0,1]$ and $m\geq 1$:

\begin{equation*}%\label{eq:ym}
Y_m(x)=\sum_{k=r_m}^{r_{m+1}-1}\frac{1}{\sqrt{p_k}}e^{2\pi i\theta_k}p_k^{-ix}.
\end{equation*}
Consider also the approximation to this where the $x$-dependent terms within each block are "frozen":
\begin{align*}%\label{eq:ymtil}
\widetilde{Y}_m(x)&=p_{r_m}^{-ix}\sum_{k=r_m}^{r_{m+1}-1}\frac{1}{\sqrt{p_k}}e^{2\pi i\theta_k}\\
\notag &=:p_{r_m}^{-ix}(C_m+iS_m),
\end{align*}
\noindent where $C_m$ is the real part of the sum, and $S_m$ the imaginary part.}
\hfill $\blacksquare$
\end{definition}

The required Gaussian approximation uses the following fairly standard result. We state it in a slightly more general form than we actually need here, since we expect it might be of use in further study of certain more general non-Gaussian chaos models. Some initial steps in this direction are contained in  \cite{J}.

\begin{proposition}\label{le2}  \quad{\bf (i)}\quad Assume that $d\geq 2$ and $H_j=(H_j^{(1)},\ldots H_j^{(d)})$, $j\in\{1,\ldots, n\}$ are independent and symmetric $\R^d$-valued random variables with $$b_0^{-1}\leq c_j:=d^{-1}{\rm Tr\,}({\rm Cov\,}(H_j))\leq b_0$$ for all   $j\in\{1,\ldots, n\}$, where $b_0>0$.    Assume also that the following uniform exponential bound holds for some $b_1,b_2>0$:
\begin{equation}\label{eq:exp}
\E \exp(b_1|H_k|)\leq b_2\quad\textrm{for all}\quad k=1,\ldots , n.
\end{equation}
Then there is a $d$-dimensional  Gaussian random variable $U$  with 
$$
{\rm Cov\,}(U)=\big(\sum_{j=1}^n c_j\big)^{-1}\big(\sum_{j=1}^n {\rm Cov\,}(H_j) \big),\qquad {\rm Tr \,} ({\rm Cov\,}(U))=d,
$$
and such that  the difference
$$
V:=U-\big(\sum_{j=1}^n c_j\big)^{-1/2}\big(\sum_{j=1}^n H_j \big)
$$
satisfies
\begin{equation}\label{e30}
\E |V|\leq a_1n^{-\beta}.
\end{equation}
Above $\beta =\beta(d)>0$ depends only on the dimension and
$a_1$ on $d, b_0,b_1,b_2$. Moreover, $U$ can be chosen to be measurable with respect to $\sigma (G, H_1,\ldots H_n),$
where $G$ is a d-dimensional standard Gaussian independent of the $H_j$:s. In addition, there is the exponential estimate
\begin{equation}\label{e32}
\E \exp(\lambda |V|)\leq 1+ a_2 e^{a_3 \lambda^2}n^{-a_4}\qquad\textrm{for}\quad 0\leq\lambda\leq a_5n^{1/2},
\end{equation} 
where the constants $a_2,a_3,a_4, a_5>0$ depend only on $b_0, b_1, b_2$ and the dimension $d.$ 

In the case where the variables $H_k$ are uniformly bounded, say $|H_k|\leq b_3$ for all $k,$ then \eqref{e32}  holds true for all $\lambda >0$, where now the constants $a_2,a_3,a_4$ may   also depend on $b_3, $ and there are   constants $a_6,a_7,a_8>0$ that  depend only on $b_0, b_1,b_2,b_3,d$ so that
\begin{equation*}%\label{e31}
\E \exp(a_6|V|^2)\leq 1+a_7n^{-a_8},
\end{equation*}

\smallskip

\noindent {\bf (ii)}\quad If we assume that $Cov(H_j)=c_jd^{-1}I$, where $I$ the the $d\times d$ identity matrix, and the dimension $d\geq 1$ is arbitrary, then the conclusion \eqref{e30} can be strengthened to 
\begin{equation}\label{e30'}
\E |V|\leq a_1\log(n)^{d+1}n^{-1/2}.
\end{equation}

\end{proposition}
\noindent We will postpone the proof of this result to Appendix \ref{app:appendixb} since its ingredients are basically well-known, however the above formulation is tailored for our purposes.

Our aim is to apply Proposition \ref{le2} to approximate $(C_m,S_m)$ by a $\R^{2}$-valued Gaussian random variable. To do this, we need to scale things a bit differently. Define the following sequence of $\R^{2}$-valued random variables (so in the setting of Proposition \ref{le2}, $d=2$)
\begin{equation*}%\label{eq:hj}
H_{j,m}=\left(\frac{\sqrt{p_{r_{m+1}-1}}}{\sqrt{p_{r_m-1+j}}}\cos (2\pi\theta_{p_{r_m-1+j}}),\frac{\sqrt{p_{r_{m+1}-1}}}{\sqrt{p_{r_m-1+j}}}\sin (2\pi\theta_{p_{r_m-1+j}})\right),\quad j=1,\ldots, r_{m+1}-r_m.
\end{equation*}
We then  have 
\begin{equation*}
|H_{j,m}|^{2}\leq \frac{p_{r_{m+1}-1}}{p_{r_m}}\leq 2
\end{equation*}
and
\begin{equation*}
Cov(H_{j,m})=\frac{1}{2}\frac{p_{r_{m+1}-1}}{p_{r_m-1+j}} I=c_{j,m}I
\end{equation*}
where $1/2\leq c_{j,m}\leq 1$. In this notation, we have 
\begin{align*}
(C_m,S_m)&=\frac{1}{\sqrt{p_{r_{m+1}-1}}}\sum_{j=1}^{r_{m+1}-r_m}H_{j,m}=
\frac{b_m}{\sqrt{\sum_{j=1}^{r_{m+1}-r_m}c_{j,m}}}\sum_{j=1}^{r_{m+1}-r_m}H_{j,m},
\end{align*}
where
\begin{equation*}
b_m=\sqrt{\frac{1}{2}\sum_{j=1}^{r_{m+1}-r_m}\frac{1}{p_{r_m-1+j}}}.
\end{equation*}
Proposition \ref{le2} (ii) thus yields a sequence of independent standard two-dimensional normal variables $(V^{(1)}_m,V^{(2)}_m)$, $m=1,2,\ldots$, so that the distance between $(C_m,S_m)$ and 
$b_m(V_m^{(1)},V_m^{(2)})$
is controlled as in the statement of Proposition \ref{le2} (ii).

We may assume that our probability space is large enough for us to write for each $m\geq 1$ and $i\in\{1,2\}$
$$
\sqrt{\frac{1}{2}\sum_{j=1}^{r_{m+1}-r_m}\frac{1}{p_{r_m-1+j}}}V_m^{(i)}=\sum_{j=r_m}^{r_{m+1}-1} \frac{1}{\sqrt{2p_j}}W^{(i)}_j,
$$
where the $W^{(i)}_j$:s are independent standard normal random variables for all $j\geq 1$ and $i\in\{1,2\}$. Finally we can write down our Gaussian approximation to the field, its blocks, and frozen versions of the blocks. 

\begin{definition}\label{de:gaussian}{\rm
Let $(W_k^{(j)})_{k\geq 1,j\in\lbrace 1,2\rbrace}$ be the i.i.d. standard Gaussians constructed above. For any $N\geq1$ and $x\in[0,1]$ the Gaussian approximation of the "$j=1$ part" of $\log\zeta_{N,\rand}$ is given by the Gaussian field

\begin{equation*}%\label{eq:gn}
\mathcal{G}_N(x):=\sum_{k=1}^N\frac{1}{\sqrt{2p_k}}\left(W^{(1)}_k+iW_k^{(2)}\right)p_k^{-ix}.
\end{equation*}

Moreover, we define the blocks of $\mathcal{G}_N$ as 

\begin{equation*}%\label{eq:ygm}
Z_m(x)=\sum_{k=r_m}^{r_{m+1}-1}\frac{1}{\sqrt{2p_k}}\left(W^{(1)}_k+iW_k^{(2)}\right)p_k^{-ix}
\end{equation*}

\noindent and a "frozen" version of the block as 

\begin{equation*}
\widetilde Z_m(x)=b_m p_{r_m}^{-ix}(V_m^{(1)}+i V_m^{(2)}).
\end{equation*}}
\hfill $\blacksquare$
\end{definition}

We then the  start the analysis of the error produced by our Gaussian approximation. This is first performed only for sums over full blocks. We introduce some notation for the errors. 
Let us call the error we make by approximating our frozen field by the frozen Gaussian one by 
\begin{equation*}
\widetilde E_{1,n}(x):=\sum_{m=1}^{n}(\widetilde{Y}_m(x)-\widetilde Z_m(x)),\qquad x\in [0,1].
\end{equation*}
In a similar vein, the error obtained from the freezing procedure is denoted by  
\begin{equation*}
\widetilde E_{2,n}(x):=\sum_{m=1}^{n}\big(Y_m(x)-\widetilde{Y}_m(x)+\widetilde Z_m(x)-Z_m(x)\big),\qquad x\in [0,1].
\end{equation*}
whence the total error  can be written as
\begin{equation*}
\widetilde E_{n}(x):=\widetilde E_{1,n}(x)+ \widetilde E_{2,n}(x).
\end{equation*}

We study first the size of the error $\widetilde{E}_{1,n}$.
\begin{lemma}\label{le:E1}
Assume $($in addition to our previous constraints on $(r_m))$ that 
\begin{equation}\label{eq:konsta1}
\sum_{m=1}^{\infty}(r_{m+1}-r_m)^{-a_4}<\infty \quad\textrm{and}\quad  \sup_{m\geq 1}\frac{r_{m+1}-r_m}{r_m}(\log r_m)^\ell <\infty
\end{equation}
for all $\ell \geq 1,$
where $a_4$ is the constant from Proposition \ref{le2}.
Then, almost surely  there exists a $C^\infty$-smooth limit function
\begin{equation*}%\label{eq:171}
\widetilde E_1(x):=\lim_{n\to\infty} \widetilde E_{1,n}(x),
\end{equation*}
where the convergence is in the norm $\|\cdot\|_{C^{\ell}[0,1]}$ for any $\ell\geq 0.$ Moreover, one has 
\begin{equation}\label{eq:172}
\E\exp\big(\lambda \sup_{0\leq n'<n}\|\widetilde E_{1,n}-\widetilde E_{1,n'}\|_{C^{\ell}[0,1]}\big)<\infty\quad \textrm{for all}\quad \lambda >0,
\end{equation}
where one applies the convention $\widetilde E_{1,0}(x) \equiv 0$, and, in particular,
\begin{equation}\label{eq:173}
\E\exp\big(\lambda ||\widetilde E_1||_{C^{\ell}[0,1]}\big)<\infty\quad \textrm{for all}\quad \lambda >0
\end{equation}
for any $\ell\geq 0.$
\end{lemma}
\begin{proof}
To prove convergence of $\widetilde{E}_{1,m}$, we fix an integer $\ell \geq 0$  and observe that
\begin{align*}
\|\widetilde{E}_{1,m-1}-\widetilde{E}_{1,m}\|_{C^{\ell}[0,1]}&\leq |(C_m,S_m)-b_m(V_m^{(1)},V_m^{(2)})|(1+\log r_m)^\ell\\
\notag &=b_m\left|\frac{1}{\sqrt{\sum_{j=1}^{r_{m+1}-r_m}c_{j,m}}}\sum_{j=1}^{r_{m+1}-r_m}H_{j,m}-(V_m^{(1)},V_m^{(2)})\right|(1+\log r_m)^\ell.
\end{align*}

We then recall that we assumed that $r_{m+1}/r_m\to 1$ so we see from the prime number theorem (and a crude estimate on the sum) that for some constant $C_1>0$
\begin{equation*}
b_m^2\leq C_1 \frac{r_{m+1}-r_m}{r_m\max(\log r_m,1)}.
\end{equation*}
\noindent Thus by the second condition in the assumption \eqref{eq:konsta1} 
\begin{equation*}
b_m^2(1+\log r_m)^{2\ell}\leq C_2.
\end{equation*}
 Proposition \ref{le2} (more precisely \eqref{e32}) then implies that we have, for some constants $C,\widetilde{C}$,
\begin{align*}
\E||\widetilde{E}_{1,m-1}-\widetilde{E}_{1,m}\|_{C^{\ell}[0,1]}&\leq C\E( e^{|(\sum_{j=1}^{r_{m+1}-r_m}c_{j,m})^{-1/2}\sum_{j=1}^{r_{m+1}-r_m}H_{j,m}-(V_m^{(1)},V_m^{(2)})|}-1)\\
&\leq \widetilde{C} (r_{m+1}-r_m)^{-a_4}.
\end{align*}
Thus by our assumption on $(r_m)$, the series
$
\widetilde{E}_1=\sum_{m=1}^\infty (\widetilde{E}_{1,m}-\widetilde{E}_{1,m-1})
$
almost surely converges absolutely in ${C^{\ell}[0,1]}$.

We next use the crude estimate
\begin{align*}
\sup_{0\leq n'<n} \|\widetilde E_{1,n}-\widetilde E_{1,n'}\|_{C^{\ell}[0,1]}
\leq \sum_{m=1}^{\infty}|(C_m,S_m)-b_m(V_m^{(1)},V_m^{(2)})|(1+\log r_m)^{\ell},
\end{align*}
so that by independence and \eqref{e32}
\begin{align*}
\E\exp(\lambda \sup_{0\leq n'<n}\|\widetilde E_{1,n}-\widetilde E_{1,n'}\|_{C^{\ell}[0,1]})&\leq \prod_{m=1}^{\infty}\E e^{\lambda |(C_m,S_m)-b_m(V_m^{(1)},V_m^{(2)})|(1+\log r_m)^{\ell}}\\
\notag &\leq \prod_{m=1}^{\infty}\left(1+a_2 e^{a_3\lambda^{2}b_m^{2}(1+\log r_m)^{2\ell}}(r_{m+1}-r_m)^{-a_4}\right)
\end{align*}
As we saw that $b_m(1+\log r_m)^{2\ell}$ is bounded, we find for some constant $C$ (depending on $\lambda$) that 
\begin{align*}
\E\exp\big(\lambda \sup_{0\leq n'<n}\|\widetilde E_{1,n}-\widetilde E_{1,n'}\|_{C^{\ell}[0,1]}\big)&\leq \prod_{m=1}^{\infty}\left(1+C(r_{m+1}-r_m)^{-a_4}\right)\\
\notag &\leq e^{C\sum_{m=1}^{\infty}(r_{m+1}-r_m)^{-a_4}},
\end{align*}
and  \eqref{eq:172} follows. Finally, \eqref{eq:173} is an obvious  consequence of \eqref{eq:172}.
\end{proof}
Let us then estimate the error due to the  freezing procedure.
\begin{lemma}\label{le:E2}
Assume that the sequence $(r_m)$ is chosen so that 
\begin{equation}\label{eq:rcond}
\sum_{m=1}^{\infty}\frac{(r_{m+1}-r_m)(p_{r_{m+1}}-p_{r_m})^{2}\log^\ell r_m}{r_m^{3}}<\infty
\end{equation}
for any $\ell \geq 1.$
Then, almost surely  there exists a $C^\infty$-smooth limit function
\begin{equation*}%\label{eq:181}
\widetilde E_2(x):=\lim_{n\to\infty} \widetilde E_{2,n}(x),
\end{equation*}
where the convergence is in the sup-norm over $[0,1]$.  Moreover,  for any $\ell \geq 1$ we have both
\begin{equation*}%\label{eq:184}
\E \exp\big(\lambda \|\widetilde E_{2}\|_{C^{\ell}[0,1]}\big)<\infty \quad  \textrm{for all}\quad \lambda >0\qquad \textrm{and}
\end{equation*}
\begin{equation*}%\label{eq:185}
\E \exp\big(\lambda \sup_{0\leq n'<n}||\widetilde E_{2,n}-\widetilde E_{2,n'}||_{C^{\ell}[0,1]}\big)<\infty \quad  \textrm{for all}\quad \lambda >0.
\end{equation*}
\end{lemma}
\begin{proof}
The proof is again based on Lemma \ref{le:perus}. It follows immediately from the definitions that 
\begin{align*}%\label{eq:190}
Y_m(x)&- \widetilde{Y}_m(x)+\widetilde{Z}_m(x)-Z_m(x)\\
\notag &=\sum_{k=r_m}^{r_{m+1}-1}\frac{1}{\sqrt{p_k}}\left(p_k^{-ix}-p_{r_m}^{-ix}\right)\left(e^{2\pi i\theta_k}-\frac{1}{\sqrt{2}}\left[W_k^{(1)}+iW_k^{(2)}\right]\right)\\
\notag &:=\sum_{k=r_m}^{r_{m+1}-1}f_k(x)\left(e^{2\pi i\theta_k}-\frac{1}{\sqrt{2}}\left[W_k^{(1)}+iW_k^{(2)}\right]\right),
\end{align*}
where $f_k(x):= \frac{1}{\sqrt{p_k}}\left(p_k^{-ix}-p_{r_m}^{-ix}\right).$
Given any integer $\ell\geq 0$ we may use the properties of the sequence $(r_m)$ and the 1-Lipschitz property of $u\to e^{iu}$ to  estimate for any $x\in [0,1]$ and $r_m\leq k\leq r_{m+1}-1$
\begin{align*}
&\; p_k^{1/2}|f^{(\ell)}_k(x)\;|=\; \big| (-i)^\ell\big( p_k^{-ix}\log^{\ell}p_k-p_{r_m}^{-ix}\log^{\ell}p_{r_m}\big)\big|\\
\leq&\; |p_k^{-ix}-p_{r_m}^{-ix}|\log^{\ell}p_k+ |\log^{\ell}p_k -\log^{\ell}p_{r_m}|
\lesssim \; \log^{\ell}(p_{r_m})(\log p_k-\log p_{r_m})\\
\lesssim &\; \log^{\ell}(p_{r_m})\frac{p_k-p_{r_m}}{p_{r_m}}.
\end{align*}
Hence
\begin{align*}
\sum_{k=1}^\infty \|f_k^{(\ell)}\|_{L^\infty [0,1]}^2&\lesssim\sum_{m=1}^\infty \frac{r_{m+1}-r_m}{p_{r_m}}\left(\frac{p_{r_{m+1}}-p_{r_m}}{p_{r_m}}\right)^2\log^{2\ell}(p_{r_m})\\
&\lesssim\sum_{m=1}^\infty (r_{m+1}-r_m)\frac{(p_{r_{m+1}}-p_{r_m})^2}{p^3_{r_m}}\log^{2\ell}(p_{r_m})<\infty
%\\ &< \infty
\end{align*}
by our assumption. The claim now follows by a two-fold application of Lemma \ref{le:perus}(ii).
\end{proof}

We next combine the error estimates proven so far and make the final choice for the subsequence $(r_m)$. For that purpose we need the following well-known lemma, whose proof we include for the reader's convenience.
\begin{lemma}\label{le:li-inverse}  For large enough $n$ it holds that
\begin{equation*}%\label{eq:prime}
-ne^{-\sqrt{\log n}}\lesssim p_n- \mathrm{Li}^{-1}(n)\lesssim ne^{-\sqrt{\log n}}.\nonumber
\end{equation*}
\end{lemma}
\begin{proof} We note first that the inverse $ \mathrm{Li}^{-1}$ is convex since $ \mathrm{Li}$ itself is concave. Furthermore,
we have $( \mathrm{Li}^{-1})'(x)=\log( \mathrm{Li}^{-1}(x))\leq \log(2x\log(x))\leq 2\log(x)$ for large enough $x.$ Hence, as a suitable  quantitative version of the prime number theorem verifies that for any $c\geq 1$ there is the error estimate
$|\pi(x)- \mathrm{Li}(x)|=\mathcal{O}\big(x\exp(-c\sqrt{\log x})\big)$, we have $n=\pi(p_n)\leq  \mathrm{Li}(p_n)+ne^{-2\sqrt{\log n}}$. In particular,  for large enough $n$
\begin{align*}
p_n&\geq \mathrm{Li}^{-1}(n- ne^{-2\sqrt{\log n}})\;
 \geq\;  \mathrm{Li}^{-1}(n)-ne^{-2\sqrt{\log n}}( \mathrm{Li}^{-1})'(n)
 \,\geq \;\mathrm{Li}^{-1}(n)-ne^{-\sqrt{\log n}}.
\end{align*}
The proof of the other direction is analogous.
\end{proof}

\begin{proposition}\label{pr2} Choose  $($for the rest of the paper$)$ $r_m=\lfloor \exp (3\log^2m)\rfloor$. Then  the combined error $\widetilde E_{n}(x)=\widetilde{E}_{n,1}(x)+
\widetilde  E_{n,2}(x)$ a.s. converges for any $\ell\geq 0 $ in $C^{\ell}[0,1]$ to a  $C^\infty$-smooth limit 
$$
E(x):=\lim_{n\to\infty} (\widetilde E_{n,1}(x)+
\widetilde  E_{n,2}(x)).
$$
Moreover, for all $\lambda>0$ and $\ell \geq 0$ 
\begin{equation}\label{eq:204}
\E \exp\big(\lambda \|  E\|_{C^{\ell}[0,1]}\big)<\infty 
\end{equation}

\noindent and 

\begin{equation*}%\label{eq:esupbound}
\E \exp\big(\lambda \sup_{0\leq n'<n}\| \widetilde E_{n}-\widetilde E_{n'}\|_{C^{\ell}[0,1]}\big)<\infty.
\end{equation*}
 \end{proposition}
\begin{proof}
We first recall the condition of Lemma \ref{le:E1} - namely that the first error term converges as soon as
 \begin{equation}\label{e15}
\sum_{m=1}^{\infty}(r_{m+1}-r_m)^{-a_4}<\infty \quad\textrm{and}\quad  \sup_{m\geq 1}\frac{r_{m+1}-r_m}{r_m}(\log r_m)^\ell <\infty
\end{equation}
for any $\ell \geq 0.$
Lemma \ref{le:li-inverse}  yields for our sequences that $p_{r_{m+1}}-p_{r_m}\lesssim (r_{m+1}-r_m)\log r_m + r_me^{-\sqrt{\log r_m}}.$ By plugging this into condition
\eqref{eq:rcond} we see that a sufficient condition to apply Lemma \ref{le:E2} in order to control the second error term is given by the pair of conditions
\begin{equation}\label{eq:205}
 \sum_{m=1}^\infty\bigg(\frac{r_{m+1}-r_m}{r_m}\bigg)^3\log^{L} (r_m)<\infty\quad\textrm{and}\quad
  \sum_{m=1}^\infty e^{-2\sqrt{\log r_m}}\log^{L} (r_m)<\infty
\end{equation} 
for all $L\geq 1$.
Finally, it remains to observe that the choice $r_m=\lfloor \exp (3 \log^2m)\rfloor$ satisfies both \eqref{e15} and \eqref{eq:205}, and satisfies the initial properties postulated for $(r_m)$ after the proof of Lemma \ref{le:jgeq2}.
\end{proof}

To complete the approximation procedure, we finally verify that the fields $\mathcal{G}_N$ are good approximations also for indices
$N$ inside the intervals $r_m\leq N<r_{m+1}$.
\begin{proposition}\label{prop:g-approx} Denote the total error of the Gaussian approximation by setting
\begin{equation*}%\label{eq:301}
E_N(x):=\sum_{k=1}^N \frac{1}{\sqrt{p_k}}p_k^{-ix}e^{2\pi i\theta_k}-\mathcal{G}_N(x)\quad \textrm{for}\quad N\geq 1\quad\textrm{and}\quad x\in[0,1].
\end{equation*}
Then, almost surely,  $E_{N}(x)$ converges in $C^{\ell}[0,1]$ to a $C^\infty$-smooth limit function
$$
E(x):=\lim_{N\to\infty} E_{N}(x),
$$
where the obtained limit is of course the same as in Proposition {\rm  \ref{pr2}}.
Moreover, for all $\lambda>0$ and any $\ell\geq 0$
\begin{equation*}%\label{eq:204v}
\E \exp\big(\lambda \|  E\|_{C^{\ell}[0,1]}\big)<\infty 
\end{equation*}

\noindent and 

\begin{equation}\label{eq:esupnbound}
\E \exp\big(\lambda \sup_{N\geq 1}\|E_{N}\|_{C^{\ell}[0,1]}\big)<\infty.
\end{equation}

\end{proposition}
\begin{proof}
After Proposition \ref{pr2} it is enough to show that any given partial sum of the original series is in fact well approximated by the sum of the blocks below it, and that a similar statement holds also true for the Gaussian approximation series.  Let us fix $m\geq 1$ and recall our notation
$$
Y_m(x)=\sum_{k=r_m}^{r_{m+1}-1}\frac{1}{\sqrt{p_k}}p_k^{-ix}e^{2\pi i\theta_k}
\;=:\;\sum_{k=r_m}^{r_{m+1}-1}A_k(x), 
$$
which is just the partial sum of our original field corresponding to the $m$:th block.  Observing first that
\begin{align*}
\sum_{k=r_m}^{r_{m+1}-1}\frac{\log^{2\ell} p_k}{p_k}&\lesssim \sum_{k=\lfloor e^{\log^2m}\rfloor}^{\lfloor e^{\log^2(m+1)}\rfloor}\frac{\log^{2\ell} k}{k\log k}
\lesssim  \log m^{4\ell-2}\big( \log^2(m+1)-\log^2m \big)
\lesssim \log m^{4\ell-1}m^{-1} \\
\notag &\lesssim m^{-1/2},
\end{align*}
 Azuma's inequality yields 
$$
\Prob (|Y^{(\ell)}_m(x)|\geq\lambda)\lesssim \exp \Big(-c'\lambda^2\Big(\sum_{k=r_m}^{r_{m+1}-1}\frac{\log^{2\ell} p_k}{p_k}\Big)^{-1}\Big)
\lesssim \exp \big(-c'\lambda^2m^{1/2}\big).
$$
In particular, we  obtain that for some constants $c'', C$ that work for all $x\in [0,1]$ we have 
$$
\E \exp(c''m^{1/2}|Y^{(\ell)}_m(x)|^2)\leq C.
$$
This holds true for all $\ell\geq 0.$
As at the end of the proof of Lemma \ref{le:perus}, we deduce that $\E \exp(c'''m^{1/2}\|Y_m(x)\|^2_{C^{\ell}[0,1]})\leq C$, and again   L\'evy's inequality enables us to gather that
\begin{equation*}%\label{eq221}
\Prob \big(\max_{r_m\leq u\leq r_{m+1}-1}\|\sum_{k=r_m}^{u}A_k\|_{C^{\ell}[0,1]}>\lambda\big) \lesssim \exp(-c'''m^{1/2}\lambda^2).
\end{equation*}
Summing over $m$  yields for $\lambda \geq 1$
\begin{align}\label{eq222}
\Prob \big(\sup_{m\geq 1}&\max_{r_m\leq u\leq r_{m+1}-1}\|\sum_{k=r_m}^{u}A_k\|_{C^{\ell}[0,1]}>\lambda\big)\lesssim \sum_{m=1}^\infty \exp(-c'''m^{1/2}\lambda^2)\lesssim \exp(-c''''\lambda^2).
\end{align}
Exactly  the same proof where Azuma is replaced by elementary estimates for Gaussian variables  yields the corresponding estimate  for our Gaussian approximation fields.  An easy Borel-Cantelli argument that uses estimates like \eqref{eq222} in combination with Proposition \ref{pr2}  then shows the existence of the uniform limit $E(x)=\lim_{N\to\infty}E_N(x).$  Finally, combining  \eqref{eq222} with \eqref {eq:204}  yields  \eqref{eq:esupnbound}.
Together with our previous considerations this concludes the proof of the proposition.
\end{proof}

Finally, putting things together we obtain

\begin{proof}[Proof of Theorem \ref{th:gaussian_appro}]
Noting that $\log \zeta_{N,\rand}=\mathcal{G}_N+E_N+L_N$ and writing $\mathcal{E}_N=E_N+L_N$, we see that Theorem \ref{th:gaussian_appro} follows by combining Lemma \ref{le:jgeq2} with Proposition \ref{prop:g-approx}.
\end{proof}

\section{The relationship to complex Gaussian multiplicative chaos: Proof of Theorem \ref{th:main}({\rm ii})}\label{sec:compconv}

In this section we prove the second part of our main result which states  that $\zeta_\rand$ can be expressed as a product of a complex Gaussian multiplicative chaos distribution and a smooth function with good regularity properties. 
We will first start by proving the existence of the complex Gaussian chaos needed. Recall from Section \ref{se:convergence} that we write $\reg{f}(x)$ for $(1+x^2)^{-1}f(x)$. 

\begin{lemma}\label{le:nu} Denote 
$$
\nu_N(x):=\prod_{j=1}^N e^{\frac{1}{\sqrt{2 p_j}}p_j^{-ix}(W_j^{(1)}+iW_j^{(2)})}.
$$
For any $\alpha >1/2$ the  sequence $\big(\reg{\nu}_N\big)_{N\geq 1}$ is an $L^2$-bounded  $W^{-\alpha,2}(\R)$-valued martingale, and consequently it converges almost surely to a $W^{-\alpha,2}(\R)$-valued random variable which we write as 
$$
\reg{\nu}:=\lim_{N\to\infty}\reg{\nu}_N.
$$
\end{lemma}
\begin{proof} By independence,  for any $a\in\C$
$$
\E e^{a(W_j^{(1)}+iW_j^{(2)})}= e^{\frac{1}{2}(a^2+(ia)^2)} =1
$$
 and we infer that $\big(\reg{\nu}_N\big)_{N\geq 1}$ is a martingale sequence taking values in $L^2\subset W^{-\alpha,2}(\R)$.
Assume first that $\int |g(x)|^2(1+x^2)dx<\infty$ and write $Z_j:=2^{-1/2}(W_j^{(1)}+iW_j^{(2)})$ so that $(Z_j)_{j\geq 1}$ is an i.i.d. sequence of standard complex Gaussians.  By using independence, the fact that $\E e^{\alpha Z_j+\beta \overline{Z_j}}=e^{\alpha\beta}$, and the series expansion of the exponential function, we easily compute
\begin{align*}
\E\left| \nu_N(g)\right|^2 =&\int_{\R^2} g(x)\overline{g(y)}\exp\Big(\sum_{j=1}^N p_j^{-1-i(x-y)}\Big)
=\sum_{n\in \N_N}\frac{1}{\alpha_1(n)!\ldots \alpha_{N}(n)!}\frac{|\widehat g(\pii \log(n))|^2}{n}\\
\leq& \|g\|_\z^2
\;\lesssim \;  \int_\R |g(x)|^2(1+x^2)dx,
\end{align*}
where the quantities $\alpha_j(n)$ were defined in the proof of Proposition \ref{pr:C} and  the last inequality comes from Lemma \ref{le:embed}. We recall the notation $e_\xi(x)= e^{-2\pi i\xi x}$ and $\reg{e}_\xi(x)= (1+x^2)^{-1}e^{-2\pi i\xi x}$ from Proposition \ref{pr:B}. Now substituting $\reg{e}_\xi$ in place of $g$, multiplying by 
$(1+\xi^2)^{-\alpha}$ and integrating over $\R$ we gather that
$$
\E \| \reg{\nu}_N\|_{W^{-\alpha,2}(\R)}^2 \;\leq \; C,
$$
where $C$ does not depend on $N$, and we are done.
\end{proof}

One should observe that the martingale considered in the above proof is non-trivial, and hence the limit random variable $\nu$ is  also non-trivial (i,e, it does not reduce to a deterministic constant). We then start the proof of Theorem \ref{th:main}(ii) by fixing  a compactly supported test function $f\in C_0^\infty(-A,A)$ and observing that we have the equality
\begin{equation*}
\zeta_{N,\rand}(f)=\nu_N(e^{\mathcal{E}_N}f)\notag ,
\end{equation*}
or, writing $g(x):= (1+x^2)f(x)$  this becomes
\begin{equation*}
\reg{\zeta}_{N,\rand}(g)=\reg{\nu}_N(e^{\mathcal{E}_N}g).
\end{equation*}
Here, almost surely $\reg{\zeta}_{N,\rand}$ converges to $\reg{\zeta}_\rand$ and $\reg{\nu}_N$ to $\reg{\nu}$ in $W^{-\alpha,2}(\R)$. Moreover,
$e^{\mathcal{E}_N}g\to e^{\mathcal{E}}g$ in $C^\infty$, with  supports  contained in $(A,A).$ We may thus take the limit\footnote{Here one simply chooses e.g. $\alpha=1$ and notes that one easily checks that the map $(f,g)\mapsto fg$ is continuous map from $C^1(\R)\times W^{1,2}(\R) \to W^{1,2}(\R)$ when $C^1(\R)$ is normed by $\|f\|_{C^1(\R)}=\| f\|_\infty+\|f'\|_\infty$. By duality, it follows that  in the same map is continuous $C^1(\R)\times W^{-1,2}(\R) \to W^{-1,2}(\R)$.} in the previous equality and obtain the almost sure equality
$$
\reg{\zeta}_{\rand}(g)=\reg{\nu}(e^{\mathcal{E}}g)=(e^{\mathcal{E}}\reg{\nu})(g).
$$
A fortiori, since $A$ was arbitrary and this holds almost surely for a countable dense subset of $f$:s in $W^{\alpha,2}(\R)$, we see that almost surely $\zeta_\rand=e^{\mathcal{E}}\nu$ (either as tempered distributions or with the interpretation that $\reg{\zeta}=e^\mathcal{E}\reg{\nu}$ as elements of $W^{-\alpha,2}(\R)$), and this completes the proof  Theorem \ref{th:main}(ii).

\section{The mesoscopic limit -- Proof of Theorem \ref{th:meso1} and Theorem \ref{th:meso2}}\label{sec:meso}

In the present section we verify our statements about the mesoscopic behavior of the zeta function. We start with some definitions and technical lemmata. Fix a two-sided complex Brownian motion $u\mapsto B^\C_u$ for $u\in \R$. More precisely, this means that $B^\C_u:=2^{-1/2}(B^{(1)}_u+iB^{(2)}_u),$ where $B^{(1)}_u,B^{(2)}_u$ are standard independent two-sided Brownian motions. For $h\in L^2(\R)$ we define a translation invariant process $G[h]$ on $\R$ by setting
$$
G[h](x):= \int_\R e^{-2\pi ixu}h(u)dB^\C_u.
$$
For a fixed $h\in L^2$, the covariance of this process is $\E G[h](x)\overline{G[h](y)}=\widehat{|h|^2}(x-y)$, while $\E G[h](x)G[h](y)=0$ for all $x,y\in \R.$
In order to make sure that $G[h]$ defines random continuous functions of the variable $x$ we shall, unless otherwise stated, assume that there is an $\varepsilon >0$ and a $C>0$ such that 
\begin{equation}\label{eq:htail}
0\leq |h(u)|\leq C(1+|u|)^{-1/2-\varepsilon}.
\end{equation}
This makes sure that $\widehat {|h|^2}$ is H\"older-continuous and hence by classical theory, $G[h]$ has a modification whose realizations are almost surely continuous in $x$. Moreover, the sup-norm over any finite interval has the standard (double-)exponential estimates. In particular, for any $f\in L^2(\R)$ with compact support we may safely compute
\begin{align}\label{eq:exp-lasku}
\E \big| \int_\R f(x) e^{G[h](x)}dx\big|^2\; &=\;\int_{\R^2}f(x)\overline{f(y)}\exp\Big(\frac{1}{2}\E\Big(G[h](x)+\overline{G[h](y)}\Big)^2\Big)dxdy \\
&=\;
\int_{\R^2}f(x)\overline{f(y)}\exp\big(\widehat{|h|^2}(x-y)\big)dxdy.\notag
\end{align}
\begin{lemma}\label{le:meso1} Assume that $h,h_j$ $($where $j=1,2,\ldots)$ satisfy \eqref{eq:htail} and for almost every $u\in\R$ one has
$\liminf_{j\to\infty} |h_j(u)|\geq |h(u)|.$ Then for any $f\in L^2(\R)$ with compact support it holds that
$$
\E \Big| \int_\R f(x) e^{G[h](x)}dx\Big|^2\leq \liminf_{j\to\infty}\E \Big| \int_\R f(x) e^{G[h_j](x)}dx\Big|^2
$$
\end{lemma}
\begin{proof} 
We start by  observing that ${\mathcal F}\big((\widehat{|h|^2})^n\big)=\big(|h|^2(-\cdot)\big)^{*n}=\big(|h|^2\big)^{*n}(-\cdot)$. Hence we apply \eqref{eq:exp-lasku} to compute
\begin{align*}
&\E \big| \int_\R f(x) e^{G[h](x)}dx\big|^2
=\sum_{n=0}^\infty\frac{1}{n!}\int_{\R^2}f(x)\overline{f(y)}\big(\widehat{|h|^2}(x-y)\big)^ndxdy\\
&=\sum_{n=0}^\infty\frac{1}{n!}\int_{\R}|\widehat f(\xi)|^2\big(|h|^2\big)^{*n}(\xi)d\xi\\
&= \sum_{n=0}^\infty\frac{1}{n!}\int_{\R^n}|\widehat f(\xi_1+\xi_2+\ldots +\xi_n)|^2|h|^2(\xi_1)\cdots|h|^2(\xi_n)d\xi_1\cdots d\xi_n. 
\end{align*}
Applying the above identity to the functions $h_j$ as well, the claim follows immediately from Fatou's lemma.
\end{proof}

\begin{lemma}\label{le:meso2} Assume that $\int_\R (1+x^2)|f(x)|^2dx<\infty.$ There exists a $C>0$ such that 
$$
\E \big| \int_\R f(x) e^{G[u^{-1/2}\chi_{[1,A]}](x)}dx)\big|^2\leq C\int_\R (1+x^2)|f(x)|^2dx\quad for \ all \quad  A>1.
$$
\end{lemma}
\begin{proof} By easy approximation (recall that $A$ is finite) we may assume that $f\in C_0^\infty(\R).$ For each $\delta\in [0,1]$ write $h_\delta(u):=e^{-\delta u}u^{-1/2}\chi_{[1,\infty)}(u)$. By the previous lemma,  instead of the functions $(u^{-1/2}\chi_{[1,A]})_{A>1}$ it is enough to prove the uniform bound for the family $(h_\delta(u))_{\delta\in(0,1)}$. 

Now $|h_0|^2(u)=u^{-1}\chi_{[1,\infty)}$ so $\widehat{|h_0|^2}\in L^2(\R)$, and as supp($|h_0|^2$)$\subset [0,\infty)$, the Paley-Wiener theorem verifies that $\widehat{|h_0|^2}$ extends to an analytic function in the lower half plane and for any $\varepsilon>0$, the analytic extension stays bounded in the half plane $\{\Im \xi<-\varepsilon\}$. Moreover $\widehat{|h_0|^2}$ is obtained as the distributional boundary value of this extension (that we denote by the same symbol). By definition,
\begin{equation}\label{eq:PW}
\widehat{|h_\delta|^2}(\xi)=\widehat{|h_0|^2}(-i2\delta+\xi)\quad\textrm{for all}\quad \xi\in\R.
\end{equation}
On the other hand, we may approximate $|h_0|^2$ by $u^{-1}\chi_{[1,A)}$, and obtain for $\xi\not=0$
\begin{align*}
\widehat{|h_0|^2}(\xi) &=\lim_{A\to\infty}\int_1^A \frac{e^{-2\pi i \xi s}}{s}ds=\int_1^\infty \frac{e^{-2\pi i \xi s}}{s}ds= \int_{2\pi \xi}^{\mathrm{sgn}(\xi)\times\infty} \frac{e^{-is}}{s}ds,
\end{align*}
where the computation is validated by the local uniform convergence of the limit on $\xi\in\R\setminus\{0\}.$ Here one e.g. uses the fact that the real and imaginary parts are given by standard cosine and sine integrals. Thus, $\widehat{|h_0|^2}$ is locally smooth, even analytic in a neighbourhood of  any given point $\xi\in \R\setminus\{0\}$. Moreover, we observe the bound
$$
 \big|\widehat{|h_0|^2}(\xi)\big|\leq c|\xi|^{-1}\quad \textrm{for}\quad |\xi|\geq 1, \;  \Im \xi \leq 0.
$$
For $\xi\in B(0,1)\cap \{\Im \xi <0\}$ we can write 
$$
\int_{2\pi \xi}^\infty \frac{e^{-iu}}{u}= C+\int_{2\pi\xi}^{1}\frac{du}{u}+\int_{2\pi\xi}^{1}\frac{e^{-i u}-1}{u}du
$$
 (where the integration contour stays in the lower half plane) and we see that
\begin{equation}\label{eq:PW2}
\widehat{|h_0|^2}(\xi) =\log(1/\xi)+ \textrm {[entire]} \quad \textrm{in}\quad \xi\in B(0,1)\cap \{\Im \xi <0\}.
\end{equation}

By \eqref{eq:PW} we deduce that the distributional limit of $\exp(\widehat{|h_\delta|^2}(\xi))$ equals the distributional boundary values on the real axis of the analytic function $\exp(\widehat{|h_0|^2}(\xi))$ in the lower half plane, assuming that the latter ones exist.  In turn, this follows follows  from  \eqref{eq:PW2} and the above discussion. Namely,  
$$
\exp(\widehat{|h_0|^2}(\xi)) =c_0\xi^{-1} \;+\;[\textrm{analytic and bounded over}\;\R].
$$
 We may then invoke Lemma \ref{le:meso1}, Cauchy-Schwarz
and the Plemelj formula to deduce that
\begin{align*}
&\E \big| \int_\R f(x) e^{G[h_\delta](x)}dx\big|^2\leq C\int_{\R^2}|f(x)f(y)|dxdy+\lim_{\varepsilon\to 0^+}\left|c_0\int_{\R^2}\frac{f(x)f(y)}{x-y+i\varepsilon}dxdy\right|\\
\leq &C'\int_\R (1+x^2)|f(x)|^2dx,
\end{align*}

\noindent where we used the estimate $\int_\R |f(x)|dx\lesssim (\int_\R (1+x^2)|f(x)|^2dx)^{1/2}$ as well as the fact that the Hilbert transform is a bounded operator on $L^2(\R)$.
\end{proof}

The next lemma records a couple of basic properties of the complex chaos defined via exponentials of stochastic integrals of the type considered above.
\begin{lemma}\label{le:meso3} For $0<a\leq 1\leq A$ denote 
$$
\eta_{a,A}(x):=\exp\left[\int_a^1\frac{e^{-2\pi i xu}-1}{u^{1/2}}dB^\C_u + \int_1^A\frac{e^{-2\pi i xu}}{u^{1/2}}dB^\C_u\right].
$$
Then, for any $\alpha>1/2$    we have
\begin{equation}\label{eq:upper1}
\sup_{a\in(0,1),A>1}\E\|\reg{\eta}_{a,A}(x)\|^2_{W^{-\alpha,2}(\R)} <\infty
\end{equation}
and  for every $a\in [0,1]$ there exists the almost sure limit 
\begin{equation}\label{eq:upper2}
(1+x^2)^{-1}\eta_a(x):=\lim_{A\to\infty}\reg{\eta}_{a,A}(x)\;\in W^{-\alpha,2}(\R).
\end{equation}
In a similar vein, there exists the almost sure limit
\begin{equation}\label{eq:upper2'}
(1+x^2)^{-1}\eta(x):=\lim_{a\to 0^+,A\to\infty}\reg{\eta}_{a,A}(x)\;\in W^{-\alpha,2}(\R).
\end{equation}
\end{lemma}
\begin{proof}
It is enough to prove the first statement since the latter ones will be  easy consequences of the first one and the fact that $(1+x^2)^{-1}\eta_{a,A}$ is a $W^{-\alpha,2}(\R)$-valued martingale with respect to decreasing $a$ and increasing $A$. For that end, write $S_a(x):=\int_a^1\frac{e^{-2\pi i xu}-1}{u^{1/2}}dB^\C_u$ so that
$$
\eta_{a,A}(x)=\exp\big(S_a(x) + G[u^{-1/2}\chi_{[1,A]}](x)\big).
$$
By Lemma \ref{le:meso2} and our standard  computation of the expectation of the Sobolev norm (see e.g. the proof of Lemma \ref{le:nu}) we have for any $\beta\in (3/4,1]$
\begin{equation}\label{eq:upper3}
\sup_{A>1}\E\|(1+x^2)^{-\beta}e^{G[u^{-1/2}\chi_{[1,A]}](x)}\|^2_{W^{-\alpha,2}(\R)} <\infty.
\end{equation}
In turn, the part $S_a(x)$ behaves nicely and defines a smooth field, and by the already familiar argument we deduce that
$
\E \exp\big(\lambda \sup_{a\in (0,1)}\| S_a\|_{C^1[0,1]}\big)<\infty\quad\textrm{for any}\quad \lambda >0,
$
whence by Chebyshev we have the exponential tail
$
\Prob ( \sup_{a\in (0,1)}\| S_a\|_{C^1[0,1]}>y) <C(\lambda)e^{-\lambda y}.
$
Assume that $\varepsilon >0$. Choose $\lambda >\varepsilon^{-1}$. By translation invariance we may compute for $y>0$

\begin{align*}
&\Prob \Big( \sup_{a\in (0,1)}\big(|S_a(x)|+ |S'_a(x)|\big)\;\geq \; y +\varepsilon\log(1+x^2)\quad \textrm{for some}\;\,x\in\R\Big)
\\ &\lesssim \sum_{n\in\Z}e^{-\lambda (y+\varepsilon \log(1+n^2))}\lesssim e^{-\lambda y}\sum_{n\in\Z}(1+n^2)^{-\lambda\varepsilon} \lesssim e^{-\lambda y}.
\end{align*}
This easily yields that $\E\big| \sup_{a\in (0,1)}\| (1+x^2)^{-\varepsilon}\exp(S_a(x))\|_{C^1(\R)}\big|^2<\infty$. We may thus choose $\beta=1-\varepsilon$ with $\varepsilon\in (0,1/4)$, assume $\alpha \in (1/2,1]$, and  obtain by independence and \eqref{eq:upper3}
\begin{align*}
&\E\|(1+x^2)^{-1}\eta_{a,A}(x)\|^2_{W^{-\alpha,2}(\R)}\\\lesssim \; &\E \left(\| (1+x^2)^{-\varepsilon}\exp(S_a(x))\|_{C^1(\R)}
\|(1+x^2)^{\varepsilon-1}e^{G[u^{-1/2}\chi_{[1,A]}](x)}\|_{W^{-\alpha,2}(\R)} \right)^2 \; \leq\; C
\end{align*}
uniformly in $a,A.$ Here we again used the fact that functions with nice enough bounds on their derivatives are multipliers in Sobolev spaces -- see the discussion at the end of Section \ref{se:convergence}.

\end{proof}

Our next task is to  approximate our Gaussian field 
$$
\mathcal{G}_N(x):=\sum_{k=1}^N\frac{1}{\sqrt{2 p_k}}(W^{(1)}_k+iW_k^{(2)})p_k^{-ix}.
$$
on the interval $[0,1]$ by stochastic integrals of the type considered above. We'll carry this out in several easy steps. First we replace the summation over primes by a more regular one in terms of the Logarithmic integral: define
\begin{equation*}
\mathcal{G}_{N,1}(x)=\sum_{j=1}^N\frac{1}{\sqrt{2 \mathrm{Li}^{-1}(j)}}(W^{(1)}_j+iW_j^{(2)})(\mathrm{Li}^{-1}(j))^{-ix}.
\end{equation*}
Let us show that this is a good approximation to $\mathcal{G}_N$.
\begin{lemma}\label{le:meso5}
There exists a random smooth function $F_1:[0,1]\to\R$ such that almost surely, $F_{N,1}:=\mathcal{G}_{N,1}-\mathcal{G}_N$ converges to $F_1$ in any $C^\ell[0,1]$, $\ell\geq 1.$
Moreover, for all $\lambda>0$ and any $\ell\geq 0$
\begin{equation*}%\label{eq:204v}
\E \exp\big(\lambda \|  F_1\|_{C^{\ell}[0,1]}\big)<\infty 
\end{equation*}
and 
\begin{equation}\label{eq:esupnbound2}
\E \exp\big(\lambda \sup_{N\geq 1}\|F_{N,1}\|_{C^{\ell}[0,1]}\big)<\infty.
\end{equation}
\end{lemma}
\begin{proof}
To apply Lemma \ref{le:perus}, the term with no derivatives can be estimated with Lemma \ref{le:li-inverse}, and we see that, uniformly for $x\in [0,1]$,
\begin{align*}
\left|\frac{1}{\sqrt{ p_j}}e^{-ix\log p_j}-\frac{1}{\sqrt{\mathrm{Li}^{-1}(j)}}e^{-ix\log \mathrm{Li}^{-1}(j)}\right|&\lesssim \frac{|p_j-\mathrm{Li}^{-1}(j)|}{p_j^{3/2}}
\lesssim j^{-1/2} e^{-\sqrt{\log j}}.
\end{align*}
Differentiation only gives an extra power of $\log j$ here. So we see that \eqref{eq:se0} is satisfied and Lemma \ref{le:perus} applies as before.
\end{proof}
The next step consists of (after first enlarging the probability space if needed) expressing the Gaussian variables $2^{-1/2}(W^{(1)}_j+iW_j^{(2)})$ in terms of $B_t^{\C}=B_t^{(1)}+iB_t^{(2)}$ in the following manner:
\begin{equation*}
W_j^{(i)}=\int_{\mathrm{Li}^{-1}(j)}^{\mathrm{Li}^{-1}(j+1)}\frac{dB_t^{(i)}}{\sqrt{\mathrm{Li}^{-1}(j+1)-\mathrm{Li}^{-1}(j)}}.
\end{equation*}
This leads to the second approximation:
\begin{lemma}\label{le:meso6}
Let 
\begin{align*}
\mathcal{G}_{N,2}(x)&:=\sum_{j=1}^N \int_{\mathrm{Li}^{-1}(j)}^{\mathrm{Li}^{-1}(j+1)}\frac{e^{-ix\log t}}{\sqrt{t}}\frac{dB^\C_t}{\sqrt{\mathrm{Li}^{-1}(j+1)-\mathrm{Li}^{-1}(j)}}.
\end{align*}
Then almost surely, $F_{2,N}:=\mathcal{G}_{N,2}-\mathcal{G}_{N,1}$ converges uniformly to a smooth  function $F_{2}$ and the $C^\ell[0,1]$-norms  of these quantities satisfy the same estimates as in the previous lemma.
\end{lemma}

\begin{proof}
While we are now not in the setting of Lemma \ref{le:perus}, we can still mimic its proof. By Ito's isometry, to get a hold of the expectations of the square of the Sobolev norms $\|\cdot\|_{W^{k,2}(0,1)}^2$, we note that in order to estimate the $L^2$-norm one needs to estimate for $x\in [0,1]$ the square of the error that takes the form
\begin{equation*}
\frac{1}{\mathrm{Li}^{-1}(j+1)-\mathrm{Li}^{-1}(j)}\int_{\mathrm{Li}^{-1}(j)}^{\mathrm{Li}^{-1}(j+1)}\left|\frac{e^{-ix\log \mathrm{Li}^{-1}(j)}}{\sqrt{\mathrm{Li}^{-1}(j)}}-\frac{e^{-ix\log t}}{\sqrt{t}}\right|^2dt.
\end{equation*}
 This quantity is of order ${\mathcal O}\big(\frac{(\mathrm{Li}^{-1}(j+1)-\mathrm{Li}^{-1} (j))^2 }{ \mathrm{Li}^{-1}(j)^{3}}\big)={\mathcal O}\big(j^{-3}\big)$, while the derivative terms come with an extra $\log^{2l} j$. All these are summable over $j$, so we can conclude as before.
\end{proof}

To proceed, we'll want to replace the $1/\sqrt{\mathrm{Li}^{-1}(j+1)-\mathrm{Li}^{-1}(j)}$ by something more convenient. More precisely, we'll make use of the following approximation.

\begin{lemma}\label{le:meso7}
Let 
\begin{equation*}
\mathcal{G}_{N,3}(x):=\int_{\mathrm{Li}^{-1}(1)}^{\mathrm{Li}^{-1}(N+1)}\frac{e^{-ix\log t}}{\sqrt{t}}\frac{dB^\C_t}{\sqrt{\log t}}.
\end{equation*}
Then, almost surely as $N\to\infty$, $F_{3,N}:=\mathcal{G}_{N,3}-\mathcal{G}_{N,2}$ converges uniformly to a random continuous function $F_{3}$, and the $C^\ell[0,1]$-norms  of these quantities satisfy the same estimates as in Lemma \ref{le:meso5}.
\end{lemma}
\begin{proof}
We again argue as in the proof of Lemma \ref{le:perus}. Now we need to estimate terms of the form 
\begin{equation*}
\int_{\mathrm{Li}^{-1}(j)}^{\mathrm{Li}^{-1}(j+1)}\left|\frac{1}{\sqrt{\mathrm{Li}^{-1}(j+1)-\mathrm{Li}^{-1}(j)}}-\frac{1}{\sqrt{\log t}}\right|^2\frac{1}{2t}dt,
\end{equation*}
and  similar ones coming with a factor of $\log^{2\ell} t$ coming from the derivative term in the Sobolev estimate. To estimate such a term, we see that it is enough for us to estimate the difference $|\mathrm{Li}^{-1}(j+1)-\mathrm{Li}^{-1}(j)-\log t|$ for $t\in[\mathrm{Li}^{-1}(j),\mathrm{Li}^{-1}(j+1)]$. For this, we recall that  $\mathrm{Li}'(x)=1/\log x$ and use a change of variable to write
\begin{align*}
&|\mathrm{Li}^{-1}(j+1)-\mathrm{Li}^{-1}(j)-\log t|
\leq\int_{j}^{j+1}|\log( \mathrm{Li}^{-1}(s))-\log t |ds\\
\lesssim&\; \mathrm{Li}^{-1}(j+1)-\mathrm{Li}^{-1}(j) \; \lesssim\; \log \frac{\mathrm{Li}^{-1}(j+1)}{\mathrm{Li}^{-1}(j)}
\;\lesssim\;  \frac{\mathrm{Li}^{-1}(j+1)-\mathrm{Li}^{-1}(j)}{\mathrm{Li}^{-1}(j)}
\;\lesssim \;j^{-1}.
\end{align*}
Above we used the asymptotics $\mathrm{Li}^{-1}(j)\sim j\log j$  and $(\mathrm{Li}^{-1})'(j)\sim \log j$.
Hence the square of the Sobolev norm can be bounded by ${\mathcal O}\big((j^{-3}(\log j)^{2\ell}\big)$, which is summable and the rest of the proof goes through as before.
\end{proof}

We note that  $(2\pi e^{2\pi u})^{-1/2}dB^\C_{e^{2\pi u}}$ is a standard Brownian motion, which we shall denote with slight abuse of notation still by $dB^\C_s$.  After performing a change of variables $t=\exp{2\pi u}$ in the integral defining our last approximation, the outcome  is 
$$
\int_{\log(\mathrm{Li}^{-1}(1))/2\pi}^{\log(\mathrm{Li}^{-1}(N+1))/2\pi}\frac{e^{-2\pi ixu}}{\sqrt{u}}dB^\C_u.
$$
By observing that  a term of the form $\int_{\log(\mathrm{Li}^{-1}(1))/2\pi}^{1}\frac{e^{-2\pi ixu}}{\sqrt{u}}dB^\C_u$ can safely be absorbed into the error term and noting that we could have equally well considered an arbitrary interval $[-A,A]$ instead of $[0,1]$  we  thus obtain a new variant of Theorem  
\ref{th:gaussian_appro}:
 \begin{proposition}\label{pr:gaussian_appro2}
For each $N\geq 1$ we may write
\begin{equation*}
\log \zeta_{N,\rand}(1/2+ix)=\mathcal{\widetilde G}_N(x)+\mathcal{\widetilde E}_N(x),
\end{equation*}
where $\mathcal{\widetilde G}_N$ is a Gaussian process on $\R$ and can be written as the stochastic integral
\begin{equation*}
\mathcal{\widetilde G}_N(x): =G\big[u^{-1/2}\chi_{[1,\log\mathrm{Li}^{-1}(N+1)]}\big](x)
\end{equation*} 
The function $\mathcal{\widetilde E}_N$ is smooth and as $N\to\infty$, it a.s. converges uniformly in every $C^\ell[-A,A]$ $($$A>0,\ell\geq 1$$)$ to a random smooth function on $\R$. Moreover, the maximal error and its derivatives in this decomposition have finite exponential moments:
\begin{equation*}
\E\exp\left(\lambda\sup_{N\geq 1}\|\mathcal{\widetilde E}_N(x)\|_{C^\ell[-A,A]}\right)<\infty \quad \mathrm{for \ all\ }\lambda>0\quad \textrm{and}\quad \ell\geq 0.
\end{equation*}
\end{proposition}

We are now ready to prepare for the actual proof of the mesoscopic scaling result. We record first a simple estimate for dilations
\begin{lemma}\label{le:vanutus} Assume that $\alpha>1/2$ and $f\in W^{-\alpha,2}(\R)$. Then for any $\delta\in (0,1)$ it holds that
\begin{equation}
\| f(\delta \cdot)\|_{W^{-\alpha,2}(\R)}\leq \delta^{-1-2\alpha}\| f\|_{W^{-\alpha,2}(\R)}.\notag
\end{equation}
\end{lemma}
\begin{proof}
One simply computes
\begin{eqnarray*}
\| f(\delta \cdot)\|^2_{W^{-\alpha,2}(\R)} &= &\int_\R|\delta^{-1}\widehat f(\delta^{-1}\xi)|^2(1+\xi^2)^{-\alpha}d\xi
= \delta^{-1}\int_\R|\widehat f(\xi)|^2(1+(\delta\xi)^2)^{-\alpha}d\xi\\
&\leq &\;\delta^{-1-2\alpha} \| f\|^2_{W^{-\alpha,2}(\R)}.
\end{eqnarray*}
\end{proof}
From now on, we focus on the interval $(0,1)$. Recall $\eta_{a,A}$ from Lemma \ref{le:meso3} and note that the same lemma verifies the existence of the complex chaos
$$
\eta:=\lim_{A\to\infty}\eta_{0,A} =\lim_{A\to\infty}\exp\big(\int_0^1(e^{-2\pi ixu}-1)u^{-1/2}dB^\C_u+\int_1^Ae^{-2\pi ixu}u^{-1/2}dB^\C_u\big)
$$
and $\eta\in W^{-\alpha,2}(0,1)$ for any $\alpha>1/2$.

The following result is our main ingredient for the proof of Theorem \ref{th:meso1}.
\begin{proposition}\label{pr:meso7} Fix $\alpha>1/2.$ 
Assume that $\delta\in (0,1)$ and consider the dilatations of the randomised zeta function on the interval $(0,1)$. Then we may decompose
\begin{equation}\label{eq:dila1}
{\zeta_\rand}(1/2+i\delta\cdot)_{|(0,1)}= h_\delta e^{Y_\delta}\eta^{(\delta)}.
\end{equation}
Here the distribution of the complex multiplicative chaos $\eta^{(\delta)}$ is independent of $\delta$:
\begin{equation}\label{eq:dila3}
\eta^{(\delta)}\sim\eta \quad \textrm{in}\;\; W^{-\alpha,2}(0,1)\quad \textrm{for any}\quad \delta\in (0,1).
\end{equation}
Moreover, $h_\delta$ is a random smooth function on $[0,1]$ which tends almost surely $($and hence in distribution$)$  to the constant function $1:$ for each $\ell\geq 0$
\begin{equation}\label{eq:dila4}
h_\delta\stackrel{a.s.}{\to}1 \quad\textrm{in}\quad C^{\ell}[0,1]\quad \textrm{as}\quad\delta\to 0^+.
\end{equation}
Finally, $Y_\delta$ is a complex $($scalar$)$ random variable that can be written in the form
\begin{equation}\label{eq:dila5}
Y_\delta\stackrel{d}{=} \sqrt{\log(1/\delta)}Z+R,
\end{equation}
where $Z$ is a standard complex normal random variable  and the random variable $R$, which is independent of $\delta$, satisfies
$\E\exp(\lambda |R|)\leq C_\lambda$ for all $\lambda >0$.
\end{proposition}
\begin{proof} Apply the dilation $x\mapsto \delta x$ in Proposition \ref{pr:gaussian_appro2} in order to write
\begin{eqnarray*}
\zeta_{\rand, N}(1/2+i\delta x)&=& \exp \big({\widetilde{\mathcal{E}}_N(\delta x)}+
\mathcal{\widetilde G}_N(\delta x)\big) \\ &=& \exp \big({\widetilde{\mathcal{E}}_N(\delta x)}\big)\exp\Big(G\big[u^{-1/2}\chi_{[1,\log\mathrm{Li}^{-1}(N+1)]}\big](\delta x)\Big).
\end{eqnarray*}
By making of variables in the stochastic integral and letting  $dB^{\C,\delta}_u =\delta^{-1/2}dB^{\C}_{\delta u}$ stand for another copy of the standard Brownian motion we obtain
\begin{eqnarray*}
&& \mathcal{\widetilde E}_N(\delta x) + \mathcal{\widetilde G}_N(\delta x)
=\mathcal{\widetilde E}_N(\delta x) +  \int_{1}^{\log \mathrm{Li}^{-1}(N+1)}e^{-2\pi i x\delta u}u^{-1/2}dB^\C_u\\
&&=\mathcal{\widetilde E}_N(\delta x) +  \int_{\delta}^{\delta\log \mathrm{Li}^{-1}(N+1)}e^{-2\pi i x u}u^{-1/2}dB^{\C,\delta}_u\\
&&=\Big(\mathcal{\widetilde E}_N(\delta x) -\mathcal{\widetilde E}_N(0)- \int_{0}^{\delta}(e^{-2\pi i x u}-1)u^{-1/2}dB^{\C,\delta}_u\Big)\\
 && +\Big(\int_{0}^{1}(e^{-2\pi i x u}-1)u^{-1/2}dB^{\C,\delta}_u
 + \int_{1}^{\log \delta \mathrm{Li}^{-1}(N+1)}e^{-2\pi i x u}u^{-1/2}dB^{\C,\delta}_u\Big)\\
 &&\quad +\left(\mathcal{\widetilde E}_N(0)+\int_\delta^1u^{-1/2}dB^{\C,\delta}_u \right)=: A_{\delta,N}(x)+B_{\delta,N}(x)+C_{\delta,N}(x).
\end{eqnarray*}
Here the results of Proposition \ref{pr:gaussian_appro2}   imply that  as $N\to\infty$, we have almost surely
$$
\exp(A_{\delta,N}(x))\to \exp\Big( \mathcal{\widetilde E}(\delta x) -\mathcal{\widetilde E}(0) - \int_{0}^{\delta}(e^{-2\pi i x u}-1)u^{-1/2}dB^{\C,\delta}_u\Big)=:h_{\delta}(x)
$$
with convergence in $C^\ell[0,1]$ for any $\ell\geq 0$. Moreover, \eqref{eq:dila4} is clearly true.
Next, we observe  the almost sure convergence
$$
C_{\delta,N}(x)\to \exp( \mathcal{\widetilde E}(0) + \int_\delta^1u^{-1/2}dB^{\C,\delta}_u)=: e^{Y_\delta},
$$
where $Y_\delta$ has the stated properties by the Ito isometry and Proposition \ref{pr:gaussian_appro2}.
Finally, recalling the definition of $\eta^{(\delta)}$ we have almost surely that
$$
\exp(B_{\delta,N}(x))\to \eta^{(\delta)}\quad \textrm{in}\quad W^{-\alpha,2}(0,1).
$$
By combining all the above observations and invoking the definition of $\zeta_\rand$ the equality \eqref{eq:dila1} follows.
\end{proof}

We can now prove our first characterization of the mesoscopic behavior of $\zeta$.

\begin{proof}[Proof of Theorem \ref{th:meso1}]
Combining Theorem \ref{th:main}, Proposition \ref{pr:meso7} and Lemma \ref{le:vanutus}, we see that for each $k\in \Z_+$, there exists a $T_k$ so that for $T\geq T_k$

\begin{align*}
\wass_2(\zeta(1/2&+ik^{-1} x+i\omega T),\zeta_\rand(1/2+ik^{-1}x))_{W^{-\alpha,2}(0,1)}\\
&=\wass_2(\zeta(1/2+ik^{-1} x+i\omega T),h_{1/k}(x) e^{Y_{1/k}} \eta^{(1/k)}(x))_{W^{-\alpha,2}(0,1)}\\
&\leq 1/k.
\end{align*}

We can naturally take $T_k$ to be increasing in $k$. Taking $\delta_T=1/k$ for $T\in[T_k,T_{k+1})$ along with $h_T=h_{\delta_T}$ etc. gives the claim.

\end{proof}

Let us now turn to the proof of Theorem \ref{th:meso2}. Since $W^{-\alpha,2}_{{\rm mult}}(0,1)$ is a bounded, complete  and separable metric space, convergence of  $W^{-\alpha,2}_{{\rm mult}}(0,1)$-valued random variables in the corresponding Wasserstein metric is equivalent to standard convergence in distribution. We need one last auxiliary result:
\begin{lemma}\label{le:wass} Assume that the $W^{-\alpha,2}(0,1)$-valued random variables $g_k$ are almost surely non-zero
and 
$$
\wass_2\big(g_k\; , \; g\big)_{W^{-\alpha,2}(0,1)}\to 0\quad \textrm{as}\quad k\to\infty,
$$
where $g$ is an almost surely non-zero $W^{-\alpha,2}(0,1)$-valued random variable. Then also
$$
\wass_2\big(g_k\; , \; g\big)_{W^{-\alpha,2}_{{\rm mult}}(0,1)}\to 0\quad \textrm{as}\quad k\to\infty,
$$
\end{lemma}
\begin{proof}
The statement follows easily after one  notes that for any $\varepsilon >0$ there is an $r$ such that $\Prob (\|g\|_{W^{-\alpha,2}(0,1)}<2r) <\varepsilon/2,$ and  hence for any $k\geq k_0(\varepsilon)$  it holds that  $\Prob (\|g_k\|_{W^{-\alpha,2}(0,1)}<r) <\varepsilon$. A fortiori, for $k\geq k_0$
$$
\wass_2\big(g_k\; , \; g\big)_{W^{-\alpha,2}_{{\rm mult}}(0,1)} \leq 4\varepsilon +r^{-1}\wass_2\big(g_k\; , \; g\big)_{W^{-\alpha,2}(0,1)},
$$
and the claim follows.
\end{proof}

We are finally prepared for

\begin{proof}[Proof of Theorem \ref{th:meso2}] 
Observe that $\eta$ is almost surely non-vanishing as an almost sure limit of a non-trivial $L^2$-bounded martingale by its definition and a simple application of Kolmogorov's zero-one theorem, and the same holds for $\zeta_\rand$.  Moreover, we note that $\eta\sim \eta^{(\delta)}$ for all $\delta >0$. Proposition \ref{pr:meso7} yields that
\begin{eqnarray*}
&&\wass_2\big(\eta\; , \; \zeta_\rand(1/2+i\delta\cdot)\big)_{W^{-\alpha,2}_{{\rm mult}}(0,1)} =\wass_2\big(\eta^{(\delta)}\; , \; h_\delta e^{Y_\delta}\eta^{(\delta)}\big)_{W^{-\alpha,2}_{{\rm mult}}(0,1)} \\
&=& \wass_2\big(\eta^{(\delta)}\; , \; h_\delta \eta^{(\delta)}\big)_{W^{-\alpha,2}_{{\rm mult}}(0,1)} \to 0\quad \textrm{as}\quad \delta\to 0^+,
\end{eqnarray*}
where in the last step we used \eqref{eq:dila4} \footnote{More precisely, using the fact that $h_\delta$ is a multiplier on $W^{-\alpha,2}(0,1)$, one readily checks that for small, but arbitrary $\varepsilon>0$, $\Prob(||(h_\delta-1)\eta^{(\delta)}||_{W^{-\alpha,2}(0,1)}\leq \varepsilon ||\eta^{(\delta)}||_{W^{-\alpha,2}(0,1)})\geq 1-\varepsilon$ for small enough $\delta$. Now on this event, we have $\big\|||\eta^{(\delta)}||_{W^{-\alpha,2}(0,1)}^{-1}\eta^{(\delta)}-||h_\delta\eta^{(\delta)}||_{W^{-\alpha,2}(0,1)}^{-1} h_\delta \eta^{(\delta)}\big\|_{W^{-\alpha,2}(0,1)}=\mathcal{O}(\varepsilon)$, where the implied constant is deterministic. This implies that in the metric of $W^{-\alpha,2}_{\mathrm{mult}}(0,1)$, as $\delta\to 0^+$, the distance between $\eta^{(\delta)}$ and $h_\delta \eta^{(\delta)}$ tends to zero in probability. As the space is bounded, this implies convergence in the Wasserstein sense as well.}.  We then set $\delta= 1/k$ and apply Lemma \ref{le:vanutus} and Lemma \ref{le:wass} to pick a strictly increasing
sequence $T_k$ so that
\begin{align*}
\wass_2(\mu_T(k^{-1}\cdot),\zeta_\rand(1/2+ik^{-1}\cdot))_{W^{-\alpha,2}_{{\rm mult}}(0,1)}  \leq 1/k \quad\textrm{for}\quad T\geq T_k
\end{align*}
By combining the above inequalities we see that  the choice  $\delta_T=1/k$ for $T\in [T_k,T_{k+1})$ applies as before.
\end{proof}

\section{Relationship to real Gaussian multiplicative chaos in the subcritical case: Proof of Theorem \ref{th:gmcconv}}\label{se:subcritical}

In this section, we prove  that for $0<\beta<\beta_c:=2$, $|\zeta_{N,\rand}(1/2+ix)|^\beta/\E |\zeta_{N,\rand}(1/2+ix)|^\beta dx$ almost surely converges with respect to the weak topology of measures to a random measure which is absolutely continuous with respect to a Gaussian multiplicative chaos measure. This will be an easy consequence of  our Gaussian coupling and the general theory of real Gaussian multiplicative chaos measures. 

For a proper introduction to the theory of Gaussian multiplicative chaos, we refer the reader to Kahane's original work \cite{Kahane} or the recent review by Rhodes and Vargas \cite{RV}. We also point out Berestycki's elegant proof for the existence and uniqueness of subcritical Gaussian multiplicative chaos measures \cite{Be}. For the convenience of the reader, we nevertheless recall the main results from the theory that are relevant to us. 

\begin{theorem}\label{th:gmc}
Assume that we have a sequence of independent Gaussian fields $(Y_k)_{k=1}^\infty$ on $[0,1]$ and that the covariance kernel of $Y_k$, $K_{Y_k}$, is continuous on $[0,1]\times[0,1]$. Define the field

\begin{equation*}
X_n=\sum_{k=1}^n Y_k,
\end{equation*}

\noindent and assume that as $n\to \infty$,  the covariance kernel $K_{X_n}$ converges locally uniformly in $[0,1]^2\setminus \{x=y\}$  to a function on $[0,1]^2$ which is of the form

\begin{equation*}
\log \frac{1}{|x-y|}+g(x,y),
\end{equation*}

\noindent where $g$ is bounded and continuous. Moreover, assume that   there is a constant $C<\infty$ so that
\begin{equation}\label{eq:ylaraja}
K_{X_n}(x,y)\leq \log \frac{1}{|x-y|}+ C \quad {\rm for \; all}\quad  x,y\in [0,1]\quad\mathrm {and}\quad n\geq 1.
\end{equation}
Then for $\beta>0$ the random measure 

\begin{equation*}
\lambda_{\beta,n}(dx)=\frac{e^{\beta X_n(x)}}{\E e^{\beta X_n(x)}}dx
\end{equation*}

\noindent converges almost surely with respect to the topology of weak convergence of measures to a limiting measure $\lambda_\beta$. This limiting measure is a non-trivial random measure for $\beta<\beta_c=\sqrt{2}$ and for $\beta\geq \beta_c$, it is the zero measure. Moreover, if $0<\beta<\sqrt{2}$, and $0<p< 2/\beta^2$, then for a compact set $A\subset [0,1]$

\begin{equation*}
\E (\lambda_\beta(A)^p)<\infty.
\end{equation*}

\end{theorem}
\begin{proof} (Sketch) By \eqref{eq:ylaraja} and Kahane's convexity inequality (see \cite[Theorem 2.1]{RV}) one may easily compare to a standard approximation of a chaos measure and deduce  that for any $\beta<\beta_c$ the random variables $\lambda_{\beta, n} [0,1]$ form an $L^p$-martingale with a suitable $p=p(\beta)>1$. At this stage the standard theory of multiplicative chaos can be applied to obtain the rest of the claims, see e.g.  \cite[Theorems 2.5 and 2.11]{RV}. 

\end{proof}

As we are studying $|\zeta_{N,\rand}(1/2+ix)|^\beta$, our relevant field is $\log |\zeta_{N,\rand}(1/2+ix)|$, and the relevant Gaussian part is the real part of $\mathcal{G}_N$. To simplify notation slightly, let us write $G_N=\mathrm{Re}(\mathcal{G}_N)$ and $X_N=\log|\zeta_{N,\rand}(1/2+ix)|$. To apply Kahane's construction of a Gaussian multiplicative chaos measure, we'll need to establish that the covariance of $G_N$ satisfies the requirements of Theorem \ref{th:gmc}. From the definition of $\mathcal{G}_N$, we see that 
$$
G_N(x)=\sum_{j=1}^N\frac{1}{\sqrt{2p_j}}\left(W^{(1)}_j\cos(x\log p_j)+W^{(2)}_j\sin(x\log p_j)\right).
$$
A direct computation shows that
$$
K_{ G_{N}}(x-y):=\E G_{N}(x)G_{N}(y) =\Psi_N(x-y),
$$
where 
$$
\Psi_N(u):=\frac{1}{2}\sum_{j=1}^N\frac{\cos(u\log p_j)}{p_j}.
$$

The following result is enough for us to be able to apply Kahane's theory for defining a multiplicative chaos measure. It is of interest to note that we are dealing with a logarithmically correlated translation invariant field whose covariance deviates from $\frac{1}{2}\log (1/|x-y|)$ by only  a  smooth function.

\begin{lemma}\label{le:covariance_approximation}
We have 
$$
\Big| K_{ G_N}(x,y)-\frac{1}{2}\log\Big(\min \Big(\frac{1}{|x-y|},\log N \Big)\Big)\Big|\;\leq \; C,
$$
where $C$ is uniform over $N\geq 1$ and $(x,y)\in [0,1].$ Moreover, if $x\not=y$
$$
K_{ G_n}(x,y)\longrightarrow K_G(x,y)=\frac{1}{2}\log\frac{1}{|x-y|}+g(x-y)\quad\textrm{as}\quad {n\to\infty},
$$
with  local uniform convergence outside the diagonal. Moreover, we have  $g\in C^\infty (-2,2)$. 
\end{lemma}

\begin{proof} We shall employ the notation where $z\add \widetilde{z}$ for given quantities $z=z_N(u),\widetilde z=\widetilde z_N(u)$ stands for the uniform inequality $|z_N(u)-\widetilde z_N(u)|\leq C$ with a universal bound $C$ and such that $\lim_{N\to\infty} (z_N(u)-\widetilde z_N(u))$ converges uniformly to a continuous function on the interval $u\in [-2,2].$ We shall employ the well-known asymptotics

\begin{equation*}%\label{eq:pr_as}
p_j=j\log j +\mathcal{O}(j\log\log j).
\end{equation*}
This implies that $\sum_{j=1}^\infty \frac{|\log p_j-\log (j\log j)|}{p_j}<\infty$ and since the cosine function is 1-Lipschitz we obtain
$$
\Psi_N(u) \add \frac{1}{2}\sum_{j=1}^N\frac{\cos\big(u\log (j\log j)\big)}{p_j}.
$$
In a similar vein, $\sum_{j=1}^\infty \big|p_j^{-1}-(j\log j)^{-1}\big|<\infty$ which leads to
\begin{equation*}%\label{eq:700}
\Psi_N(u) \add \frac{1}{2}\sum_{j=1}^N\frac{\cos\big(u\log (j\log j)\big)}{j\log j}.
\end{equation*}
Next we observe that for all $u\in [-2,2]$ and $x\geq 10$
$$
 \left|\frac{d}{dx} \left(\frac{\cos\big(u\log (x\log x)\big)}{x\log x}\right)\right| 
 \;\leq \; \frac{6}{x^2\log x}.
$$
Since $\int_{10}^\infty (x^2\log x)^{-1}dx <\infty ,$ it follows that
\begin{equation*}%\label{eq:701}
\Psi_N(u) \add \frac{1}{2}\int_{x=10}^N\frac{\cos\big(u\log (x\log x)\big)dx}{x\log x}.
\end{equation*}
To continue, we note that
$$
\int_{x=10}^\infty\left| 1-\frac{1+\log x}{\log x+\log\log x}\right|\frac{dx}{x\log x}<\infty
$$
so that
\begin{eqnarray}\label{eq:702}
\Psi_N(u) &\add& \frac{1}{2}\int_{x=10}^N\frac{\cos\big(u\log (x\log x)\big)}{\log(x\log x)}\frac{(1+\log x)dx}{x\log x}\nonumber\\
&\add& \frac{1}{2}\int_{1}^{\log N+\log\log N}\frac{\cos(ut)}{t}dt\add  \frac{1}{2}\int_{1}^{\log N}\frac{\cos(ut)}{t}dt\\
&=& \frac{1}{2}\int_{u}^{u\log N}\frac{\cos(x)}{x}dx\; =:\; A(u,N).\nonumber
\end{eqnarray}
Above in the first step we performed the change of variables $u=\log(x\log x)$ and noted that $du=(1+\log x)dx/x\log x$. In the second to last step we used the fact that $\int_{\log N}^{\log N+\log\log N}t^{-1}dt =o(1)$ as $N\to\infty.$

It remains to prove the claim for $A(u, N)$ defined in \eqref{eq:702}. 
Since $\lim_{z\to\infty} \int_{1}^{z}\frac{\cos(x)}{x}dx$  exists and is finite, we see
directly from the definition that for any $\varepsilon_0 >0$ in the set $\{ \varepsilon_0\leq |u|\leq 2\}$ the function $A(u,N)$ converges uniformly to a continuous function of $u$ as $N\to\infty$. Moreover, since $\int_0^1|\cos (x)-1|x^{-1}dx<\infty,$ we get  for $|u|\geq (\log N)^{-1}$
$$
|A(u,N)-\frac{1}{2}\int_u^1x^{-1}dx|=|A(u,N)-\frac{1}{2}\log(1/u)|\leq C,
$$
where $C$ is independent of $N$ and $u$. 
Finally, if $|u|\leq(\log N)^{-1}$ we get in a similar manner
$$
|A(u,N)-\frac{1}{2}\int_u^{u\log N}x^{-1}dx|=|A(u,N)-\frac{1}{2}\log\log N|\leq C',
$$
and now $C'$ is independent of $N$ and $u\in \{|u|\geq (\log N)^{-1}\}.$ This proves the first statement of the lemma.

By \eqref{eq:702} we deduce that there is a continuous function $\widetilde b(u)$ on $[-2,2]$ so that the limit $\Psi$ of the functions $\Psi_N$ takes the form 
\begin{equation*}%\label{eq:704}
\Psi (u)= \widetilde b(u)+ \frac{1}{2}\int_{u}^{\infty}\frac{\cos(t)}{t}dt{u} =
\frac{1}{2}\log\left(|u|^{-1}\right)+ b(u)\quad\textrm{for}\quad 0<|u|<2,
\end{equation*}
with $b\in C[-2,2]$  as $u\mapsto \int_u^1(\cos (x)-1)x^{-1}dx$ is continuous over $u\in [0,1].$  Especially, we know that $\Psi (x-y)$ yields the covariance operator of our limit field since the estimates we have proven imply that $\Psi_N(x-y)\to\Psi(x-y)$ in $L^2([0,1]^2)$, and convergence in the Hilbert-Schmidt norm is enough to identify the limit covariance of a sequence of Gaussian fields converging a.s. in the sense of distributions. We still want to upgrade $b$ to be smooth. For that end we first fix $\delta_0 >0$ and observe that what we have proved up to now (see especially \eqref{eq:702}) yields that  we have 
\begin{equation*}%\label{eq:705}
\Psi (u)=\frac{1}{2} {\rm Re\,}\left(\lim_{N\to\infty} \sum_{j=1}^Np_j^{-1-iu}\right)
\end{equation*}
with uniform convergence in the set $\{ \delta_0\leq |u|\leq 2\}$. However, if we apply exactly the same argument as above to the sum ${\rm Re\,}\big(\sum_{j=1}^Np^{-1-\varepsilon -iu}\big)$
for, say, $\varepsilon\in [0,1/2],$ we obtain uniform (in $\varepsilon$) estimates for the convergence of the series
$$
{\rm Re\,}\left(\sum_{j=1}^\infty p_j^{-1-\varepsilon -iu}\right)
$$
for any fixed $u\in (0,2)$.
Especially, we deduce by invoking the logarithm of the Euler product of the Riemann zeta function that 
\begin{eqnarray*}%\label{eq:706}
\Psi (u)&=&\lim_{\varepsilon \to 0^+} \frac{1}{2}{\rm Re\,}\left(\sum_{j=1}^\infty p_j^{-1-\varepsilon -iu}\right)
=\lim_{\varepsilon \to 0^+} \frac{1}{2}{\rm Re\,}\left(\zeta(1+\varepsilon+iu)-\sum_{k=2}^\infty\sum_{j=1}^\infty \frac{p_j^{-k(1+\varepsilon +iu)}}{k}\right)\nonumber\\
&=& \frac{1}{2}{\rm Re\,}\left(\zeta(1+iu)-\sum_{k=2}^\infty\sum_{j=1}^\infty k^{-1}p_j^{-k(1+iu)}\right),\nonumber
\;=:\; \frac{1}{2}{\rm Re\,}\left(\log(\zeta(1+iu))-A(u)\right),\nonumber
\end{eqnarray*}
as the last written double sum converges absolutely (uniformly in $\varepsilon$).
It remains to note that $\log(\zeta(1+iu))$ is real analytic on $(0,\infty)$, and the function $A$ is $C^\infty$-smooth on the same set as termwise differentiation of $A$ $\ell$ times with respect to $u$ produces a series with the majorant series
$$
\sum_p\sum_{k=2}^\infty k^{\ell-1}p^{-k}\log^\ell p\leq
\sum_pp^{-3/2}\Big(\sum_{r=0}^\infty (r+2)^{\ell-1}p^{-r}\Big)<\infty .
$$
\end{proof}

\begin{remark}{\rm
Note that in our case the asymptotic covariance has a singularity of the form $-\frac{1}{2}\log |x-y|$ instead of $-\log |x-y|$ as in Theorem \ref{th:gmc}. This simply means that we replace $\beta$ by $\beta/\sqrt{2}$ in Theorem \ref{th:gmc}.}
\hfill $\blacksquare$
\end{remark}

Before proving the convergence of the subcritical chaos we still need to note that the normalizing constant in our exponential martingale obtained via the Gaussian approximation behaves like that of our original martingale. 
\begin{lemma}\label{le:exp} For any $\beta >0$ there is a constant $C=C(\beta)$ such that for $N\geq 1$ and $x\in[0,1]$
$$
C^{-1}\E \exp(\beta G_N(x))\leq \E |\zeta_{N,\rand}(1/2+ix)|^\beta \leq C\E \exp(\beta G_N(x))\quad \textrm{for all }.
$$
\end{lemma}
\begin{proof} By independence, rotational invariance of the law of $e^{2\pi i \theta_j}$, and the finiteness of the sum $\sum_{j=1}^\infty \sum_{k=3}^\infty k^{-1}p_j^{-k/2}<\infty$, we see that

\begin{align}\label{eq:1point1}
\E|\zeta_{N,\rand}(1/2+ix)|^\beta =\prod_{j=1}^N \E e^{-\beta \log |1-p_j^{-1/2}e^{2\pi i \theta_j}|}&=\prod_{j=1}^N \E e^{\beta \sum_{k=1}^\infty \frac{1}{k}p_j^{-k/2} \cos (2\pi k \theta_j)}\\
\notag&\asymp \prod_{j=1}^N \E e^{\beta p_j^{-1/2}\cos (2\pi \theta_j)+\frac{\beta}{2}p_j^{-1}\cos (4\pi \theta_j)},
\end{align}

\noindent where we use the notation $a_n\asymp b_n$ to indicate that the ratios $a_n/b_n$ and $b_n/a_n$ are bounded. Now consider the function $\lambda\mapsto \E e^{\beta \lambda \cos (2\pi \theta_j)+\beta \frac{\lambda^2}{2}\cos (4\pi \theta_j)}$. This is analytic and one can easily check that 

\begin{equation}\label{eq:1point2}
\E e^{\beta \lambda \cos (2\pi \theta_j)+\beta \frac{\lambda^2}{2}\cos (4\pi \theta_j)}= e^{\frac{\beta^2}{4}\lambda^2+\mathcal{O}(\lambda^3)}
\end{equation} 

\noindent as $\lambda\to 0$. Substituting $\lambda=p_j^{-1/2}$ and noting that $\sum_{j=1}^\infty p_j^{-3/2}<\infty$, combining \eqref{eq:1point1} and \eqref{eq:1point2} implies that 

\begin{equation*}
\E|\zeta_{N,\rand}(1/2+ix)|^\beta\asymp  e^{\frac{\beta^2}{4}\sum_{j=1}^N \frac{1}{p_j}}.
\end{equation*}
From the definition of $G_N$, we see that $\E G_N(x)^2=\sum_{j=1}^N \frac{1}{2p_j}$ so we conclude that $$\E e^{\beta G_N(x)}=e^{\frac{\beta^2}{2}\sum_{j=1}^N \frac{1}{2p_j}}\asymp \E|\zeta_{N,\rand}(1/2+ix)|^\beta,$$ which was the claim.
\end{proof}

One should note that combining the above lemmas  we see that for $N\geq 1$
\begin{equation*}%\label{eq:expmoments}
\E |\zeta_{N,\rand}(1/2+ix)|^\beta\asymp \E \exp(\beta G_N(x))\asymp\exp(\frac{\beta^2}{4}\log\log N)=(\log N)^{\beta^2/4}.
\end{equation*}

Finally we are ready for:

\begin{proof}[Proof of Theorem \ref{th:gmcconv}] Consider the Gaussian field $G$ that is the limit of the fields $G_N$ (this limit exists for example in a suitable Sobolev space). For $\beta<2$ the corresponding Gaussian multiplicative chaos measure $\lambda_\beta$, exists due to Theorem \ref{th:gmc}, and the approximating measures obtained from the fields $G_N$ (we denote these measures by $\lambda_{\beta,N}$) converge  to $\lambda_\beta$. In particular, there is a $\widetilde p>1$ such that 
$\E (\lambda_{\beta, N}[0,1]^{\widetilde p})\leq C<\infty$ for all $N\geq 1$. Recall that we want to prove that for each continuous $f:[0,1]\to \R_+$, 
\begin{equation*}
\int_0^1 \frac{|\zeta_{N,\rand}(1/2+ix)|^\beta}{\E |\zeta_{N,\rand}(1/2+ix)|^\beta}f(x)dx
\end{equation*}
converges almost surely to a non-trivial random variable (almost sure convergence with respect to the weak topology then follows from the separability of the space $C[0,1]$). By the construction of the measure, this is a positive martingale, so it is enough to prove that it is bounded in $L^p$ for some $p>1$. For this it is then enough to show that for the special case where we choose $f=1$, the martingale is bounded in $L^p$ for some $p>1$. Choose $p\in (1,\widetilde p)$ and note that since the normalization factors are comparable, we obtain by H\"older's inequality
and Theorem \ref{th:gaussian_appro}
\begin{align*}
\E&\left[\int_0^1\frac{|\zeta_{N,\rand}(1/2+ix)|^\beta}{\E |\zeta_{N,\rand}(1/2+ix)|^\beta}dx\right]^p
\notag \;\leq \; C \E\Big( \exp(p\beta \| \mathcal{E}_N\|_{L^\infty[0,1]})(\lambda_{\beta,N}[0,1])^p\Big) \\
&\leq C\left(\E \exp\Big(p(\widetilde p/p)'\beta \| \mathcal{E}_N\|_{L^\infty[0,1]}\Big)\right)^{1/(\widetilde p/p)'}
\big(\E(\lambda_{\beta,N}[0,1])^{\widetilde p}\big)^{p/\widetilde p}
\;\;\leq \;C',
\end{align*}
where $'$ denotes H\"older conjugation. This yields uniform integrability of the quantity $\int_0^1|\zeta_{N,\rand}(1/2+ix)|^\beta/\E |\zeta_{N,\rand}(1/2+ix)|^\beta dx$ which proves the existence of a non-trivial limit. 

For $\beta\geq \beta_c$, we see similarly using Theorem \ref{th:gmc} and Theorem \ref{th:gaussian_appro} that the limiting measure is zero since the corresponding Gaussian limiting measure is zero. The claim about the existence of all moments for the Radon-Nikod\'ym derivative follows from Theorem \ref{th:gaussian_appro}. The existence of moments of order $p<4/\beta^2$ again follows from a simple H\"older argument making use of Theorem \ref{th:gaussian_appro} and the corresponding result for the Gaussian case (see Theorem \ref{th:gmc}).
\end{proof}

\section{The critical measure: Proof of Theorem \ref{th:critical}}\label{sec:critical}

In this section we establish the existence of the critical measure. We shall do this by showing that $G_N(x)=\widetilde{G}_N(x)+D_N(x)$ (recall that $G_N$ was the real part of the field $\mathcal{G}_N$), where $D_N$ converges almost surely to a nice continuous Gaussian field and $\widetilde{G}_N$ is a sequence of  Gaussian fields for which the critical measure can be shown to exist (using results from \cite{JS}). We will make use of the following result.

\begin{theorem}[{\cite[Theorem 1.1]{JS}}]\label{th:js1}
Let $(X_N)$ and $(\widetilde{X}_N)$ be two sequences of H\"older regular Gaussian fields on $[0,1]$ $($that is, the mapping $(x,y)\mapsto \sqrt{\E(X_N(x)-X_N(y))^2}$ is H\"older continuous on $[0,1]^2$ $)$. Assume that the measure $A_Ne^{\widetilde{X}_N(x)-\frac{1}{2}\E \widetilde{X}_N(x)^2}dx$ converges weakly $($that is with respect to the topology of weak convergence of measures$)$ in distribution to an almost surely non-atomic measure $\widetilde{\lambda}$, where $A_N$ is a deterministic scalar sequence. Assume further that the covariances $C_N(x,y)=\E X_N(x)X_N(y)$ and $\widetilde{C}_N(x,y)=\E \widetilde{X}_N(x) \widetilde{X}_N(y)$ satisfy the following conditions$:$ there exists a constant $K\in(0,\infty)$ $($independent of $N)$ such that for all $N\geq 1$, 

\begin{equation*}
\sup_{x,y\in[0,1]}|C_N(x,y)-\widetilde{C}_N(x,y)|\leq K
\end{equation*}

\noindent and for each $\delta>0$

\begin{equation*}
\lim_{N\to\infty}\sup_{|x-y|>\delta}|C_N(x,y)-\widetilde{C}_N(x,y)|=0.
\end{equation*}

Then also $A_ne^{X_N(x)-\frac{1}{2}\E X_N(x)^2}dx$ converges weakly in distribution to $\widetilde{\mu}$.
\end{theorem}

To do this, we thus need a reference approximation for which convergence is known, and a representation of our field which gives us good control on the covariance. Let us first discuss the reference field. For this, we recall a construction from \cite{BM} and make use of results in \cite{DRSV1}. 

\begin{definition}{\rm
Let $W$ denote a white noise on $\R\times [-1/2,3/2]$. For $t\in \R$ and $x\in[0,1]$, write

\begin{equation*}
\widetilde{G}_t(x)=\frac{1}{\sqrt{2}}\int_{-\infty}^t \int_{-1/2}^{3/2}\mathbf{1}\left\lbrace |x-y|\leq \frac{1}{2}\min(e^{-s},1)\right\rbrace e^{s/2}W(ds,dy).
\end{equation*}}
\hfill $\blacksquare$
\end{definition}

The covariance of the field is

\begin{equation*}
\E\left(\widetilde{G}_t(x)\widetilde{G}_t(y)\right)=\begin{cases}
\frac{1}{2}\left[1+t-e^t |x-y|\right], & |x-y|\leq e^{-t}\\
-\frac{1}{2}\log |x-y|,  & e^{-t}\leq |x-y|\leq 1
\end{cases}.
\end{equation*}

\noindent Obviously the above field is H\"older-regular as it is $C^1$. As pointed out in \cite[Remark 3]{DRSV1}, the main results of \cite{DRSV1} apply to the measure $\sqrt{t}e^{2\widetilde{G}_t(x)-2\E \widetilde{G}_t(x)^2}dx$ as well, whence it converges weakly in probability to a non-trivial, and non-atomic random measure, as $t\to \infty$.

Our next task is to approximate our field by one whose covariance we can control. Here our starting point is Proposition \ref{pr:gaussian_appro2}, or rather it's proof. It follows immediately from this that we can write 
$$
\mathcal{G}_N(x)=\widetilde{G}_N(x)+\widetilde{E}_N(x)=\int_1^{\log \mathrm{Li}^{-1}(N+1)}e^{-2\pi i x u}u^{-1/2}dB_u^{\C}+\widetilde{E}_N(x)
$$
 for a nice Gaussian error field $\widetilde{E}_N$. We now want to replace $1/\sqrt{u}$ by something that will allow us to reach the desired covariance in the limit. Let us consider the translation invariant covariance, already alluded to before, that is induced by the function $C(x)=\max(-\log |x|,0)$. Then

\begin{align*}
\widehat{C}(k)&=\int_{-1}^1 e^{-2\pi ikx}\log \frac{1}{|x|}dx
=2\int_0^1 \cos(2\pi kx)\log \frac{1}{x}dx
= \frac{1}{\pi k}\int_0^{2\pi k} \cos y\log \frac{2\pi k}{y}dy\\
\notag &=\frac{1}{\pi k}\int_0^{2\pi k}\frac{\sin y}{y}dy,
\end{align*}

\noindent where in the last step we integrated by parts. This is positive (as it should since it's the Fourier transform of a translation invariant covariance), and as $k\to\infty$, it behaves like $1/(2k)+\mathcal{O}(k^{-2})$. Thus it should be possible to replace $1/\sqrt{u}$ in our field by $\sqrt{2\widehat{C}(s)}$, and this actually works out.

\begin{lemma}
Let 

\begin{align*}
\mathcal{G}_{N,4}(x)&=\int_{1}^{\log \mathrm{Li}^{-1}(N+1)}\sqrt{2\widehat{C}(s)}e^{-2\pi ix s}dB_s^{\C}.
\end{align*}

Then almost surely, $\widetilde{\mathcal{G}}_{N}-\mathcal{G}_{N,4}$ converges uniformly to a random continuous function $F_4$.
\end{lemma}

\begin{proof}
In this case, making use of the same Sobolev estimate as before would lead to a non-summable series, but we still can proceed by employing the following simple lemma.

\begin{lemma}\label{le:frac-sobo}
Let $g:\R\to\C$ be a bounded measurable function with compact support. Let us denote by
$$
F(x):=\int_{\R}e^{-2\pi ixu}g(u)dB^\C_u
$$
the Fourier transform of the $($almost-surely well-defined$)$ compactly supported distribution $gdB^\C_u$.
Then for any $r>1/2$ we have
$$
\E \|F\|_{L^{\infty} [0,1]}^2 \lesssim \int_\R|g(\xi)|^2(1+\xi^2)^rd\xi.
$$
\end{lemma}
\begin{proof}
Let us first note that for, say smooth Schwartz test functions we obtain by Cauchy-Schwarz
\begin{equation*}%\label{eq:upotus1}
\|f\|_{L^\infty(\R)}\lesssim\|\widehat{f}\|_{L^1(\R)}\lesssim \|\widehat f(\xi) (1+\xi^2)^{r/2}\|_{L^2(\R)}\nonumber
\end{equation*}
since $\|(1+(\cdot)^2)^{-r/2}\|_2<\infty$ for $r>1/2$ (actually this yields a proof of a special case of the Sobolev embedding theorem, see e.g. \cite[Theorem 6.2.4]{G}). In order to localize in the case where $f$ is smooth but not compactly supported, we  pick a real-valued and symmetric Schwartz test function $\phi$ with supp$(\phi)\subset[-1,1]$. We demand further that ${\mathcal F}^{-1}\phi(x)\geq 1/2$ on $[0,1]$.
We then observe that by the previous inequality
\begin{equation}\label{eq:upotus2}
\|f\|_{L^\infty[0,1]}\lesssim\|[{\mathcal F}^{-1}\phi ]f\|_{L^\infty(\R)}\lesssim\|(\phi*\widehat f )(\xi) (1+\xi^2)^{r/2}\|_{L^2(\R)}.
\end{equation}
Note next that for any $\xi\in\R$ we may compute using the symmetry of $\phi$
\begin{align*}
\E |(\widehat F*\phi)(\xi)|^2=\E \big|(gdB^\C_u)*\phi(-\xi)\big|^2  &=
\E\int_R\int_\R g(u)\phi(\xi+u)\overline{g(u')}\phi(\xi+ u')dB^\C_ud\overline{B^\C_{u'}} \\ 
\notag &=\int_\R|g(u)|^2\phi^2 (\xi+u)du= \big(|g|^2*\phi^2)(-\xi).
\end{align*}
By combing this with \eqref{eq:upotus2} it follows that
\begin{align*}
\E\|F\|^2_{L^\infty[0,1]}&\lesssim\int_\R(|g|^2*\phi^2)(\xi)(1+\xi^2)^rd\xi= \int_\R|g(\xi)|^2[(1+(\cdot)^2)^r*\phi^2](\xi)d\xi,
\end{align*}
and the claim follows by noting that trivially $[(1+(\cdot)^2)^r*\phi^2](\xi)\lesssim (1+\xi^2)^r.$

\end{proof}

In our case, if we define $F_{N,4}=\widetilde{\mathcal{G}}_{N}-\mathcal{G}_{N,4}$, an application of the above lemma with the choice $r=3/4$ results in the bound (for say $M\leq N$)

\begin{equation*}
\E||F_{N,4}-F_{M,4}||_{L^\infty[0,1]}^2\lesssim \int_{\log \mathrm{Li}^{-1}(M+1)}^{\log \mathrm{Li}^{-1}(N+1)} (1+s^2)^{3/4}\left[\sqrt{\widehat{C}(s)}-\frac{1}{\sqrt{2s}}\right]^2ds
\end{equation*}
Note that
\begin{align*}
\left|\sqrt{\widehat{C}(s)}-\frac{1}{\sqrt{2s}}\right|&=\frac{1}{\sqrt{2s}}\left|\sqrt{\frac{2}{\pi}\int_0^{2\pi s}\frac{\sin y}{y}dy}-1\right|\leq \frac{1}{\sqrt{2s}}\frac{2}{\pi}\int_{2\pi s}^\infty \frac{\sin y}{y}dy=\mathcal{O}(s^{-3/2}),
\end{align*}
where we made use of the fact that $\frac{2}{\pi}\int_0^\infty \frac{\sin y}{y}dy=1$ and the already mentioned asymptotic bound 
$
\int_{s}^\infty \frac{\sin y}{y}dy=\mathcal{O}(s^{-1}). 
$
It follows that 
\begin{equation*}
\E||F_{N,4}-F_{M,4}||_{L^\infty[0,1]}^2\lesssim \int_{\log \mathrm{Li}^{-1}(M+1)}^{\log \mathrm{Li}^{-1}(N+1)} (1+s^2)^{3/4}s^{-3}ds,
\end{equation*}
which is bounded in $N$ and $M$, so we proceed as before. 
\end{proof}

To make use of Theorem \ref{th:js1} and compare $\Re \mathcal{G}_{N,4}$ to $\widetilde{G}_t$, we should see how $N$ and $t$ are related. To do this, let us calculate the variance of $\Re \mathcal{G}_{N,4}$ and require it to be $\frac{1}{2}t+\mathcal{O}(1)$. We have 

\begin{align*}
\E \Re \mathcal{G}_{N,4}(x)^2&=\int_{1}^{\log \mathrm{Li}^{-1}(N+1)}\widehat{C}(s)ds=\frac{1}{2}\int_{1}^{\log \mathrm{Li}^{-1}(N+1)}\frac{1}{s}ds+\int_{1}^{\log \mathrm{Li}^{-1}(N+1)}\mathcal{O}(s^{-2})ds\\
\notag &=\frac{1}{2}\log\log\mathrm{Li}^{-1}(N+1)+\mathcal{O}(1),
\end{align*}

\noindent where we used the expansion of $\widehat{C}(s)$. Thus we should expect that $t=\log\log \mathrm{Li}^{-1}(N+1)$ should give a good estimate for the covariances. Indeed, for $|x-y|\leq 1/\log \mathrm{Li}^{-1}(N+1)$, we have 

\begin{align*}
&\E \Re \mathcal{G}_{N,4}(x)\Re \mathcal{G}_{N,4}(y)=\frac{1}{2}\int_1^{\log\mathrm{Li}^{-1}(N+1)}\frac{1}{s}\cos(s|x-y|)ds+\mathcal{O}(1)\\
\notag &=\frac{1}{2}\int_{|x-y|}^{|x-y|\log \mathrm{Li}^{-1}(N+1)}\frac{1}{s}\cos sds+\mathcal{O}(1)\\
\notag &=\frac{1}{2}\int_{|x-y|}^{|x-y|\log \mathrm{Li}^{-1}(N+1)}\frac{1}{s}ds+\frac{1}{2}\int_{|x-y|}^{|x-y|\log \mathrm{Li}^{-1}(N+1)}\frac{\cos s-1}{s}ds+\mathcal{O}(1)\\
\notag &=\frac{1}{2}\log \log \mathrm{Li}^{-1}(N+1)+\mathcal{O}(1)
\end{align*}

\noindent where the $\mathcal{O}(1)$ terms are uniform in $x,y$. For $|x-y|\geq 1/\log \mathrm{Li}^{-1}(N+1)$, elementary calculations show that

\begin{align*}
\E \Re \mathcal{G}_{N,4}(x)\Re \mathcal{G}_{N,4}(y)&=\frac{1}{2}C(x-y)+\frac{1}{2}\int_{|x-y|[\log\mathrm{Li}^{-1}(N+1)+1]}^\infty \frac{\cos s}{s}ds+\mathit{o}(1),
\end{align*}

\noindent where the $\mathit{o}(1)$ term is uniform in $x,y$. From this we see that if we write $C_N(x,y)=\E \Re \mathcal{G}_{N,4}(x)\Re \mathcal{G}_{N,4}(y)$ and $\widetilde{C}_N(x,y)=\E \widetilde{G}_t(x)\widetilde{G}_t(y)$ with $t=\log\log\mathrm{Li}^{-1}(N+1)$, the conditions on the distances between the covariances in Theorem \ref{th:js1} are satisfied. Let us finally note that all our approximating  fields are smooth, and in particular, they have H\"older covariances.

Before finishing our proof of Theorem \ref{th:critical}, we recall a further result we need from \cite{JS}.

\begin{lemma}[{\cite[Lemma 4.2 (ii)]{JS}}]\label{le:js}
Let $X$ be a H\"older regular Gaussian field on $[0,1]$ and assume that it is independent of the sequence of measures $(\lambda_n)$ on $[0,1]$. If $e^{X}\lambda_n$ converges weakly in distribution, then $\lambda_n$ does as well.
\end{lemma}

Finally we give the

\begin{proof}[Proof of Theorem \ref{th:critical}]
Let us introduce some notation. For $M\geq 0$, let 

\begin{align*}
&\lambda_{\beta_c,M, N}(dx)=\sqrt{\log\log \mathrm{Li}^{-1}(N+1)}e^{\beta_c (\Re \mathcal{G}_{N,4}(x)-\Re \mathcal{G}_{M,4}(x))-\frac{\beta_c^2}{2}\E[ \Re \mathcal{G}_{N,4}(x)^2-\Re \mathcal{G}_{M,4}(x)^2]}dx,
\end{align*}

\noindent where $\mathcal{G}_{0,4}=0$. From Theorem \ref{th:js1} we see that $\lambda_{\beta_c,0,N}$ converges weakly in distribution (to a non-trivial random measure $\lambda_{\beta_c}$) as $N\to\infty$. Then from Lemma \ref{le:js} we see that also $\lambda_{\beta_c,M,N}$ converges weakly in law for any fixed $M\geq 0$. We also note that 

\begin{equation*}
\sqrt{\log\log \mathrm{Li}^{-1}(N+1)}\frac{|\zeta_{N,\rand}(1/2+ix)|^{\beta_c}}{\E|\zeta_{N,\rand}(1/2+ix)|^{\beta_c}}dx=e^{f_N(x)}\lambda_{\beta_c,0,N}(dx)
\end{equation*}

\noindent where $f_N$ is a sequence of  continuous functions converging uniformly almost surely to a continuous function $f$ and, by construction, $f_M$ is independent from $f_N-f_M$ for each $0\leq M<N$. Recall that we want to show that for each non-negative continuous $g:[0,1]\to[0,\infty)$, 

\begin{equation*}
\sqrt{\log\log \mathrm{Li}^{-1}(N+1)}\int_0^1\frac{|\zeta_{N,\rand}(1/2+ix)|^{\beta_c}}{\E|\zeta_{N,\rand}(1/2+ix)|^{\beta_c}}g(x)dx
\end{equation*}

\noindent converges in law to $\lambda_{\beta_c}(e^fg)$. Observe that for any $M\geq 1$

\begin{align*}
\sqrt{\log\log \mathrm{Li}^{-1}(N+1)}& e^{f_M-f_N}\frac{|\zeta_{N,\rand}(1/2+ix)|^{\beta_c}}{\E|\zeta_{N,\rand}(1/2+ix)|^{\beta_c}}dx\\
\notag &=\big(e^{f_M(x) +\beta_c \Re \mathcal{G}_{M,4}(x)-\frac{\beta_c^2}{2}\E\Re \mathcal{G}_{M,4}(x)^2}\big)\lambda_{\beta_c,M, N}(dx).
\end{align*}
 On the right hand side the first factor is a random continuous function, independent of the measure $\lambda_{\beta_c,M, N}(dx)$,
 which in turn converges in distribution as $N\to\infty.$ A simple argument that employs conditioning (i.e. Fubini) then shows  that the full product on right hand side converges in distribution, whence the same is true for the left hand side.
 As $\sup_{N\geq M}\|f_M-f_N\|_{L^\infty [0,1]}\to 0$ in probability as $M\to\infty$, it is then an easy matter to use the asymptotics of $\mathrm{Li}^{-1}(N)$ to verify the claim of Theorem \ref{th:critical}.
\end{proof}

\begin{remark} \label{rem:mesoscopic}{\rm 
We sketch here how a mesoscopic result can be shown for the statistical model that we are considering here. Observe first that  by Lemma \ref{le:covariance_approximation}, we may choose a sequence $\varepsilon_N\to 0^+$ (slower than $1/\log N$) and $\alpha_N\to\infty$ (in fact $\alpha_N=-\frac{1}{2}\log \varepsilon_N$) so that the covariance of $G_N(\varepsilon_N x)$ satisfies
$$
K_{G_N(\varepsilon_N \cdot)}(x,y)= \frac{1}{2}\min( \log(1/|x-y|),\log[\varepsilon_N\log N])+\alpha_N+ {\mathcal O}(1),
$$
and, uniformly outside the diagonal, one even has
$$
K_{G_N(\varepsilon_N \cdot)}(x,y)= \frac{1}{2} \log(1/|x-y|)+\alpha_N +{\mathit o}(1).
$$
On the other hand, we know that our error term $\mathcal{E}_N$ converges uniformly to a bounded continuous function. Thus, in the scaling $x\to \varepsilon_N x$ we may write

\begin{equation*}
\frac{|\nu_{N,\rand}(\varepsilon_N x)|^{\beta}}{\E|\nu_{N,\rand}(\varepsilon_N x)|^{\beta}}dx
=e^{\beta \sqrt{\alpha_N} G_0-\frac{\beta^2}{2}\alpha_N+R+{\mathit o}(1)}\widetilde {\lambda}_{N,\beta}(dx),
\end{equation*}

\noindent where $G_0$ is a fixed standard normal random variable, independent from each $\widetilde \lambda_{ N,\beta }$, $R:=\beta \mathrm{Re}\mathcal{E}(0)-C$ is a random variable, where $C$ is a constant depending on $\beta$ (it comes from the difference between the true value of $\E |\nu_N(x)|^\beta$ and the normalization of the Gaussian multiplicative chaos measure),  and $\widetilde\lambda_{N,\beta}$ is  obtained by exponentiating  a good approximation of a Gaussian field with the strictly logarithmic covariance structure $\frac{1}{2}\log(1/|x-y|)$ on $[0,1].$ In particular,  $\widetilde \lambda_{ N,\beta }$ converges to a standard Gaussian multiplicative chaos on $[0,1]$. A similar statement holds also true in the case $\beta=\beta_c$.}
\hfill $\blacksquare$
\end{remark}

\appendix

\section{The second shifted moment of the zeta function on the critical line}\label{app:appendixa}

In this first appendix we prove Theorem  \ref{th:twopoint}. Before embarking upon proving the auxiliary estimates of  shifted second moments, and suitable truncated versions of them as discussed in the introduction, let us first recall some fundamental growth estimates for the zeta function and record an easy consequence of these.

A classical growth estimate of the zeta function over the critical line $\{\sigma=1/2\}$ states that
\begin{equation}\label{eq:zcgrowth}
\zeta\left(\frac{1}{2}+it\right)=\mathcal{O}(t^{\mu_0}), \quad\textrm{with}\quad \mu_0<1/6.
\end{equation}
A proof of such a bound can be found in \cite[Chapter V]{Titchmarsh}, which is also an excellent introduction to the classical theory of the zeta function. 
Currently the best known bound is  $ \mu_0=13/84+\varepsilon$, due to Bourgain \cite{Bourgain}. The famous Lindel\"of hypothesis declares that \eqref{eq:zcgrowth} is true for any $\mu_0>0$, and would in turn be a consequence of the Riemann hypothesis.
Moreover, for $\sigma\geq \frac{1}{2}$, $t\in \R,$ and arbitrary fixed $\varepsilon>0$, the convexity of the indicator function of Dirichlet series (see e.g. \cite[II.1 Theorem 16]{Tenenbaum}) allows one to interpolate and to obtain a bound of the form
\begin{equation}\label{eq:zcgrowth2}
\zeta\left(\sigma+it\right)=\mathcal{O}\big((1+|t|)^{2\mu_0\max(1-\sigma,\varepsilon)}\big) = \mathcal{O}\big((1+|t|)^{\max(1-\sigma,\varepsilon)/3}\big) ,
\end{equation}
where $\mathcal{O}$ is uniform for $\sigma\geq 1/2$ and $t\in\R.$ The books \cite{Titchmarsh,Ivic,Tenenbaum}, among many others, are excellent references for further properties of the Riemann zeta function.

A simple corollary of this growth estimate we now prove is that the average size of $\zeta$ (and hence the mean in the limit for the random shifts of $\zeta$) over the critical line is $1$ in a strong quantitative sense:
\begin{lemma}\label{le:average}
For any real numbers $A<B$ it holds that
$$
\frac{1}{B-A}\int_A^B\zeta (1/2+it)dt = 1\;+ \;\mathcal{O}\big(\frac{1+|A|^{1/6}+|B|^{1/6}}{B-A}\big). 
$$
\end{lemma}
\begin{proof} We may obviously assume that $|A|,|B|\geq 1.$ We integrate $\zeta$ over a rectangular path with vertices $1/2+iA, 1/2+iB, 2+iB, 2+iA$ and estimate the integrals over the horizontal paths using \eqref{eq:zcgrowth2} to obtain
\begin{align*}
\notag \int_A^B\zeta (1/2+it)dt \;=\; \mathcal{O}(1+|A|^{1/6}+|B|^{1/6})+\int_{A}^{B}\zeta(2+it)dt,
\end{align*}
where the possible effect of the pole at $1$ is covered by the $\mathcal{O}(1)$-term. The claim follows since
$$
\int_{2+iA}^{2+iB}\zeta(s)ds=(B-A)+ \sum_{n=2}^\infty\frac{1}{i\log n}(n^{-2-iA}-n^{-2-iB})=B-A+\mathcal{O}(1).
$$
\end {proof}

We also point out the following trivial fact: 
\begin{equation}\label{le:zetanbound}
\sup_{\sigma\geq \frac{1}{2},t\in \R}|\zeta_N(\sigma+it)|\leq C_N<\infty\quad \textrm{for any}\quad N\geq 1.
\end{equation}

Let us  start with some  useful definitions. Recall that  $\N_N$ consists  of positive integers $n\geq 1$ whose prime factors do not exceed $p_N$.
\begin{definition}\label{def:divisors}{\rm
For $q\in \C$ and $n,N\in \Z_+$, let 
\begin{equation*}
\sigma_q(n)=\sum_{d: d|n}d^q \quad\textrm{and}\quad \sigma_q(n;N)=\sum_{d\in \N_N: d|n}d^q\notag .
\end{equation*}
For $T>0$  the corresponding divisor sums are given by
\begin{equation*}
D_q(T)=\sum_{n\leq T}\sigma_{-q}(n)\quad\textrm{and}\quad D_q(T;N)=\sum_{n\leq T}\sigma_{-q}(n;N) 
\notag
\end{equation*}}
\hfill $\blacksquare$
\end{definition}
\noindent The way the divisor functions $\sigma_{-q}(k)$ come into play is by the fact  that for $\mathrm{Re}(s)>1$ and $\mathrm{Re}(s+z)>1$
\begin{align*}%\label{eq:tulo1}
\zeta(s)\zeta(s+z)&=\sum_{k,m=1}^\infty (km)^{-s} m^{-z}
 =\sum_{n=1}^\infty n^{-s}\sum_{m: m|n} m^{-z}
=\sum_{n=1}^\infty \frac{\sigma_{-z}(n)}{n^s},
\end{align*}
and in a similar way
\begin{equation}\label{eq:tulo2}
\zeta(s)\zeta_N(s+z)=\sum_{n=1}^\infty \frac{\sigma_{-z}(n;N)}{n^s}.
\end{equation}

In what follows, the constant $\varepsilon>0$, if not otherwise stated,  stands for a positive quantity that can be taken as small as we wish, with the possible cost of increasing the implicit multiplicative but uninteresting constants.

We shall repeatedly make use of the following classical estimate for the standard divisor function  $d(n):=\sigma_0(n)$ (see e.g. \cite[Theorem 315]{HW})
\begin{equation}\label{le:dbound}
d(n)=\sigma_0(n)=\mathcal{O}(n^\varepsilon).
\end{equation}
Obviously $\sigma_{-z}(n)$ satisfies the same estimate uniformly for $z\in i\R$.

The  following result encodes the necessary estimates for $D_z(T)$ and $D_z(T;N)$ when $z$ is purely imaginary.

\begin{proposition}\label{prop:divsum}
For $a\in\R$ and $T\geq 1$ with $T-1/2\in\Z$ we have
\begin{equation*}%\label{eq:one}
\frac{1}{T}D_{ia}(T)=\zeta(1+{ia})+\zeta(1-{ia})T^{-{ia}}\frac{1}{1-{ia}}\; +\; \mathcal{O}(T^{-1/3})+\mathcal{O}(T^{-5/12}|a|^{1/6}),
\end{equation*}
In a similar vein,
\begin{equation*}%\label{eq:two}
\frac{1}{T}D_{{ia}}(T;N)=\zeta_N(1+{ia})+\mathcal{O}(T^{-1/3}+T^{-5/12}|{a}|^{1/6}),
\end{equation*}
$($where the error term  is not claimed to be uniform in $N\geq 1)$.
\end{proposition}

\begin{proof}  
By the estimate \eqref{le:dbound} the abscissa of absolute convergence of the Dirichlet series $\sum_{k=1}^\infty \frac{\sigma_{-ia}(k)}{k^s}$ is 1. 
For any $T\geq 2$ set $\varepsilon_T:=(\log T)^{-1} $and note that  by \eqref{le:dbound} and  the effective Perron formula (see e.g. \cite[II.2, Theorem 2]{Tenenbaum}) we may write 
\begin{equation*}
D_{{ia}}(T)=\frac{1}{2\pi i}\int_{1+\varepsilon_T-iT^{1/2}}^{1+\varepsilon_T+iT^{1/2}}\zeta(s)\zeta(s+{ia})T^s\frac{ds}{s}+\mathcal{O}\left(\sum_{n=1}^\infty |\sigma_{-{ia}}(n)|\frac{\left(\frac{T}{n}\right)^{1+\varepsilon_T}}{T^{1/2}|\log\frac{T}{n}|}\right).
\end{equation*}
We then split the sum in a standard way into three parts: $n\leq T/2$, $T/2<n\leq 2T$, and $n>2T$. In the first domain, the logarithm is bounded from below by $\log(2)$, and $\sigma_{-{ia}}(n)$  by $n^\varepsilon$. One infers that the contribution to the sum from the first domain is of order $T^{1/2+\varepsilon}$. For the second sum, we estimate the logarithm by 
$
\frac{1}{\log T-\log n}=\mathcal{O}\left(\frac{T}{T-n}\right).
$
\noindent Since $T$ is an half an odd integer we may estimate the sum of these terms by a harmonic series. Invoking again the bound $\sigma_{-z}(n)=\mathcal{O}(n^\varepsilon)$, we see that the contribution from this sum is of order $T^{1/2+\varepsilon}\log(T)$, i.e  of the form $\mathcal{O}(T^{1/2+\varepsilon})$.

For the last sum, we bound the logarithm from below by a constant, so the sum is again of order 
\begin{equation*}\notag
T^{1/2+\varepsilon_T}\sum_{n>2T}d(n)n^{-1-\varepsilon_T}\lesssim T^{1/2}\zeta(1+\varepsilon_T)^2=\mathcal{O}(T^{1/2}\log^2T),
\end{equation*}
where we used the fact that $\zeta(1+s)$ has a simple pole at the origin. We conclude that 
\begin{equation*}
D_{{ia}}(T)=\frac{1}{2\pi i }\int_{1+\varepsilon_T-iT^{1/2}}^{1+\varepsilon_T+iT^{1/2}}\zeta(s)\zeta(s+{ia})T^s\frac{ds}{s}+\mathcal{O}(T^{1/2+\varepsilon}),
\end{equation*}
with a uniform error term over $a\in \R$. The argument for $D_{{ia}}(T;N)$ is essentially identical as one starts from  \eqref{eq:tulo2} and the estimate $|\sigma_{-{ia}}(n;N)|\leq d(n)$ for the first two sums and for the last one, one gets a bound related to $\zeta_N(1+\varepsilon_T)\zeta(1+\varepsilon_T)$. Recalling \eqref{le:zetanbound}   we thus have uniformly in $z\in i\R$
\begin{equation*}
D_{{ia}}(T;N)=\frac{1}{2\pi i }\int_{1+\varepsilon_T-iT^{1/2}}^{1+\varepsilon_T+iT^{1/2}}\zeta(s)\zeta_N(s+{ia})T^s\frac{ds}{s}+\mathcal{O}(T^{1/2+\varepsilon}).
\end{equation*}
Actually, the estimate is uniform also in $N$, but we do not need this fact later on. 

We next move the integration to the critical line and consider first the case  $|a|<T^{1/2}-1.$ By applying residue calculus in the relevant domain and noting that the integrand has poles at $s=1$ and $s=1-{ia}$ we obtain
\begin{align}\label{eq:dint}
D_{{ia}}(T)&=\zeta(1+{ia})T+\zeta(1-{ia})T^{1-{ia}}\frac{1}{1-{ia}}+\mathcal{O}(T^{1/2+\varepsilon})\\
\notag  -\;
\frac{1}{2\pi i }&\bigg[\int_{1+\varepsilon_T+iT^{1/2}}^{\frac{1}{2}+iT^{1/2}}+\int_{\frac{1}{2}+iT^{1/2}}^{\frac{1}{2}-iT^{1/2}}+\int_{\frac{1}{2}-iT^{1/2}}^{1+\varepsilon_T-iT^{1/2}}\bigg] \zeta(s)\zeta(s+{ia})T^s\frac{ds}{s},
\end{align}
and the corresponding formula for the integral involving the truncated Euler product reads
\begin{align}\label{eq:dnint}
D_{z}(T;N)&=\zeta_N(1+{ia})T+\mathcal{O}(T^{1/2+\varepsilon})\\
\notag  -\; \frac{1}{2\pi i }\bigg[\int_{1+\varepsilon_T+iT^{1/2}}^{\frac{1}{2}+iT^{1/2}}&+\int_{\frac{1}{2}+iT^{1/2}}^{\frac{1}{2}-iT^{1/2}}+\int_{\frac{1}{2}-iT^{1/2}}^{1+\varepsilon_T-iT^{1/2}}\bigg] \zeta(s)\zeta_N(s+{ia})T^s\frac{ds}{s}.
\end{align}
Consider first the horizontal integrals in \eqref{eq:dint}. We see by  \eqref{eq:zcgrowth2} that these are bounded by 
\begin{align}\label{eq:error3}
&\int_{\frac{1}{2}}^{1+\varepsilon_T}\big(T^{\max(1-\sigma,\varepsilon)/6}(T^{\max(1-\sigma,\varepsilon)/6}+|a|^{\max(1-\sigma,\varepsilon)/3})\big)\frac{T^{\sigma}}{T^{1/2}}d\sigma=\mathcal{O}\left(( |a|^{1/6}+1)T^{1/2+\varepsilon}\right).
\end{align}
Next we note that according to \eqref{eq:zcgrowth}, the vertical integral in \eqref{eq:dint} is of order 
\begin{equation}\label{eq:error4}
\mathcal{O}\Big(T^{1/2}\int_{1}^{T^{1/2}}(t^{1/3}+t^{1/6} |a|^{1/6})\frac{dt}{t}\Big)=\mathcal{O}\big(T^{1/2} (T^{1/6}+ T^{1/12}|a|^{1/6})\big).
\end{equation}

By combining  \eqref{eq:dint} with  the estimates \eqref{eq:error3} and  \eqref{eq:error4} we finally have
\begin{equation*}\notag
\frac{1}{T}D_{{ia}}(T)=\zeta(1+{ia})+\zeta(1-{ia})\frac{T^{-{ia}}}{1-{ia}}+\mathcal{O}(T^{-1/3})+\mathcal{O}(T^{-5/12}|a|^{1/6}).
\end{equation*}
If  $|a|>T^{1/2}+1$ we argue as before with the only change that we do not obtain the second term at all, but in this case this difference is easily immersed in the stated error term in the actual proposition. Finally, in case $||a|-T^{1/2}|\leq 1$ we use instead the limits $1+\varepsilon_T\pm 2iT^{1/2}$ in the Perron formula and proceed as before.

Establishing the analogous estimate for $T^{-1}D_{{ia}}(T;N)$ is even simpler as we can  repeat our steps starting from \eqref{eq:dnint} and this time apply  \eqref{le:zetanbound} to bound $\zeta_N(s+{ia})$ uniformly by $C_N$. 
\end{proof}

Let us recall the functional equation of $\zeta (s)$ in the form
\begin{equation*}%\label{eq:func}
\zeta(1-s)=\chi(1-s)\zeta(s),
\end{equation*}
where
\begin{equation*}
\chi(1-s):=\frac{2}{(2\pi)^s}\Gamma(s)\cos \frac{s\pi}{2}, \quad s\in\C.
\end{equation*}
The following two simple auxiliary results are found in the classical monograph  of Titchmarsh \cite[(7.4.2) and (7.4.3)]{Titchmarsh}:
\begin{lemma}\label{le:chiint}
If $n<T/2\pi$, 
\begin{equation*}
\frac{1}{2\pi i}\int_{\frac{1}{2}-iT}^{\frac{1}{2}+iT}\chi(1-s)n^{-s}ds=2+\mathcal{O}\left(\frac{1}{\sqrt{n}\log \frac{T}{2\pi n}}\right)+\mathcal{O}\left(\frac{\log T}{\sqrt{n}}\right)
\end{equation*}
and if $n> T/2\pi$ and $c>1/2$, then 
\begin{equation*}
\frac{1}{2\pi i }\int_{c-iT}^{c+iT}\chi(1-s)n^{-s}ds=\mathcal{O}\left(\frac{T^{c-1/2}}{n^c\log \frac{2\pi n}{T}}\right)+\mathcal{O}\left(\frac{T^{c-1/2}}{n^c}\right).
\end{equation*}
\end{lemma}

As stated before, the following proposition is essentially due to Ingham and Bettin.
\begin{proposition}\label{pr:bettin} Assume that $a\in\R$ and  $T\geq 1.$ Then
\begin{align*}%\label{eq:prop1}
\frac{1}{2iT}\int_{\frac{1}{2}-iT}^{\frac{1}{2}+iT}\chi(1-s)\zeta(s+ia)\zeta(s)ds&=\frac{2\pi}{T}D_{ia}(T/2\pi) \;+\; \mathcal{O}\Big(T^{-1/2+\varepsilon}\big( 1+ |a|^{1/6}\big)\Big)\\
&= \zeta(1+ia)+\zeta(1-ia)\Big(\frac{T}{2\pi}\Big)^{-ia}\frac{1}{1-ia}\\
& \quad  \;+\; \mathcal{O}\big(T^{-1/3}(1+|a|^{1/6})\big),
\end{align*}
where the right hand side must be understood as the limit $a\to 0$ in case $a=0$,
and 
\begin{align*}%\label{eq:prop2}
\frac{1}{2iT}\int_{\frac{1}{2}-iT}^{\frac{1}{2}+iT}\chi(1-s)\zeta_N(s+ia)\zeta(s)ds&= \frac{2\pi}{T}D_{i(x-y)}(T/2\pi;N)\\
&\quad  +\mathcal{O}\Big(T^{-1/2+\varepsilon}\big( 1+ |a|^{1/6}\big)\Big))\Big)\\
  &=\zeta_N(1+ia) \;+\; \mathcal{O}\big(T^{-1/3}(1+|a|^{1/6})\big)
\end{align*}
\end{proposition}
\begin{proof}
We may clearly assume that $a\not=0$. Moreover, we may assume   that $|T-|a||\geq 1$,  and  $1/2+T/2\pi\in\Z$ by slightly modifying the value of $T$ if needed. The adjustment  of $T$ by a uniformly bounded amount is justified by the estimate \eqref{eq:stirling} below and the growth bounds \eqref{eq:zcgrowth} and \eqref{eq:zcgrowth2}. Let us write 
\begin{equation*}\notag
I^{(1)}(a):=\; \frac{1}{2iT}\int_{\frac{1}{2}-iT}^{\frac{1}{2}+iT}\chi(1-s)\zeta(s+ia)\zeta(s)ds=I_1^{(1)}(a)+I_2^{(1)}(a),
\end{equation*}
where
\begin{equation*}\notag
I_1^{(1)}(a)\; := \; \frac{1}{2iT}\int_{\frac{1}{2}-iT}^{\frac{1}{2}+iT}\chi(1-s)\sum_{1\leq n\leq T/2\pi}\frac{\sigma_{-ia}(n)}{n^{s}}ds
\end{equation*}
and $I_2^{(1)}(a)=I^{(1)}(a)-I_1^{(1)}(a)$. Similarly we write 
\begin{equation*}\notag
I^{(2)}(a)\; :=\; \frac{1}{2iT}\int_{\frac{1}{2}-iT}^{\frac{1}{2}+iT}\chi(1-s)\zeta_N(s+ia)\zeta(s)ds
\; =\; I_1^{(2)}(a)+I_2^{(2)}(a)
\end{equation*}
where
\begin{equation*}\notag
I_1^{(2)}(a)\; :=\; \frac{1}{2iT}\int_{\frac{1}{2}-iT}^{\frac{1}{2}+iT}\chi(1-s)\sum_{1\leq n\leq T/2\pi}\frac{\sigma_{-ia}(n;N)}{n^{s}}ds
\end{equation*}
and $I_2^{(2)}(a)=I^{(2)}(a)-I_1^{(2)}(a)$. Using Lemma \ref{le:chiint}, and \eqref{le:dbound}
we obtain
\begin{align}\label{eq:error10}
I_1^{(1)}(a)=\;&\frac{2\pi}{T}\sum_{1\leq n\leq T/2\pi}\sigma_{-ia}(n)
 +\frac{1}{T}\sum_{1\leq n\leq T/2\pi}\mathcal{O}\left(\frac{n^\varepsilon}{\sqrt{n}\log \frac{T}{2\pi n}}+\frac{n^\varepsilon\log T }{\sqrt{n}}\right)\\
\notag =\;&\frac{2\pi}{T}D_{ia}(T/2\pi)+\frac{1}{T}\mathcal{O}\left(\sum_{1\leq n\leq T/4\pi}n^{\varepsilon-\frac{1}{2}}\right)\\
\notag&\quad +\frac{1}{T}\mathcal{O}\left(\sum_{T/4\pi\leq n\leq T/2\pi}n^{\varepsilon-\frac{1}{2}}\frac{T}{\frac{T}{2\pi}-n}\right)+\frac{\log T}{T}\mathcal{O}\left(\sum_{1\leq n\leq T/2\pi}n^{\varepsilon-\frac{1}{2}}\right)\\
\notag =\;&\frac{2\pi}{T}D_{ia}(T/2\pi)+\mathcal{O}(T^{\varepsilon-1/2}),
\end{align}
where the error is uniform in $a\in\R$ and we used the fact that $T/2\pi$ is half an odd integer. With an identical argument (using $\sigma_{-ia}(n;N)=\mathcal{O}(n^\varepsilon)$ uniformly in $a$) we see that 
\begin{equation*}
I_1^{(2)}(a)=\frac{2\pi}{T}D_{ia}\left(\frac{T}{2\pi};N\right)\; +\;\mathcal{O}(T^{\varepsilon-1/2}).
\end{equation*}

In turn,  in case $|a|\leq T-1$ we obtain for the $I^{(1)}_2$-term by residue calculus  (for arbitrary $c>1$)
\begin{align}\label{eq:i21}
I_2^{(1)}(a)&=\frac{1}{2i T}\left[\int_{\frac{1}{2}-iT}^{c-iT}+\int_{c+iT}^{\frac{1}{2}+iT}\right]\chi(1-s)\Bigg(\zeta(s+ia)\zeta(s)-\sum_{1\leq n\leq T/2\pi}\frac{\sigma_{-ia}(n)}{n^s}\Bigg)ds\\
\notag &\qquad +\frac{1}{2iT}\int_{c-iT}^{c+iT}\chi(1-s)\sum_{n>T/2\pi}\frac{\sigma_{-ia}(n)}{n^s}ds-\frac{\pi}{T}\zeta(1-ia)\chi(ia).
\end{align}
The integrand in the second term on the right hand side comes simply from the fact that on this integration contour, we can write $\zeta(s)\zeta(s+ia)=\sum_{n=1}^\infty\sigma_{-ia}(n)n^{-s}$. The last term comes from the fact that inside the domain bounded by the contours, $\chi(1-s)\zeta(s)\zeta(s+ia)$ has only a single pole at $s=1-ia$. The last term  appears only in the case  $|a|\leq T-1$ and then it is of the order $T^{-1}|a|^{1/2+\varepsilon}$ i.e. of order $T^{-1/2+\varepsilon}$. Here we used the growth estimate  $\chi(ia) = \mathcal{O}(|a|^{1/2})$ which follows  from the more general estimate
\begin{equation}\label{eq:stirling}
\chi(1-(\sigma\pm iT))=\mathcal{O}(T^{\sigma-\frac{1}{2}}).
\end{equation}
This works uniformly with respect to $\sigma\in [1/2,1]$ and $|T|\geq 1,$ and is an easy consequence of Stirling's formula. 

We then consider the horizontal integrals. 
Using \eqref{eq:zcgrowth2}, and the fact that we are safely avoiding the pole of the second factor, we have on the other hand that on these contours 
\begin{align*}
\zeta(\sigma\pm iT)\zeta(\sigma\pm iT+ia)=\mathcal{O}\Big((T^{\max(1-\sigma,\varepsilon)/3})(T^{\max(1-\sigma,\varepsilon)/3} +|a|^{\max(1-\sigma,\varepsilon)/3})\Big)
\end{align*}
 and recalling that $\sigma_{-ia}(n)=\mathcal{O}(n^\varepsilon)$, we get 
\begin{align*}
\sum_{1\leq n\leq T/2\pi}\sigma_{-ia}(n)n^{-\sigma\mp iT}
= \mathcal{O}(T^{1+\varepsilon-\sigma})
\end{align*}
Putting everything together, with the choice $c=1+2\varepsilon$ the total contribution from the horizontal integrals and the residue term in \eqref{eq:i21} is of the order 
$$
T^{-1/2+\varepsilon}\big( 1+ |a|^{1/6}\big).
$$

For the vertical integral, by the choice of $c$  we are able pull the sum out of the integral due to the bound $\sigma_{-ia}(n)=\mathcal{O}(n^\varepsilon)$ and we can then make use of Lemma \ref{le:chiint}  and our assumption that $T/2\pi$ is half an odd integer to find that the contribution from the term with the vertical integral is 
less than
\begin{align*}
\frac{1}{T}&\sum_{n>T/2\pi}n^\varepsilon \left(\frac{T^{2\varepsilon+1/2}}{n^{1+2\varepsilon}}+\frac{T^{2\varepsilon+1/2}}{n^{1+2\varepsilon}\log \frac{2\pi n}{T}}\right)\\
\notag &=\mathcal{O}(T^{-1/2+2\varepsilon})+\mathcal{O}\left(T^{-1/2+\varepsilon}\sum_{\frac{T}{2\pi}<n\leq \frac{T}{\pi}}\frac{1}{n-\frac{T}{2\pi}}\right)+\mathcal{O}\left(T^{2\varepsilon-1/2}\sum_{n>\frac{T}{\pi}}n^{-1-\varepsilon}\right)\\
\notag &=\mathcal{O}(T^{-1/2+2\varepsilon}),
\end{align*}
uniformly in $a\in\R.$

Put together, we have established  that
\begin{equation*}%\label{eq:I2}
I_2^{(1)}(a)=\mathcal{O}\Big(T^{-1/2+\varepsilon}\big( 1+ |a|^{1/6}\big)\Big).
\end{equation*}
By combining  the above estimate with \eqref{eq:error10} and  recalling Proposition \ref{prop:divsum} we see that 
\begin{align*}
I^{(1)}(a) &\; = \; \frac{2\pi}{T}D_{ia}(T/2\pi) \;+\; \mathcal{O}\Big(T^{-1/2+\varepsilon}\big( 1+ |a|^{1/6}\big)\Big).\\
&\; = \;  \zeta(1+ia)+\zeta(1-ia)\Big(\frac{T}{2\pi}\Big)^{-ia}\frac{1}{1-ia} \;+\; \mathcal{O}\big(T^{-1/3}(1+|a|^{1/6})\big),
\end{align*}
as was to be shown.

For $I_2^{(2)}(a)$ the argument is very similar. The main differences are that there is no pole for the integrand since $\zeta_N(s+ia)$ is analytic in the relevant domain, and   we also get better bounds on the different contours since $\zeta_N(s+ia)$ is uniformly bounded in the relevant domain. We use the bound $\sigma_{-ia}(n;N)=\mathcal{O}(n^\varepsilon)$ as for the $I_2^{(1)}(a)$ term, and a repetition of the whole argument produces analogous estimates. Finally, by again combining with Proposition \ref{prop:divsum}  we obtain the claim for $I_2^{(2)}(a)$.
\end{proof}

We note that   
$\zeta$ has the Laurent expansion
\begin{equation}\label{eq:laurent}
\zeta (1+s)=s^{-1}+\gamma_0+\mathcal{O}(s)
\end{equation}
for small $s$, where $\gamma_0$ is the Euler-Mascheroni constant, see \cite[equation (2.1.16)]{Titchmarsh}. Hence, by letting $a\to 0$ in the first part of the previous proposition and using the fact that $\zeta(1/2+it)= \overline{\zeta(1/2-it)}$
 we deduce for $T\geq 2$ the classical mean square result for the zeta-function in the form
\begin{equation*}%\label{eq:classic}
\frac{1}{T}\int_0^T|\zeta(1/2+it)|^2dt\; = \; \log (T/2\pi)  +(2\gamma_0-1)+\mathcal{O}(T^{-1/3}),
\end{equation*}
where the error estimate could of course be improved, but this is not needed for our purposes.

We are finally  ready to prove Theorem \ref{th:twopoint}.

\begin{proof}[Proof of Theorem \ref{th:twopoint}]
$V_T^{(1)}$ and $V_T^{(2)}$ are handled in an almost identical way so let us discuss them first. We divide the considerations into two cases.

\smallskip

\noindent \emph{\underline{Case 1: $\max(x,y)> T$}.\, \,}
This case can be disposed of by rough size estimates. Namely, immediately from the bound \eqref{eq:zcgrowth} we see that 
\begin{equation}\label{eq:virhe11}
|V_T^{(1)}(x,y)|\lesssim (T+|x|+|y|)^{2/6}\lesssim (|x|+|y|)^{1/3}\lesssim T^{-1/12}(|x|+|y|)^{5/12}.
\end{equation}
Let us then assume that $|x-y|\leq 1$. Then an application of \eqref{eq:laurent} and the $\log(T/2\pi)$-Lipschitz property of the map $u\mapsto (T/2\pi)^{iu}$ yield that
\begin{align*}
\zeta(1+i(x-y))+\frac{\zeta(1-i(x-y))}{1-i(x-y)}\left(\frac{T}{2\pi}\right)^{i(x-y)}=&(i(x-y))^{-1}\Big(1-\left(\frac{T}{2\pi}\right)^{i(x-y)} \Big)+\mathcal{O}(1) \\
=&\mathcal{O}(1) +\mathcal{O}(\log (T))\\
=& \mathcal{O}(T^{-1/12}(|x|+|y|)^{5/12}).
\end{align*}
On the other hand, if $|x-y| > 1$, the same quantity can obviously be estimated by $\mathcal{O}(\log (|x|+|y|))=  \mathcal{O}(T^{-1/12}(|x|+|y|)^{5/12})$ according to the assumption $\max(x,y)>T$.  Both these bounds and the one in \eqref{eq:virhe11} can hence be subsumed into the desired error estimate of $V_T^{(1)}$ on the right hand side of Theorem \ref{th:twopoint}. Even easier bounds show that the same is  true in case of $V_T^{(2)}$.

\smallskip

\noindent \emph{\underline{Case 2: $\max(x,y)\leq T$}.\quad}
In this case we first  extend  the range on integration  in the definition of $V_T^{(1)}$ to $[-T,T]$, and for that end we write
\begin{eqnarray*}%\label{eq:virhe1}
&&\frac{1}{T}\int_{-T}^0\zeta\left(\frac{1}{2}+it+ix\right)\zeta\left(\frac{1}{2}-it-iy\right)dt\notag\\
\notag &=&\frac{1}{T}\int_{0}^T\zeta\left(\frac{1}{2}-it+ix\right)\zeta\left(\frac{1}{2}+it-iy\right)dt\\
\notag &=&\frac{1}{T}\int_{-(x+y)}^{T-(x+y)}\zeta\left(\frac{1}{2}-it-iy\right)\zeta\left(\frac{1}{2}+it+ix\right)dt\\
\notag &=&\frac{1}{T}\int_{0}^{T}\zeta\left(\frac{1}{2}-it-iy\right)\zeta\left(\frac{1}{2}+it+ix\right)dt+ \; E (x,y,T).
\end{eqnarray*} 
By crudely bounding the integrals over the segments of length $|x+y|$ and employing again  \eqref{eq:zcgrowth} we obtain an error estimate of the order
\begin{align*}%\label{eq:virhe0} 
E(x,y,T)\; =\; &\mathcal{O}\Big(T^{-1} (|x|+|y|)(T^{1/6}+|x|^{1/6}+|y|^{1/6})^2\Big)
=\mathcal{O}\Big( T^{-1}(|x|+|y|)T^{1/3}\Big)\\
\notag =\; &\mathcal{O}\Big( T^{-1/12}\big(1+|x|^{5/12}+|y|^{5/12})\Big)\\
=:&\;\mathcal{O}_1(x,y,T),
\end{align*}
where we made essential use of the condition  $\max(x,y)\leq T$. 
By making another shift in the integration range and using the functional equation we obtain analogously
\begin{eqnarray}
\label{eq:eka1}V_T^{(1)}(x,y)
&=&\frac{1}{2iT}\int_{\frac{1}{2}-iT}^{\frac{1}{2}+iT}\zeta(s+ix)\zeta(1-(s+iy))ds\; +\; \mathcal{O}_1 (x,y,T)\\
\notag &=&\frac{1}{2iT}\int_{\frac{1}{2}-iT}^{\frac{1}{2}+iT}\zeta(s+i(x-y))\zeta(1-s)ds\; +\; \mathcal{O}_1 (x,y,T)\\
\notag &=&\frac{1}{2iT}\int_{\frac{1}{2}-iT}^{\frac{1}{2}+iT}\chi(1-s)\zeta(s+i(x-y))\zeta(s)ds\; +\; \mathcal{O}_1 (x,y,T).
\end{eqnarray}
In a similar vein,
\begin{align}\label{eq:eka2}
V_{T}^{(2)}(x,y)
=\frac{1}{Ti}\int_{\frac{1}{2}-iT}^{\frac{1}{2}+iT}\chi(1-s)\zeta_N(s+i(x-y))\zeta(s)ds\; +\; E_2(x,y,T)
\end{align}
The first two claims of the proposition now follow  in view of \eqref{eq:eka1}, \eqref{eq:eka2}
and Proposition \ref{pr:bettin}.

We still need to study the $V_{T,N}^{(3)}(x,y)$ term. Fix $N$, and employ the absolute  convergence (uniform over $t,x,y\in\R$) of the relevant sums to pick $n_0$ such that 
\begin{align*}%\label{eq:zetanapprox}
\sup_{x,y,t\in\R}&\bigg|\zeta_N\left(\frac{1}{2}+it+ix\right)\zeta_N\left(\frac{1}{2}-it-iy\right)-\prod_{k=1}^N \sum_{j_k,l_k=0}^{n_0} p_k^{-\frac{1}{2}(j_k+l_k)-i(xj_k-yl_k)-it(j_k-l_k)}\bigg|<\frac{\varepsilon}{4}
\end{align*}
and
\begin{equation*}%\label{eq:zetanapprox2}
\sup_{x,y\in\R }\bigg|\prod_{k=1}^N \sum_{j_k=0}^{\infty} p_k^{-(1+i(x-y))j_k}
-\prod_{k=1}^N \sum_{j_k=0}^{n_0} p_k^{-(1+i(x-y))j_k}\bigg|<\frac{\varepsilon}{2}.
\end{equation*}
Next select $T_0=T_0(n,N,\varepsilon)$ so that for $T\geq T_0$
\begin{align*}%\label{eq:ero1}
\sup_{x,y\in[0,1]}&\bigg|\frac{1}{T}\int_0^T\zeta_N\left(\frac{1}{2}+it+ix\right)\zeta_N\left(\frac{1}{2}-it-iy\right)dt
 \; -\prod_{k=1}^N \sum_{j_k=0}^{n_0} p_k^{-(1+i(x-y))j_k}\bigg|< &\frac{\varepsilon}{2}.
\end{align*}
This is possible because of \eqref{eq:bohr2}. Put together, as $\varepsilon$ was arbitrary, the inequalities clearly imply that
\begin{equation*}\notag
\lim_{T\to\infty}\frac{1}{T}\int_0^T\zeta_N\left(\frac{1}{2}+it+ix\right)\zeta_N\left(\frac{1}{2}-it-iy\right)dt=\zeta_N(1+i(x-y)),
\end{equation*}
uniformly over $x,y\in\R$, which was precisely the claim.
\end{proof}

\section{Proof of the Gaussian approximation result}\label{app:appendixb}

In this appendix we provide a self contained proof of  Proposition \ref{le2} for the reader' convenience. We commence with by considering some general facts about coupling of random variables and then we apply them to Gaussian approximation. Basically all the ingredients here are well-known, and we have no need to strive for optimal bounds.

We again will make use of the notion of the Wasserstein distance, though now the appropriate notion is the  1-Wasserstein distance. Assume that $\mu$ and $\nu$ are Borel probability distributions on a separable metric space $(X,d)$. Then
$$
\wass_1(\mu,\nu)_X\; :=\; \inf_{(U,V)} \E  d(U,V),
$$
where  $U,V$ are random variables on a common probability space taking values in $X$ so that $U\sim \mu$ and $V\sim \nu$. We start with a simple observation.

\begin{lemma}\label{le11}
In the above situation one has 
$$
\wass_1(\mu,\nu)_X\leq \inf_{R>0,\; x_0\in X} \Big(4R|\mu-\nu|(B(x_0;R))+32\int_{R/2}^\infty |\mu-\nu|(B(x_0,r)^c)dr.
\Big)$$
\end{lemma}
\begin{proof}
Observe that 
$$
\beta:=\mu-(\mu-\nu)_+=\nu-(\nu-\mu)_+\geq 0.
$$
and define the measure $\beta_\Delta$ on $X\times X$ by $\beta_\Delta (A)=\beta (\{ x\; :\; (x,x)\in A \})$ and note that the measure 
$$
\beta_\Delta +\frac{2}{\|\mu-\nu\|_{TV}}(\mu-\nu)_+\times (\nu-\mu)_+
$$
has the right marginals since $\mu$ and $\nu$ are probability measures so  $\|(\mu-\nu)_+\|_{TV} =\|(\nu-\mu)_+\|_{TV}=(1/2) \|\mu-\nu\|_{TV}$, and both of the marginals of $\beta_\Delta$ are simply $\beta$. As $\beta_{\Delta}$ lives on the diagonal, it follows that
\begin{align}\label{40}
\wass_1(\mu,\nu)_X&\leq \frac{2}{\|\mu-\nu\|_{TV}}\int_{X\times X}d(x,y)\; (\mu-\nu)_+\times (\nu-\mu)_+ (dx\times dy)\\
&\leq \frac{2}{\|\mu-\nu\|_{TV}}\int_{X\times X}d(x,y)\; |\mu-\nu|\times |\nu-\mu| (dx\times dy).\nonumber
\end{align}
Fix now some $x_0\in X$ and $R>0$ and split the integral into ones over $B(x_0,R)\times B(x_0,R)$ and its complement. Thus
\begin{align}\label{41}
\frac{2}{\|\mu-\nu\|_{TV}}&\int_{B(x_0,R)\times B(x_0,R)}d(x,y)\; |\mu-\nu|\times |\nu-\mu| (dx\times dy)\\
&\leq {2\cdot 2R}\big({\|\mu-\nu\|_{TV}}\big)^{-1} |\mu-\nu|\times |\nu-\mu| (B(x_0,R)\times B(x_0,R))\nonumber\\
\notag &\leq 4R  |\nu-\mu| (B(x_0,R))
\end{align}
By symmetry, the integral over the rest has the upper bound
\begin{eqnarray}\label{42}
&&\frac{4}{\|\mu-\nu\|_{TV}}\int_{d(x,x_0)\geq d(y,x_0)\vee R}d(x,y)\; |\mu-\nu|\times |\nu-\mu| (dx\times dy)\\
&\leq&\frac{8}{\|\mu-\nu\|_{TV}}\int_{d(x,x_0)\geq R}d(x,x_0)\; |\mu-\nu|\times |\nu-\mu| (dx\times dy)\nonumber\\ 
&\leq&8\int_{d(x,x_0)\geq  R}d(x,x_0)\; |\mu-\nu| (dx)\nonumber\\ 
&\leq&8\sum_{k=1}^\infty2kR\; \Big(|\mu-\nu|(B(x_0,kR)^c)-|\mu-\nu|(B(x_0,(k+1)R)^c)\Big)\nonumber\\ 
&\leq&16R\sum_{k=1}^\infty\; |\mu-\nu|(B(x_0,kR)^c)\leq 32\int_{R/2}^\infty  |\mu-\nu|(B(x_0,r)^c)dr\nonumber
\end{eqnarray}
The claim follows by combining the estimates \eqref{40}--\eqref{42}.
\end{proof}

We denote by $\widehat\mu$ the Fourier transform of the measure $\mu$ on $\R^d$ (i.e. up to a scaling by $-2\pi$, the characteristic function of a random variable with distribution $\mu$).
\begin{corollary}\label{co12} Assume that $\mu$ and $\nu$ are absolutely continuous measures on $\R^d$. Then
$$
\wass_1(\mu,\nu)_X\leq \inf_{R\geq 1} C_d\Big(R^{d+1}\|\widehat\mu-\widehat\nu\|_{L^1(\R^d)}+\int_{R/2}^\infty (\mu+\nu)(B(0,r)^c)dr.
\Big)$$
\end{corollary}

\begin{proof}
Let $f$ (resp. $g$) stand for  the density of $\mu$ (resp. $\nu$). The desired statement  follows from the previous lemma as soon as we observe that 
\begin{eqnarray*}
&&\int_{B(0,R)}|f(x)-g(x)|dx\leq C_dR^d\|f-g\|_{L^\infty(\R^d)}\leq C''_dR^d\|\widehat f-\widehat g\|_{L^1(\R^d)}.
\end{eqnarray*}
\end{proof}

We are  ready for

\begin{proof}[Proof of Proposition \ref{le2}]
All the unspecified constants  (and  the $\mathcal{O}(\cdot)$ terms) in the proof are universal in the sense that they may depend only on the the quantities  $d,b_0, b_1,b_2,b_3$. We let $C_j={\rm Cov}(H_j)$ stand for the covariance matrix of the variable $H_j$. 
Denote $\ell_n:= (\sum_{j=1}^n c_j)^{1/2}$  and observe that 
\begin{equation}\label{eq:000}
b_0^{-1/2}n^{1/2}\leq\ell_n\leq b_0^{1/2}n^{1/2}.
\end{equation}
Moreover, set
$$
\displaystyle W:=\ell_n^{-1}\sum_{j=1}^n H_j,
$$
so that ${\rm Tr}({\rm Cov}(W))=d$. By considering instead the random variables $RH_j$ where $R:\R^d\to\R^d$ is a rotation matrix chosen so that $R{\rm Cov}(W)R^{\rm T}$ is diagonal, we may assume that $A:={\rm Cov}(W)$ is diagonal:
$$
A={\rm Cov}(W)=\begin{bmatrix}
\lambda_1&0&0& \ldots &0\\
0&\lambda_2&0&\ldots &0\\
\vdots&&&&\vdots\\
0&\ldots&&0&\lambda_d
\end{bmatrix}
$$

\noindent where $\lambda_1\geq\lambda_2\geq\ldots \geq \lambda_d\geq 0$ and $
\sum_{k=1}^d\lambda_j=d$.

We start by proving an estimate of  type \eqref{e30} by first assuming that the smallest eigenvalue of $A$ satisfies the lower bound $\lambda_d\geq n^{-2\delta} $, where the constant $\delta\in [0,1/6)$ will be chosen later on. 
Towards this goal, we note that the exponential moment bound \eqref{eq:exp} for $H_k$:s implies that   $\|D^m\varphi_{H_j}\|_{L^\infty (\R^d)}\leq C$ for $m=1,2,3$ and all $j=1,\ldots ,n$, where $\varphi_{H_j}$ stands for the  characteristic function of the variable $H_j$.  Also, we have $D^2\varphi_{H_j}(0)= -{\rm Cov}(H_j)$, whence
$$
\varphi_{H_j}(\xi)=1-\frac{1}{2}\xi^{\rm T}{\rm Cov}(H_j)\xi+ \mathcal{O}(|\xi|^3) \qquad \textrm{for all}\;\; \xi.
$$
Hence for the branch of the logarithm that takes value 0 at the point 1 we have for a universal $r_1>0$
\begin{equation*}%\label{51}
\log \varphi_{H_j}(\xi) = -\frac{1}{2}\xi^{\rm T}{\rm Cov}(H_j)\xi+ \mathcal{O}(|\xi|^3) \qquad \textrm{for }\;\; |\xi|\leq 2r_1.
\end{equation*}
By independence (and since $b_0^{-1}\leq c_j\leq b_0$ for all $j$) we gather that for another universal $r_2>0$
\begin{equation}\label{52}
\log\big(\varphi_{W}(\xi)\big) = \sum_{j=1}^n\log\big(\varphi_{H_j}(\xi/\ell_n)\big) = -\frac{1}{2}\xi^{\rm T}A\xi+ n^{-1/2}\mathcal{O}(|\xi|^3) 
\end{equation}

\noindent for $|\xi|\leq r_2\sqrt{n}$. We note that $\lambda_1\geq 1$ and each $\lambda_j\geq n^{-2\delta}.$ Hence, as $|\xi|^3\leq C(d)\sum_{k=1}^d|\xi_k|^3,$ we may estimate component-wise and deduce (by also decreasing $r_2$ universally, if needed) 
\begin{equation}\label{53}
|\varphi_{W}(\xi)|\leq \exp\Big( -\frac{1}{4}\xi^{\rm T}\widetilde A\xi\Big)\qquad \textrm{for }\;\;|\widetilde A^{-1}\xi |\leq r_2\sqrt{n},
\end{equation}
where $\widetilde A$ is the $d\times d$ diagonal matrix
$$
\widetilde A:= {\rm diag\,}(1, n^{-2\delta},\ldots, n^{-2\delta})\leq A.
$$

We next choose  a $d$-dimensional centred Gaussian $G$ (independent from the $H_j$:s) such that  
\begin{align}\label{605}
B:={\rm Cov\,}(G)&=r_2^{-2}\log^2(n)\; {\rm diag\,}(n^{-1}, n^{4\delta-1},\ldots, n^{4\delta-1})\\
\notag &= (r_2^{-1} \log(n)n^{-1/2}\widetilde A^{-1})^{2}
\end{align}
and set
$
\widetilde W:= G+W.
$
Then $\varphi_{ \widetilde W}(\xi)=\varphi_{W}(\xi)\exp \big(-\frac{1}{2}\xi^{\rm T}B\xi\big)$ and we estimate
\begin{align}\label{61}
&\| \exp(-\frac{1}{2}\xi^{\rm T}A\xi)-\varphi_{\widetilde W}(\xi)\|_{L^1(\R^d)}\\
\notag &= \Big(\int_{|\widetilde A^{1/2}\xi|\leq \log n}+\int_{\left\{\begin{array}{l}{|\widetilde A^{1/2}\xi|> \log n}\\
{|B^{1/2}\xi|\leq \log n}\end{array}\right.}+\int_{\left\{\begin{array}{l}{|\widetilde A^{1/2}\xi|> \log n}\\
{|B^{1/2}\xi|> \log n}\end{array}\right.}\Big)\\
\notag & \qquad \times \big|\exp(-\frac{1}{2}\xi^{\rm T}A\xi)-\varphi_{\widetilde W}(\xi)\big|\; d\xi\nonumber
\;\;= T_1+T_2+T_3.\nonumber
\end{align}

We make use of the following simple observation for $d\times d$ symmetric matrices $J$:
\begin{equation}\label{6051}
\textrm{If  }\;\,J\geq n^{-\alpha}I,\;\; \textrm{where} \;\; \alpha >0,\;\;\textrm{then}\quad
\int_{|J^{1/2}\xi|\geq\log n}e^{-\frac{1}{4}\xi^{\rm T}J\xi}d\xi =\mathcal{O}(n^{-1/2}).
\end{equation}
This is verified by computing
\begin{align*}
\int_{|J^{1/2}\xi|\geq\log n}e^{-\frac{1}{4}\xi^{\rm T}J\xi}d\xi &\;=\; |J|^{-1/2}\int_{|\xi|\geq\log(n)}e^{-|\xi|^2/4}d\xi
\;\lesssim \;n^{d\alpha/2}\int_{r\geq \log(n)}e^{-r^2/4}r^{d-1}dr\\
\notag &\lesssim n^{d\alpha/2}\int_{r\geq \log(n)}e^{-r^2/8}dr
=\mathcal{O}\big(n^{d\alpha/2}e^{-\frac{1}{8}\log^2(n)}\big) 
\;=\;\mathcal{O}(n^{-1/2})
\end{align*}
Towards estimating the first term $T_1$ we note that since $\delta<1/6,$ we have 
$$
\sup_{\{|\widetilde A^{1/2}\xi|\leq \log n\}}n^{-1/2}|\xi|^3 =o(1)\quad \textrm{as}\quad n\to\infty.
$$ Hence we may apply \eqref{52}, the ordering $A\geq \widetilde A$ and the inequality $|e^x-1|\leq 2|x|$ for $x\in (-\infty, 1]$ to obtain the bound 
\begin{align}\label{62}
T_1&\leq \int_{| \widetilde A^{1/2}\xi|\leq \log n}e^{-\frac{1}{2}\xi^{\rm T}A\xi}\bigg|\exp \Big(-\frac{1}{2}\xi^{\rm T}B\xi+  n^{-1/2}\mathcal{O}(|\xi|^3)\Big)-1\bigg|\; d\xi\\
&\leq 2\int_{|\widetilde A^{1/2}\xi|\leq \log n}e^{-\frac{1}{2}\xi^{\rm T}\widetilde A\xi}\Big(\frac{1}{2}\xi^{\rm T}B\xi +  n^{-1/2}\mathcal{O}(|\xi|^3)\Big)\; d\xi\nonumber\\
&\leq 2|\widetilde A^{-1/2}|\int_{\R^d}e^{-|\xi|^2/2}\Big(\|\widetilde A^{-1/2}B\widetilde A^{-1/2}\||\xi|^2 + \|\widetilde A^{-1/2}\|^{3} n^{-1/2}\mathcal{O}(|\xi|^3)\Big)\; d\xi\nonumber\\
&\lesssim  n^{(d-1)\delta}\big(n^\delta \log^2(n)n^{-1+4\delta} n^\delta+ n^{-1/2}n^{3\delta}\big)
\nonumber\;
=\; \mathcal{O}\big(n^{-1/2+(d+2)\delta}\big)\nonumber,
\end{align}
since $\delta<1/6$. Next, by the last equality in \eqref{605}, the condition $|B^{1/2}\xi|\leq \log(n)$ is equivalent to $|\widetilde A^{-1}\xi|\leq r_2n^{1/2}$. Then \eqref{53} and the estimate \eqref{6051} yield
\begin{align}\label{63}
T_2&\leq \int_{|\widetilde A^{1/2}\xi|>\log n}\big(e^{-\frac{1}{4}\xi^{\rm T}\widetilde A\xi}+e^{-\frac{1}{2}\xi^{\rm T}A\xi}\big)d\xi
\;\lesssim\; \int_{|\widetilde A^{1/2}\xi|>\log n}e^{-\frac{1}{4}\xi^{\rm T}\widetilde A\xi}d\xi \\\notag
&=\; \mathcal{O}(n^{-1/2}).
\end{align}
Finally, for the remaining term $T_3$ we can again invoke \eqref{6051} to obtain
\begin{align}\label{64}
T_3&\leq\int_{\left\{\begin{array}{l}{|\widetilde A^{1/2}\xi|> \log n}\\
{|B^{1/2}\xi|> \log n}\end{array}\right.}\Big(\exp(-\frac{1}{2}\xi^{\rm T}A\xi)+\exp(-\frac{1}{2}\xi^{\rm T}B\xi)\Big)\; d\xi\\
\notag & =\;\mathcal{O}(n^{-1/2})\nonumber
\end{align}
Combining  the estimates \eqref{62}--\eqref{64} with \eqref{61} we arrive at the  estimate
\begin{eqnarray}\label{65}
\| e^{-|\xi|^2/2}-\varphi_{\widetilde W}(\xi)\|_{L^1(\R^d)} = \mathcal{O}(n^{-1/2+(d+2)\delta}).
\end{eqnarray}

We shall next apply  Bernstein's inequality on $W$, which was defined as the sum of independent random variables. Actually one may use H\"older's inequality to reduce the $d$-dimensional case to the one-dimensional one, whence we may invoke the Bernstein  type tails estimates in \cite[Theorem 2.1]{Rio} to easily infer that
there are universal constants $n_0,b'_4, b'_5$ (depending only on dimension) such that for $n\geq n_0$ it holds that
\begin{equation}\label{eq:bernstein}
\E \exp(\lambda |W|) \;\leq \;b'_4\exp(b'_4\lambda^2) \quad \textrm{and for all}\quad
\lambda \leq b'_5n^{1/2}.
\end{equation}
Choosing e.g. $\lambda =3$ here and combining  with the excellent Gaussian tail (better than $\lesssim e^{-|\xi|^2/4}$) for $G$ we see that $\Prob (|\widetilde W|>\lambda)<b''_5\exp (-2\lambda).$ This yields for $R\geq 1$ the
estimate
\begin{eqnarray}\label{66}
\int_{R/2}^\infty \Prob(|\widetilde W|\geq r)dr =\mathcal{O}(e^{-R})
\end{eqnarray}

We are now ready to invoke Corollary \ref{co12} in combination with \eqref{65} and \eqref{66} in order to deduce the existence of a Gaussian random variable $U$ such that
$$
\E |U-\widetilde W|\lesssim \inf_{R\geq 1}\big( R^{d+1}n^{-1/2+(d+2)\delta} +e^{-R}\big)\lesssim \log^{d+1}(n) n^{-1/2+(d+2)\delta} .
$$
This yields our basic estimate  
\begin{align}\label{1111}
\E |V|&= \E |U- W|\leq \E |U-\widetilde W|+\E |G|\,\\
\notag &\lesssim\; \log^{d+1}(n) n^{-1/2+(d+2)\delta}+ \log(n)n^{-1+4\delta}\\
&= \mathcal{O}\big(\log^{d+1}(n) n^{-1/2+(d+2)\delta}\big).\nonumber
\end{align}

We next explain how to infer  from \eqref{1111} the inequality \eqref{e30} or \eqref{e30'}. For  part (ii) of the proposition (which also covers the case $d=1$) we may choose $\delta=0$ in \eqref{1111} and obtain directly  
 \eqref{e30'}.  In order to deal with part (i) of the proposition (where $d\geq 2$) we assume first that $\lambda_d\geq n^{-(4d+12)^{-1}}.$ Then we may apply directly \eqref{1111} with the choice $\delta=(2d+6)^{-1}$ and obtain the inequality \eqref{e30} with the exponent
 $$
 \beta_0= 1/2-(d+2)(2d+5)^{-1}>0,
 $$ 
 that depends only on $d.$
In the remaining case there is $k_0\in\{2,\ldots d-1\}$ so that $\lambda_{k_0}\geq n^{-(4d+12)^{-1}}$ but $\lambda_{k_0+1}<n^{-(4d+12)^{-1}}.$ Write $W':=(W_1,\ldots,W_{k_0})$ and $W'':=(W_{k_0+1},\ldots,W_d).$ We may apply the above proof on
 $W'$ and find a $k_0$-dimensional Gaussian approximation $U'$ for $W'$ so that $\E | U'-W' |= 
 \mathcal{O}\big( n^{-\beta_0}\big)$ (some unimportant adjustments of the multiplicative constants are needed to get the $k_0$-dimensional normalisation right, and we actually could get a better decay). We then define   the trivial extension $U'$ to a $d$-dimensional random variable $U$ by setting 
 $U=(U',U'')$, where the components of $U''$ are identically zero. Now 
 $$
 \E |W''|\leq (\E |W''|^2)^{1/2}=(\sum_{k=k_0+1}^d\lambda_k)^{1/2}\lesssim n^{-(2d+6)^{-1}}\lesssim n^{-\beta_0}
 $$
 Finally,
\begin{align*}
 \E |V| &\leq \E |W'-U'|+\E |W''|\;\lesssim n^{-\beta_0}.
\end{align*}
This proves the desired estimate \eqref{e30}.

We turn to estimating   the exponential moments. Their proof is based solely on \eqref{e30} and the uniform exponential estimates of $V$ which follow from those of $W$ and from the automatic ones for the Gaussian approximation. By invoking the  Bernstein estimate \eqref{eq:bernstein}  and using the double exponential decay (with constants depending on $d$) of our Gaussian approximation $U$ we obtain universal (i.e. depending only on $d$) constants $b_4, b_5$ such that
\begin{equation}\label{eq:1111}
\Prob (|V|\geq u)\leq b_4e^{-2\lambda u}e^{4b_4\lambda^2}\qquad\textrm{for any } \quad u>0\quad\textrm{and}\quad \lambda\in(0,b_5\sqrt{n}).
\end{equation}
We denote  $\alpha:=a_1n^{-\beta}$ and  assume that $\lambda\in (0, b_5\sqrt{n})$. Letting  $M\geq 0$ stand for an auxiliary parameter, to be chosen later, we may  write
\begin{align*}
\E e^{\lambda |V|} &=1+ \E \big( |V|\frac{\exp(\lambda |V|)-1}{|V|}\chi_{\{|V|\leq M\}}\big)\\
\notag &\quad  + (e^{\lambda M}-1)\Prob (|V|>M)+\lambda \int_M^\infty e^{\lambda u}\Prob (|V|>u)du.
\end{align*}
We recall that by  \eqref{e30} we have $\E |V|\leq \alpha$. Also, we note that $t\mapsto t^{-1}(e^{\lambda t}-1)$ (defined to be $\lambda$ at zero) is increasing on $[0,M]$, and hence less than $M^{-1}(e^{\lambda M}-1)$ on that interval. Using these facts and the bound \eqref{eq:1111} we deduce that
\begin{align*}
\E \exp(\lambda |V|) -1 &\leq \delta (e^{\lambda M}-1)M^{-1}+ (e^{\lambda M}-1)b_4e^{4b_4\lambda^2-2\lambda M}+b_4e^{4b_4\lambda^2}\lambda \int_M^\infty e^{-\lambda u}du\\
\notag &\leq \delta e^{\lambda M}M^{-1}+2b_4e^{-\lambda M}e^{4b_4\lambda^2}
\end{align*}
The desired estimate is obtained by choosing $M$ so that $\sqrt{\delta}=e^{-\lambda M}$ and plugging in the definition of $\delta.$

Assume then the that variables $H_k$ are uniformly bounded. In this case a standard application of Azuma's inequality yields universal constants $s, r>0$ so that
$$
\Prob (|V|\geq u)\leq  se^{-2ru^2} \quad \textrm{for all}\quad u>0.
$$
In an analogous manner  to what we just did for the exponential moments,  for any $M>0$ it follows that
\begin{align*}
\E e^{rV^2} &=1+ \E \big( |V|\frac{e^{rV^2}-1}{|V|}\chi_{\{|V|\leq M\}}\big) + (e^{rM^2}-1)\Prob (|V|>M)\\
\notag &\quad +2r\int_M^\infty xe^{rx^2}\Prob (|V|>x)dx
\end{align*}
and we deduce that
\begin{align*}
\E \exp(r|V|^2) &\leq 1+\delta (e^{rM^2}-1)M^{-1}+ s(e^{rM^2}-1)e^{-2rM^2}+s\int_M^\infty 2rxe^{-rx^2}dx\\
\notag &\leq 1+\delta M^{-1}e^{rM^2}+2se^{-rM^2}.
\end{align*}
The desired estimate is  obtained by  this time choosing $M$ so that $\sqrt{\delta}=e^{-rM^2}$.
\end{proof}

\section{Gaussian multiplicative chaos and Random unitary matrices -- the global scale}\label{app:rmt}

In this appendix we prove that on the global scale, the characteristic polynomial of a Haar distributed random unitary matrix converges to a multiplicative chaos distribution. The proof is very similar to the zeta case, though naturally the moment calculations are different. For these, we use known results from \cite{CFKRS,DIK}

To fix notation, let $U_N$ be a $N\times N$ random unitary matrix whose law is the Haar measure on the unitary group. Our main object of interest is the characteristic polynomial of $U_N$ and we consider it evaluated on the unit circle. For $\theta\in[0,2\pi]$, we define

\begin{equation*}
\upsilon_N(\theta):=\det(I-e^{-i\theta}U_N).
\end{equation*}

We introduce an approximation, which is analogous to the truncated Euler product we consider above. More precisely, for $M\in \Z_+$, we define

\begin{equation*}
\upsilon_{N,M}(\theta)=e^{-\sum_{k=1}^M \frac{1}{k}\mathrm{Tr}U_N^k e^{-ik\theta}}.
\end{equation*}

Our goal is to again control the difference $\upsilon_N-\upsilon_{N,M}$ in the limit where $N\to\infty$ and then $M\to\infty$ while proving that in this limit, $\upsilon_{N,M}$ converges to a certain multiplicative chaos distribution. 
A suitable space of generalized functions to carry out this analysis in  is provided by the Sobolev space on the circle:

\begin{equation*}
\mathcal{H}^\alpha=\left\lbrace f(\theta)=\sum_{j\in \Z}\widehat{f}_j e^{ij\theta}: ||f||_\alpha^2:=\sum_{j\in \Z}(1+j^2)^{\alpha}|\widehat{f}_j|^2<\infty\right\rbrace.
\end{equation*}

Let us begin with an analogue of Proposition \ref{pr:B}.

\begin{lemma}\label{le:rmtmom}
Let $\alpha>1/2$. Then 

\begin{equation*}
\lim_{M\to\infty}\limsup_{N\to\infty}\E ||\upsilon_N-\upsilon_{N,M}||_{-\alpha}^2=0.
\end{equation*}
\end{lemma}

\begin{proof}
As in the zeta case, we expand the square and study the asymptotics of each term. Consider first the $||\upsilon_N||^2$-term. We have 

\begin{align}\label{eq:upsilonsq}
\E ||\upsilon_N||_{-\alpha}^2=\sum_{j\in \Z}\frac{1}{(1+j^2)^\alpha}\int_{[0,2\pi]^2}e^{ij(\theta-\theta')}\E \left[\det(I-e^{-i\theta}U_N)\det(I-e^{i\theta'}U_N^*)\right]\frac{d\theta}{2\pi}\frac{d\theta'}{2\pi},
\end{align}

\noindent where $U_N^*$ denotes Hermitian conjugation. One can justify changing the order of the different integrals by the facts that for each fixed $N$, the absolute value of the determinant is bounded by $2^N$ (deterministically) and as $\alpha>1/2$ the sum $\sum_{j\in \Z}(1+j^2)^{-\alpha}$ converges. 

We thus need to understand the expectation above. The expression that is relevant to us follows from e.g. \cite[equation (2.16)]{CFKRS} (in the notation of \cite{CFKRS} we have $\E \det(I-e^{-i\theta}U_N)\det(I-e^{i\theta'}U_N^*)=e^{-iN\theta}I_{1,2}(U(N);e^{i\theta'};e^{i\theta})$):

\begin{equation*}
\E \left[\det(I-e^{-i\theta}U_N)\det(I-e^{i\theta'}U_N^*)\right]=\frac{1}{1-e^{i(\theta'-\theta)}}+\frac{e^{iN(\theta'-\theta)}}{1-e^{i(\theta-\theta')}}=\sum_{l=0}^N e^{-il(\theta-\theta')}.
\end{equation*}

Plugging this into \eqref{eq:upsilonsq}, we see that

\begin{equation}\label{eq:upsilonasy}
\E ||\upsilon_N||_{-\alpha}^2=\sum_{l=0}^N \frac{1}{(1+l^2)^{\alpha}}\to \sum_{l=0}^\infty \frac{1}{(1+l^2)^{\alpha}}<\infty
\end{equation}

\noindent as $N\to\infty$.

For the terms with $\upsilon_{N,M}$ we'll need results from \cite{DIK} which concern the asymptotics of Toeplitz determinants with Fisher-Hartwig singularities. For the cross term, we note that instead of \eqref{eq:upsilonsq}, one now is interested in expectations of the form $\E \det(I-e^{-i\theta}U_N) e^{-\sum_{k=1}^M \frac{1}{k}e^{ik\theta'}\mathrm{Tr}[U_N^*]^k}$ (again with an elementary argument one can justify changing the order of sums and integrals in calculating the second moment of the Sobolev norm). Using the rotation invariance of the law of $U_N$, we can write this expectation as $\E \prod_{j=1}^N f(e^{i\varphi_j})$, where $(e^{i\varphi_j})_{j=1}^N$ are the eigenvalues of $U_N$, and $f(z)=e^{-\sum_{k=1}^M \frac{1}{k}e^{ik(\theta'-\theta)}z^{-k}}(1-z)$. Using elementary geometry, one can check that $f$ can be written in the form

\begin{equation}\label{eq:FH}
f(z)=-ie^{V(z;\theta-\theta')}z^{1/2}|z-1|,
\end{equation}

\noindent where 

\begin{equation*}
V(z;\theta-\theta')=\sum_{k\in \Z}V_k z^k=-\sum_{k=1}^M \frac{1}{k}e^{ik(\theta'-\theta)}z^{-k}
\end{equation*}

\noindent and the interpretation of $z^{1/2}$ is the following: for $\mathrm{arg} z\in[0,2\pi)$, $z^{1/2}=e^{\frac{1}{2}i \mathrm{arg} z}$.

Due to the Heine-Szeg\H{o} identity, the expectation $\E \prod_{j=1}^N f(e^{i\varphi_j})$ can be written as the determinant of a $N\times N$ Toeplitz matrix with symbol $f$. The study of the asymptotics of such determinants with symbols of the form \eqref{eq:FH}, namely symbols with Fisher-Hartwig singularities, has a long history and the asymptotics are well known. We refer to \cite[Theorem 1.1]{DIK} for a more general result and a proof of the following fact: for $f$ of the form \eqref{eq:FH}

\begin{align}\label{eq:FHasy}
\E\prod_{j=1}^N f(e^{i\varphi_j})&=e^{-\sum_{k=-\infty}^{-1} V_k }(1+\mathit{o}(1))=e^{\sum_{k=1}^M\frac{1}{k}e^{ik(\theta'-\theta)}}(1+\mathit{o}(1))
\end{align}

\noindent as $N\to\infty$ for fixed $M$.
To help the interested reader compare with \cite[Theorem 1.1]{DIK}, we point out that in their notation, our case becomes $m=0$, $\alpha_0=\beta_0=1/2$, $V_k=-\frac{1}{|k|}e^{-ik(\theta'-\theta)}$ for $k\in\lbrace -M,...,-1\rbrace$ and zero otherwise. Thus the only non-trivial quantity in \cite[equation (1.10)]{DIK} is $b_-(1)^{-\alpha_0-\beta_0}$ -- every other term in the product is one. We also point out that similarly to the discussion in \cite[Remark 1.4]{DIK}, one can check that these asymptotics are uniform in $\theta,\theta'$. More precisely, one can check in general that if $V$ is a Laurent polynomial of fixed degree and if we require the coefficients of $V$ to be restricted to some fixed compact set, then the asymptotics of \cite[Theorem 1.1]{DIK} are uniform in the coefficients. This can be deduced from corresponding uniform bounds on the relevant jump matrices.

Due to the uniformity of \eqref{eq:FHasy} and our assumption that $\alpha>1/2$, we thus have (writing $\langle\cdot,\cdot\rangle_{-\alpha}$ for the inner product of $\mathcal{H}^{-\alpha}$ and using the convention that it is conjugate linear in the second entry) 

\begin{align}\label{eq:mixedasy2}
\notag \E\langle \upsilon_{N},\upsilon_{N,M}\rangle_{-\alpha}&=\sum_{j\in \Z}\frac{1}{(1+j^2)^{\alpha}}\int_{[0,2\pi]^2} e^{-ij(\theta-\theta')}e^{\sum_{k=1}^M \frac{1}{k}e^{ik(\theta'-\theta)}}\frac{d\theta}{2\pi}\frac{d\theta'}{2\pi}+\mathit{o}(1)\\
&=\sum_{j\in \Z}\frac{1}{(1+j^2)^{\alpha}}\int_0^{2\pi}e^{-ij\theta}e^{\sum_{k=1}^M \frac{1}{k}e^{-ik\theta}}\frac{d\theta}{2\pi}+o(1)
\end{align}

\noindent in the limit where $N\to\infty$ and $M$ is fixed. Now by expanding the exponential and applying the dominated convergence theorem  one can check that the relevant asymptotics become

\begin{align*}
\E\langle \upsilon_{N},\upsilon_{N,M}\rangle_{-\alpha}&=\sum_{j=0}^\infty\frac{1}{(1+j^2)^\alpha}\sum_{l=0}^\infty\frac{1}{l!}\sum_{k_1,...,k_l=1}^M\frac{1}{k_1\cdots k_l}\delta_{j,k_1+\dots +k_l}+o(1),
\end{align*}

\noindent where the interpretation of the $l=0$ term in the sum is $\delta_{j,0}$. Now by monotonicity/positivity, we see that 

\begin{equation*}
\lim_{M\to\infty}\lim_{N\to\infty}\E\langle \upsilon_{N},\upsilon_{N,M}\rangle_{-\alpha}=\sum_{j=0}^\infty\frac{1}{(1+j^2)^\alpha}\sum_{l=0}^\infty\frac{1}{l!}\sum_{k_1,...,k_l=1}^\infty\frac{1}{k_1\cdots k_l}\delta_{j,k_1+\dots +k_l}.
\end{equation*}
Invoking Lemma \ref{le:joulu2} (which we prove shortly in the next appendix) we conclude that
\begin{equation}\label{eq:mixedasy}
\lim_{M\to\infty}\lim_{N\to\infty}\E\langle \upsilon_{N},\upsilon_{N,M}\rangle_{-\alpha}=\sum_{j=0}^\infty\frac{1}{(1+j^2)^\alpha}.
\end{equation}

Finally we need to understand $\E ||\upsilon_{N,M}||_{-\alpha}^2$. This can again be reduced to understanding expectations of the form $\E \prod_{j=1}^N e^{\mathcal{V}(e^{i\varphi_j};e^{i\theta},e^{i\theta'})}$, where 

\begin{equation*}
\mathcal{V}(z;e^{i\theta};e^{i\theta'})=-\sum_{k=1}^M\frac{1}{k}(e^{-ik\theta}z^k +e^{ik\theta'}z^{-k})=\sum_{k\in \Z}\mathcal{V}_k z^k.
\end{equation*}
This expectation is again a Toeplitz determinant with symbol $e^{\mathcal{V}}$ and its asymptotics follow from the strong Szeg\H{o} theorem. As we need uniformity in $\theta,\theta'$ since we want to interchange the order of the limit and the integrals, we need something slightly stronger than the standard strong Szeg\H{o} theorem. This corresponds to a  simplified version of the problem studied in \cite{DIK} where there are no singularities (in their notation $\alpha_k=\beta_k=0$ for all $k$) and again in the Riemann-Hilbert problem the jump matrices will have bounds that are uniform in $\theta,\theta'$ and one finds 

\begin{equation*}
\E \prod_{j=1}^N e^{\mathcal{V}(e^{i\varphi_j};e^{i\theta},e^{i\theta'})}=e^{N \mathcal{V}_0+\sum_{j=1}^\infty j \mathcal{V}_j\mathcal{V}_{-j}}(1+o(1))=e^{\sum_{k=1}^M \frac{1}{k}e^{-ik(\theta-\theta')}}(1+o(1)),
\end{equation*}
where the error is again uniform in $\theta,\theta'$ as $N\to\infty$ with fixed $M$. We thus find as before that as $N\to\infty$ for a fixed $M$

\begin{equation*}
\E||\upsilon_{N,M}||_{-\alpha}^2=\sum_{j\in \Z}\frac{1}{(1+j^2)^\alpha}\int_{[0,2\pi]^2}e^{-ij(\theta-\theta')}e^{\sum_{k=1}^M\frac{1}{k}e^{-ik(\theta-\theta')}}\frac{d\theta}{2\pi}\frac{d\theta'}{2\pi}+o(1),
\end{equation*}

\noindent which is precisely the same as what we found in \eqref{eq:mixedasy2}. Thus
\begin{equation}\label{eq:truncasy}
\lim_{M\to\infty}\lim_{N\to\infty}\E||\upsilon_{N,M}||_{-\alpha}^2=\sum_{j=0}^\infty \frac{1}{(1+j^2)^\alpha}.
\end{equation}

Putting together \eqref{eq:upsilonasy}, \eqref{eq:mixedasy} and \eqref{eq:truncasy}, we infer that 

\begin{equation*}
\lim_{M\to\infty}\lim_{N\to\infty}\E||\upsilon_N-\upsilon_{N,M}||_{-\alpha}^2=0.
\end{equation*}
\end{proof}

The next step in proving that $\upsilon_N$ converges in $\mathcal{H}^{-\alpha}$ is to show that if we first let $N\to\infty$ and then $M\to\infty$, then $\nu_{N,M}$ converges in $\mathcal{H}^{-\alpha}$. Let us prove this now.

\begin{lemma}\label{le:rmttruncconv}
If we first let $N\to\infty$ and then $M\to\infty$, then $\upsilon_{N,M}$ converges in law with respect to the topology of $\mathcal{H}^{-\alpha}$ to an object $\upsilon$ which can be formally written as $\upsilon=e^{H(e^{i\theta})}$, where 

\begin{equation*}
H(e^{i\theta})=\sum_{k=1}^\infty \frac{1}{\sqrt{k}}Z_k e^{-ik\theta}.
\end{equation*}

\noindent Here $(Z_k)_{k=1}^\infty$ are i.i.d. standard complex Gaussians. 
\end{lemma}

\begin{proof}
Our main tool is  a result due to Diaconis and Shahshahani \cite{DiSh}, who proved that for any fixed $M$
\begin{equation}\label{eq:DiSh}
(\frac{1}{\sqrt{k}}\mathrm{Tr}U_N^k)_{k=1}^M\stackrel{d}{\to}(Z_k)_{k=1}^M
\quad\textrm{as}\quad N\to\infty.
\end{equation}
 We'll transform this into a statement about the convergence of $\upsilon_{N,M}$ as $N\to\infty$. Lemma \ref{le:simple} does not apply directly, so 
we take a slightly different approach here. Consider the function $F:\C^M\to \mathcal{H}^{-\alpha}$ with $\alpha>1/2$ defined by 
\begin{equation*}
[F(z_1,...,z_M)](e^{i\theta})=e^{-\sum_{j=1}^M \frac{1}{\sqrt{j}}e^{-ij\theta}z_j}.
\end{equation*}
Let us write this as 
\begin{equation*}
[F(z_1,...,z_M)](e^{i\theta})=\sum_{m=0}^\infty F_{-m}(z_1,...,z_M) e^{-im\theta},
\end{equation*}
whence the definition of Fourier coefficients gives the rough estimate
\begin{equation*}
|F_{-m}(z_1,...,z_M)|\leq e^{||z||_\infty \sum_{k=1}^M \frac{1}{\sqrt{k}}}.
\end{equation*}
This yields a uniform bound for the Fourier coefficients in each ball $\{||z||_\infty \leq R\}$. Hence $F(z)\in \mathcal{H}^{-\alpha}$ for $\alpha>1/2$, and since the $F_{-j}$:s  are polynomials in $z_1,\ldots z_M$ we deduce that, in fact, $F:\C^M\to \mathcal{H}^{-\alpha}$ is a continuous mapping.  At this stage \eqref{eq:DiSh} implies directly by definition of convergence in distribution that
\begin{equation*}
\upsilon_{N,M}\stackrel{d}{\to}F(Z_1,...,Z_M)\quad\textrm{as}\quad N\to\infty
\end{equation*}
in the topology of $\mathcal{H}^{-\alpha}$. 

Now if we realize the i.i.d. Gaussians $Z_j$ on the same probability space, then the random variables $(F(Z_1,...,Z_M))_M$ will form a $\mathcal{H}^{-\alpha}$-valued martingale and one has 

\begin{equation*}
\E ||F(Z_1,...,Z_M)||_{-\alpha}^2=\sum_{j\in \Z}\frac{1}{(1+j^2)^\alpha}\int_{[0,2\pi]^2}e^{-ij(\theta-\theta')}e^{\sum_{k=1}^M \frac{1}{k}e^{-ik(\theta-\theta')}}\frac{d\theta}{2\pi}\frac{d\theta'}{2\pi}.
\end{equation*} 
We already saw in the proof of Lemma \ref{le:rmtmom} that this quantity remains bounded as $M\to\infty$, so in particular, $(F(Z_1,...,Z_M))_{M=1}^\infty$ is a $\mathcal{H}^{-\alpha}$-valued $L^2$-bounded martingale. As in the zeta-case, this martingale converges almost surely to a $\mathcal{H}^{-\alpha}$-valued random variable $\upsilon $:
$$
F(Z_1,...,Z_M)\stackrel{d}{\to}\upsilon\quad\textrm{as}\quad M\to\infty.
$$
We conclude that if we first let $N\to\infty$ and then $M\to\infty$, then $\upsilon_{N,M}$ converges in law to $\upsilon$ (in the space $\mathcal{H}^{-\alpha}$). Finally, the formal interpretation of the representation in terms of $H$ simply comes from $\log F(Z_1,...,Z_M)$ converging to $H$ (in say $\mathcal{H}^{-\varepsilon}$).
\end{proof}

To argue as in the zeta case to prove that $\upsilon_N$ converges to $\upsilon$, we would need to know that one has convergence in the Wasserstein metric in the previous proof. But as we see from the proof of Lemma \ref{le:rmtmom}, the second moment of $||\upsilon_{N,M}||_{\mathcal{H}^{-\alpha}}$ is bounded when we first let $N\to\infty$ and then $M\to\infty$, so combined with Lemma \ref{le:rmttruncconv}, we have an identical setting as in the zeta case, and in complete analogy with the proof of Theorem \ref{th:main} ({\bf{\rm{i}}}), one finds the following result (we omit the proof as it is essentially identical to the zeta-case).

\begin{proposition}\label{pr:rmtconv}
For $\alpha>1/2$, $\upsilon_N$ converges in law $($with respect to the topology of $\mathcal{H}^{-\alpha})$ to $\upsilon$ when we let $N\to\infty$, and

\begin{equation*}
\lim_{N\to\infty}\wass_2(\upsilon_N,\upsilon)_{\mathcal{H}^{-\alpha}}=0.
\end{equation*}
\end{proposition}

\section{Gaussian multiplicative chaos and Random unitary matrices -- the mesoscopic scale}\label{app:mesormt}

The goal of this appendix is to prove Theorem \ref{th:mesormt}. While we expect that an identical result as in the zeta case can be indeed proven with slightly more effort, this would involve mainly analysis of purely Gaussian multiplicative chaos with no questions related to random matrix theory or number theory, so we will be satisfied with Theorem \ref{th:mesormt}.

\iffalse

Recall below that $\eta$ was defined in Section \ref{sec:meso} after Lemma  \ref{le:vanutus}. 
%Recall also that $\|f\|^2_V=\int_\R |f(x)|^2(1+x^2)^{-1}dx$ and $\reg{f}(x)=(1+x^2)^{-1}f(x).$

\begin{theorem}\label{th:rmtmeso} Fix $\alpha\in (1/2,1).$
There exists a deterministic $\varepsilon_N$ such that $\varepsilon_N\to 0^+$ as $N\to \infty$, and

\begin{equation*} 
\lim_{N\to\infty}\wass_2\left(\upsilon_N(\varepsilon_N \cdot)\; , \; e^{\left(\sqrt{\log \frac{1}{\varepsilon_N}}+\mathcal{O}(1)\right)Z}\left[\eta+o(1)\right]\right)_{W^{-\alpha, 2}(0,1)}=0,
\end{equation*}

\noindent where $Z$ is a standard complex Gaussian, $\mathcal{O}(1)$ is deterministic and independent of $f$, and $o(1)$ is a random $W^{-\alpha, 2}(0,1)$-valued quantity that tends to zero in probability.
\end{theorem}

\fi

We begin by recording some auxiliary observations needed for the proof. The following easy inequality  is an analogue of Lemma \ref{le:embed}.

\begin{lemma}\label{le:joulu1} 
There exists a positive constant $C$ such that for all $g$ with $\int_\R |g(x)|^2(1+x^2)dx<\infty$  and $M\geq 1$
$$
M^{-1}\sum_{n\in\Z}|\widehat {g}(n/M)|^2\; \leq \; C\int_\R |g(x)|^2(1+x^2)dx.
$$
\end{lemma}
\begin{proof} 
Consider first the case $M=1$. We obtain by some elementary identities and Cauchy-Schwarz 
\begin{eqnarray*}
\sum_{n\in\Z}|\widehat g(n)|^2&=&\sum_{n\in\N} \Big|\int_n^{n+1}\big(\widehat g(u)du-(n+1-u)\widehat g'(u)\big)du\Big|^2
%\leq \sum_{n\in\N} \Big(\int_n^{n+1}\big(|\widehat g(u)|+|\widehat g'(u)|\big)du\Big)^2
\\
&\leq&
2\int_\R\big(|\widehat g(u)|^2+|\widehat g'(u)|^2\big)du =2\int_\R |g(x)|^2(1+4\pi^2x^2)dx
\end{eqnarray*}
The general case $M\geq 1$ then follows by  a simple scaling argument.
\end{proof}
\noindent We will also make use of the following simple facts.

\begin{lemma}\label{le:joulu2} {\bf (i)}\quad
For any $\varepsilon \geq 0$ one has
$\displaystyle
\sum_{k=0}^\infty \frac{1}{k!}\sum_{j_1,...,j_k=1}^\infty\frac{e^{-\varepsilon (j_1+\cdots j_k)}}{j_1\cdots j_k}\delta_{j_1+\cdots +j_k,l}=e^{-\varepsilon l}.
$
\noindent{\bf (ii)}\quad Let $\varepsilon \in (0,1/2)$ and $M\in \N$. Then
$$\displaystyle
\sum_{j=1}^{M-1}\frac{e^{-2\varepsilon j/M}}{j} =\sum_{j=1}^{M-1}\frac{1}{j}+o(1),
$$
where $o(1)\to 0$ as  $\varepsilon \to 0^+$, uniformly in $M\geq 2.$ 
\end{lemma}
\begin{proof} For the first claim we assume that
$\varepsilon >0$ and denote $e^{-\varepsilon}=u\in (0,1),$. The statement is then equivalent to the identity $\exp(\log(1/(1-u))= 1/(1-u)$. The case $\varepsilon =0$ follows by a limiting argument. Towards (ii) one uses the inequality $1-e^{-x}\le x$ to obtain
$$
\sum_{j=1}^{M-1}\frac{1-e^{-2\varepsilon j/M}}{j} \leq \frac{2\varepsilon}{ M}\sum_{j=1}^{M-1} 1\leq 2\varepsilon .
$$
\end{proof}

We continue with a statement about how scaling affects local Sobolev norms of distributions on the unit circle after they are spread on $\R$. This is analogous to Lemma \ref{le:vanutus}.

\begin{lemma}\label{le:scalingcirc}
Assume that $\lambda\in {\mathcal H}^{-\alpha}$ with $\alpha\in (1/2,1).$ Then for $\varepsilon\in (0,1]$
$$
\| \lambda (e^{2\pi i\varepsilon \cdot})\|_{W^{-\alpha,2}(0,1)}\lesssim \varepsilon^{-1-2\alpha}\|\lambda\|_{{\mathcal H}^{-\alpha}}.
$$
\end{lemma}
\begin{proof}
Fix a cut-off function $\phi_0\in C^\infty_0(\R)$ such that ${\phi_0}_{|(-1,2)}=1.$ By Lemma \ref{le:vanutus} and the definition of the norm of $W^{-\alpha, 2}(0,1)$ it is enough to consider the case $\varepsilon =1$ and estimate
$\| \phi_0\lambda(e^{2\pi i\cdot})\|_{W^{-\alpha,2}(\R)}.$ One  may write $\lambda(e^{2\pi i x})=\sum_{k\in\Z}a_ke^{2\pi ikx}$
with $\sum_{k\in\Z}|a_k|^2(1+k^2)^{-\alpha}\sim \|\lambda\|_{{\mathcal H}^{-\alpha}}^2<\infty.$
We note that $\widehat{\phi_0\lambda}(\xi)=\sum_{k\in\Z}a_k\widehat\phi_0(\xi-k)$ and observe that by the regularity of $\phi_0$, $|\widehat\phi_0(\xi)|\lesssim (1+|\xi |^2)^{-2}$. Moreover, we have the estimates 
$$
\| \sum_{j\in\Z}(1+|\xi-j|^2)^{-2}\|_{L^\infty(\R)}<\infty\quad\textrm{and}\quad \int_\R(1+(\varepsilon\xi)^2)^{-\alpha}(1+|\xi-k|^2)^{-2}\lesssim \varepsilon^{-2\alpha}(1+k^2)^{-\alpha},
$$
for any $k\in\Z.$
The latter estimate is obtained by assuming that $k\geq 1$  and considering separately the  integration ranges $\xi <k/2$ and $\xi\geq k/2.$ 
Using these estimates and Cauchy-Schwarz inequality we deduce that
\begin{eqnarray*}
&&\| \phi_0\lambda(e^{2\pi \varepsilon i\cdot})\|^2_{W^{-\alpha,2}(\R)}\; \lesssim \;\int_\R \left| \varepsilon^{-1} \sum_{k\in\Z} \frac{|a_k|}{(1+|\varepsilon^{-1}\xi-k|^2)^2}\right|^2(1+\xi^2)^{-\alpha}d\xi \\
&\lesssim& \varepsilon^{-1}\int_\R \sum_{k\in\Z} \frac{|a_k|^2}{(1+|\xi'-k|^2)^2}(1+\varepsilon^2\xi'^2)^{-\alpha}d\xi'  \; \lesssim\;
\varepsilon^{-1-2\alpha}\sum_{k\in\Z}|a_k|^2(1+k^2)^{-\alpha}\\
&\lesssim& \varepsilon^{-1-2\alpha}\|\lambda\|^2_{{\mathcal H}^{-\alpha}}.
\end{eqnarray*}
\end{proof}

As in the $\zeta$-case in Section \ref{sec:meso}, one applies the above observation  to $\varphi(\varepsilon_N\cdot)=\upsilon_N(\varepsilon_N\cdot)-\upsilon(\varepsilon_N\cdot)$, where the coupling is the one given through the Wasserstein distance. At this stage Proposition \ref{pr:rmtconv} and Lemma \ref{le:scalingcirc} together verify that there exists some $\varepsilon_N$ under which the Wasserstein distance between $\upsilon_N(\varepsilon_N\cdot)-\upsilon(\varepsilon_N\cdot)$ tends to zero as $N\to\infty$. It remains for us to understand how $\upsilon(\varepsilon\theta)$ behaves as $\varepsilon\to 0$. By our construction, we may actually restrict to the case where $\varepsilon_N=1/M_N$, where $M_N$ is some integer.  For this situation we make the following statement.

\begin{lemma}\label{le:mesogaussianlemma} There is a $C<\infty$  such that for
$M\geq 1$ and $f\in C^\infty_0 (\R)$   $($and a suitable version of $\eta)$ it holds that
\begin{equation}\label{eq:joulu10}
\E \big|e^{-\sum_{j=1}^{M-1}\frac{1}{\sqrt{j}}Z_j}\int_\R \upsilon(x/M)f(x)dx\big|^2\leq C\int_\R |f(x)|^2(1+x^2)dx,
\end{equation}
\begin{equation}\label{eq:joulu11}
\E \big|\eta(f)|^2\leq C\int_\R |f(x)|^2(1+x^2)dx,
\end{equation}
and 
\begin{equation}\label{eq:joulu12}
\E \big|\eta(f) \; -\, e^{-\sum_{j=1}^{M-1}\frac{1}{\sqrt{j}}Z_j}\int_\R \upsilon(x/M)f(x)dx\big|^2\to 0\qquad\textrm{as}\quad M\to\infty,
\end{equation}
where $\eta$ is as in Theorem \ref{th:meso1}.
\end{lemma}
\noindent
Assuming the above Lemma, we may finish the
\begin{proof}[Proof of Theorem \ref{th:mesormt}] Let us denote  $\iota_M:=\eta-e^{-\sum_{j=1}^{M-1}\frac{1}{\sqrt{j}}Z_j}\int_\R \upsilon(\cdot/M)_{|(0,1)}.$ We apply the above lemma on $f_\xi:= e^{-2\pi i\xi x}\phi_0(x)$, where again $\phi_0$ is a suitable  smooth cut-off function, multiply by $(1+\xi^2)^{-\alpha}$, and note that \eqref{eq:joulu10} and \eqref{eq:joulu11} verify that the dominated convergence theorem applies to show that $\E \|\iota_M\|_{W^{-\alpha,2}(0,1)}^2\to 0$ as $M\to\infty.$ Especially, $ \iota_M\to 0$ in distribution in the Sobolev space, and the theorem follows by
writing $v(\cdot/M)= e^{\sum_{j=1}^{M-1}\frac{1}{\sqrt{j}}Z_j}(\eta+\iota_M).$
\end{proof}

Let us now turn to the

\begin{proof}[Proof of Lemma \ref{le:mesogaussianlemma}]
We note that \eqref{eq:joulu11} was already essentially established in Section \ref{sec:meso}  (see Lemma \ref{le:meso2} and Lemma \ref{le:meso3}).
In order to proceed towards the remaining claims, we simplify notation slightly and introduce the following objects:

\begin{equation*}
U_{N,M}(x):=\exp\left(\sum_{j=1}^{M-1} \frac{e^{-ijx/M}-1}{\sqrt{j}}Z_j+\sum_{j=M}^N \frac{e^{-ijx/M}}{\sqrt{j}}Z_j\right)
\end{equation*}

\begin{equation*}
U_{N,M}^{(\varepsilon)}(x)=\exp\left(\sum_{j=1}^{M-1} \frac{e^{-ijx/M}-1}{\sqrt{j}}e^{-\varepsilon j/M}Z_j+\sum_{j=M}^N \frac{e^{-\varepsilon j/M}e^{-ijx/M}}{\sqrt{j}}Z_j\right),
\end{equation*}

\noindent and 

\begin{equation*}
\eta^{(\varepsilon)}(x)= \exp\left(\int_{0}^1 \frac{e^{-iux}-1}{\sqrt{u}}e^{-\varepsilon u}dB_u^\C+\int_1^{\infty}\frac{e^{-\varepsilon u}e^{-iux}}{\sqrt{u}}dB_u^\C\right),
\end{equation*}

\noindent where $\varepsilon>0$. Note that formally $\eta^{(0)}=\eta$ and $U_{\infty,M}(x)=e^{-\sum_{j=1}^{M-1}Z_j/\sqrt{j}}\upsilon(x/M)$. So to prove \eqref{eq:joulu12}, we wish to show that $U_{\infty,M}(f)\stackrel{d}{\to}\eta(f)$ as $M\to\infty$. To do this, we couple the two objects by choosing

\begin{equation}\label{eq:zetacoupl}
Z_j:=\sqrt{M}\int_{\frac{j}{M}}^{\frac{j+1}{M}}dB_u^\C.
\end{equation}
We then note that for a constant $C$ independent of $N,M,\varepsilon$ and $f$
\begin{align*}
\E |U_{\infty,M}(f)-\eta(f)|^2&\leq C\big[\E|U_{\infty,M}(f)-U_{\infty,M}^{(\varepsilon)}(f)|^2+\E|U_{\infty,M}^{(\varepsilon)}(f)-\eta^{(\varepsilon)}(f)|^2\\
&\quad +\E|\eta^{(\varepsilon)}(f)-\eta(f)|^2\big].
\end{align*}
Our goal is to show that the first term on the right hand side is small uniformly in $M$ when $\varepsilon$ is small enough, the last term tends to zero when $\varepsilon\to 0$,  while the second term converges to zero when we let $M\to\infty$ for any fixed $\varepsilon>0$. This will then imply \eqref{eq:joulu12}.

Let us begin with the second term. 

\begin{lemma}\label{le:secondterm}
For any fixed $\varepsilon>0$ and any $f\in C^\infty_0(\R)$,

\begin{equation*}
\lim_{M\to\infty}\lim_{N\to\infty}\E|U_{N,M}^{(\varepsilon)}(f)-\eta^{(\varepsilon)}(f)|^2=0.
\end{equation*}
\end{lemma}

\begin{proof}
Let us expand the square and calculate the expectation. The $U$-term is

\begin{align*}
\E |U_{N,M}^{(\varepsilon)}(f)|^2=\int_\R\int_\R f(x)\overline{f(y)} e^{\sum_{j=1}^{M-1}\frac{\left(e^{-i\frac{jx}{M}}-1\right)\left(e^{i\frac{jy}{M}}-1\right)e^{-2\varepsilon \frac{j}{M}}}{j}+\sum_{j=M}^N\frac{e^{-2j\varepsilon/M}e^{-ij(x-y)/M}}{j}}dxdy
\end{align*}
First of all, by dominated convergence, we can take the $N\to\infty$ limit inside of the integrals. Since $\varepsilon >0$, also the latter sum in the exponential is uniformly bounded, whence  by dominated convergence in the limit where first  $N\to\infty$ and then $M\to\infty$
\begin{align*}
\E |U_{N,M}^{(\varepsilon)}(f)|^2&\to \int_\R\int_\R f(x)\overline{f(y)}e^{\int_0^1 \frac{(e^{-iux}-1)(e^{iuy}-1)e^{-2\varepsilon u}}{u}du+\int_1^\infty e^{-2\varepsilon u} \frac{e^{-iu(x-y)}}{u}du}dxdy\\
&=\E|\eta^{(\varepsilon)}(f)|^2,
\end{align*}

\noindent which is finite. Thus our task is to analyze the cross terms. From \eqref{eq:zetacoupl} -- the definition of $Z_j$ -- one finds

\begin{align*}
\E U_{N,M}(f)\overline{\eta^{(\varepsilon)}(f)}&=\int_\R\int_\R f(x)\overline{f(y)} e^{\sum_{j=1}^{M-1}\frac{e^{-ijx/M}-1}{\sqrt{j/M}}e^{-\varepsilon j/M}\int_{\frac{j}{M}}^{\frac{j+1}{M}}\frac{e^{iuy}-1}{\sqrt{u}}e^{-\varepsilon u}du}\\
&\quad \times e^{\sum_{j=M}^{N}e^{-\varepsilon j/M}\frac{e^{-ijx/M}}{\sqrt{j/M}}\int_{\frac{j}{M}}^{\frac{j+1}{M}}\frac{e^{-\varepsilon u}e^{iuy}}{\sqrt{u}}du}dxdy
\end{align*}

Again by dominated convergence one can justify taking the limits inside of the integral and the question becomes studying the limit of the exponential in the last written integrand as $M\to\infty .$

%\begin{equation*}
%\lim_{M\to\infty}e^{\sum_{j=1}^{M-1}\frac{e^{-ijx/M}-1}{\sqrt{j}}e^{-\varepsilon j/M}\sqrt{M}\int_{\frac{j}{M}}^{\frac{j+1}{M}}\frac{e^{iuy}-1}{\sqrt{u}}e^{-\varepsilon u}du+\sum_{j=M}^{\infty}e^{-\varepsilon j/M}\frac{e^{-ijx/M}}{\sqrt{j}}\sqrt{M}\int_{\frac{j}{M}}^{\frac{j+1}{M}}\frac{e^{-\varepsilon u}e^{iuy}}{\sqrt{u}}du}
%\end{equation*}

Noting that uniformly for $y$ in the support of $f$ we have
\begin{equation*}
\int_{\frac{j}{M}}^{\frac{j+1}{M}}\frac{e^{iuy}-1}{\sqrt{u}}e^{-\varepsilon u}du=M^{-1}\frac{e^{ijy/M}-1}{\sqrt{j/M}}e^{-\varepsilon j/M}+\mathcal{O}(M^{-2}(j/M)^{-1/2}e^{-\varepsilon j/M})
\end{equation*}
 and 
\begin{equation*}
\int_{\frac{j}{M}}^{\frac{j+1}{M}}\frac{e^{-\varepsilon u}e^{iuy}}{\sqrt{u}}du=M^{-1}\frac{e^{-\varepsilon j/M} e^{ijy/M}}{\sqrt{j/M}}+\mathcal{O}(M^{-2}(j/M)^{-1/2} e^{-\varepsilon j/M}),
\end{equation*}
 one finds that the exponential in the question converges to
\begin{align*}
%\lim_{M\to\infty}&e^{\sum_{j=1}^{M-1}\frac{e^{-ijx/M}-1}{\sqrt{j}}e^{-\varepsilon j/M}\sqrt{M}\int_{\frac{j}{M}}^{\frac{j+1}{M}}\frac{e^{iuy}-1}{\sqrt{u}}e^{-\varepsilon u}du+\sum_{j=M}^{\infty}e^{-\varepsilon j/M}\frac{e^{-ijx/M}}{\sqrt{j}}\sqrt{M}\int_{\frac{j}{M}}^{\frac{j+1}{M}}\frac{e^{-\varepsilon u}e^{iuy}}{\sqrt{u}}du}\\&
=\exp\left(\int_0^1 \frac{(e^{-ixu}-1)(e^{iyu}-1)}{u}e^{-2\varepsilon u}du+\int_1^\infty \frac{e^{-2\varepsilon u}e^{-iu(x-y)}}{u}du
\right).
\end{align*}
The other cross term is handled in an identical way, and  the claim follows.
\end{proof}

Next we consider the $U_{N,M}-U_{N,M}^{(\varepsilon)}$-term.

\begin{lemma}\label{le:firstterm}
Let $f\in C^\infty_0(\R)$. Then 

\begin{equation*}
\lim_{\varepsilon\to 0^+}\lim_{N\to\infty}\E \left|U_{N,M}^{(\varepsilon)}(f)-U_{N,M}(f)\right|^2=0,
\end{equation*}
 uniformly in $M$.

\end{lemma}

\begin{proof}
Let us write $f^{M,\varepsilon}(x)=f(x)e^{-M^{-1}\sum_{j=1}^{M-1}\frac{e^{-ijx/M}-1}{j/M}e^{-\varepsilon j/M}}$ and $f^M =f^{M,0}$.  Then $|f_M(x)|\leq C|f(x)|$, where $C$ is independent of $f$ and $M$. An application of Lemma \ref{le:joulu1} and Lemma \ref{le:joulu2} together with a straightforward computation yield that in the limit $N\to\infty$
\begin{align}\label{eq:joulu100}
\E \left|U_{N,M}(f)\right|^2&=\int_\R\int_\R f^M(x)\overline{f^M(y)} e^{-\sum_{j=1}^{M-1}\frac{1}{j}} e^{\sum_{j=1}^N \frac{e^{-ij(x-y)/M}}{j}}dxdy\nonumber\\
&=e^{-\sum_{j=1}^{M-1}\frac{1}{j}}\sum_{k=0}^\infty \frac{1}{k!}\sum_{j_1,...,j_k=1}^N\frac{1}{j_1\cdots j_k}\left|\widehat{f^M}\left(\frac{j_1+\cdots +j_k}{2\pi M}\right)\right|^2\nonumber\\
&\nearrow e^{-\sum_{j=1}^{M-1} \frac{1}{j}}\sum_{l=0}^\infty \left|\widehat{f^M}\left(\frac{l}{2\pi M}\right)\right|^2\sum_{k=0}^\infty \frac{1}{k!}\sum_{j_1,...,j_k=1}^\infty\frac{1}{j_1\cdots j_k}\delta_{j_1+\cdots +j_k,l}\nonumber\\
&=e^{-\sum_{j=1}^{M-1} \frac{1}{j}}\sum_{l=0}^\infty \left|\widehat{f^M}\left(\frac{l}{2\pi M}\right)\right|^2\;\;\lesssim
\int_\R |f(x)|^2(1+x^2)dx,
\end{align}
since $\exp(-\sum_{j=1}^{M-1} \frac{1}{j})= M^{-1}e^{-\gamma}(1+o(1)),$ as $M\to\infty$.
 Similar arguments lead to 

\begin{equation*}
\lim_{N\to\infty}\E\left|U_{N,M}^{(\varepsilon)}(f)\right|^2=e^{-\sum_{j=1}^{M-1} \frac{e^{-2\varepsilon j/M}}{j}}\sum_{l=0}^\infty\left|\widehat{f^{M,2\varepsilon}}\left(\frac{l}{2\pi M}\right)\right|^2 e^{-2\varepsilon l/M}
\end{equation*}

\noindent and 

\begin{equation*}
\lim_{N\to\infty}\E \left[U_{N,M}(f)\overline{U_{N,M}^{(\varepsilon)}(f)}\right]=e^{-\sum_{j=1}^{M-1} \frac{e^{-\varepsilon j/M}}{j}}\sum_{l=0}^\infty\widehat{f^{M}}\left(\frac{l}{2\pi M}\right)\overline{\widehat{f^{M,\varepsilon}}\left(\frac{l}{2\pi M}\right)} e^{-\varepsilon l/M}.
\end{equation*}

Formally we may then let $\varepsilon\to 0^+$ in the above computations and obtain the desired claim. To be able to do this we just need to recall  Lemma \eqref{le:joulu2}(ii), and observe the quantitative estimate
\begin{eqnarray}\label{eq:uusivuosi1}
\| M^{-1/2}\widehat{f^{M,\varepsilon}}\left(\frac{\cdot}{2\pi M}\right)- M^{-1/2}\widehat{f^{M}}\left(\frac{\cdot}{2\pi M}\right)e^{-\varepsilon l/M}\|_{\ell^2(\Z)}
\to 0\quad
\end{eqnarray}
as $\varepsilon\to 0^+,$ which is uniform in $M\geq 1.$ This in turn follows easily by combining  the estimate
\begin{eqnarray*}
&&\| M^{-1/2}\widehat{f^{M,\varepsilon}}\left(\frac{\cdot}{2\pi M}\right)- M^{-1/2}\widehat{f^{M}}\left(\frac{\cdot}{2\pi M}\right)\|_{\ell^2(\Z)}\\
&\lesssim& \left(\int_\R |f^{M}(x)-f^{M,\varepsilon}(x)|^2(1+x^2)\right)^{1/2}
\;\;\longrightarrow 0\quad
\end{eqnarray*}
from Lemma \eqref{le:joulu2}, with uniformity  in $M$ as $\varepsilon\to 0^+$, with the observation that $|\widehat{f^{M,\varepsilon}}(\xi)|\lesssim ((1+|\xi|)^2)^{-2}$ uniformly in $M$ and $\varepsilon$.  For the last mentioned fact one notes in addition that the supports of the functions 
$f^{M,\varepsilon}$ are contained in the same compact interval and applies notes the uniform bound $\| (d/dx)^k f^{M,\varepsilon}\|_\infty \leq C_k$ for $k\geq 0$.
\end{proof}

In order to finish the proof of  Lemma  \ref{le:mesogaussianlemma} we need to establish the uniform estimate \eqref{eq:joulu10}. However, this we already obtained in \eqref{eq:joulu100}, and the proof is complete.
\end{proof}

\end{document}